\documentclass[11pt]{article}
\usepackage[tbtags]{amsmath}
\usepackage{amssymb}
\usepackage{amsthm}
\usepackage[misc]{ifsym}
\usepackage{graphicx}
\usepackage{tikz}
\usetikzlibrary{shapes,arrows}
\usepackage{cases}
\numberwithin{equation}{section}
\newtheorem{definition}{Definition}[section]
\newtheorem{theorem}{Theorem}[section]
\newtheorem{lemma}{Lemma}[section]

\normalsize
\usepackage{hyperref,amsmath,amssymb,amscd}
\title{\bf A Stackelberg Game of Backward Stochastic Differential Equations with Partial Information
\thanks{This work is supported by National Key R\&D Program of China (Grant No. 2018YFB1305400) and National Natural Science Foundations of China (Grant No. 11971266, 11831010, 11571205).}}
\author{\normalsize Yueyang Zheng\thanks{\it School of Mathematics, Shandong University, Jinan 250100, P.R.China, E-mail: zhengyueyang0106@163.com} , Jingtao Shi\thanks{\it Corresponding author. School of Mathematics, Shandong University, Jinan 250100, P.R.China, E-mail: shijingtao@sdu.edu.cn}}
\newtheorem{Remark}{Remark}[section]
\begin{document}
\maketitle

\noindent{\bf Abstract:} This paper is concerned with a Stackelberg game of backward stochastic differential equations (BSDEs) with partial information, where the information of the follower is a sub-$\sigma$-algebra of that of the leader. Necessary and sufficient conditions of the optimality for the follower and the leader are first given for the general problem, by the partial information stochastic maximum principles of BSDEs and forward-backward stochastic differential equations (FBSDEs), respectively. Then a linear-quadratic (LQ) Stackelberg game of BSDEs with partial information is investigated. The state estimate feedback representation for the optimal control of the follower is first given via two Riccati equations. Then the leader's problem is formulated as an optimal control problem of FBSDE. Four high-dimensional Riccati equations are introduced to represent the state estimate feedback for the optimal control of the leader. Theoretic results are applied to a pension fund management problem of two players in the financial market.

\vspace{1mm}

\vspace{1mm}


\section{Introduction}

The Stackelberg game is also known as the leader-follower game, which can be traced back to the early work by Stackelberg \cite{S52}, when he defined a concept of a hierarchical solution for markets where some firms have power of domination over others. The solutions, in the context of the differential game, is called the corresponding Stackelberg equilibrium points in which there are two players with asymmetric roles, one leader and one follower. For obtaining the Stackelberg solutions, it is usual to divide the game problem into two parts. In the first part, which is also known as the follower's problem, firstly the leader announces his strategy, then the follower will make an instantaneous response, and choose an optimal strategy corresponding to the given leader's strategy to minimize (or maximize) his cost functional. In the second part, knowing the follower would take such an optimal strategy, the leader will choose an optimal strategy to minimize (or maximize) his cost functional. Overall, the decisions must be made by two player and one of them is subordinated to the other because of the asymmetric roles, therefore one player must making a decision after the other player's decision is made.

The Stackelberg game has wide practical financial and economical backgrounds, and has attracted more and more research attentions with applications. Simann and Cruz \cite{SC73} made an early study on the properties of the Stackelberg solution in static and dynamic non-zero sum two-player games. Bagchi and Ba\c{s}ar \cite{BB81} investigated an LQ stochastic Stackelberg differential game, where the state and control variables do not enter the diffusion coefficient in the state equation. Yong \cite{Yong02} studied an LQ leader-follower differential game in a more general framework where the coefficients of the system and the cost functionals are random, the diffusion of the state equation contains the control variables, and the weight matrices for the controls in cost functionals are not necessarily positive definite. \O ksendal et al. \cite{OSU13} proved a maximum principle for a Stackelberg differential game with jump-diffusion, and applied the result to a continuous time newsvendor problem. Bensoussan et al. \cite{BCS15} introduced several solution concepts in terms of the players' information sets, and studied LQ Stackelberg games under both adapted open-loop and closed-loop memoryless information structures, whereas the control variables do not enter the diffusion coefficient of the state equation. Meanwhile, the Stackelberg games have been investigated in the mean-field, time-delay, partial information and other fields. Recently, Xu and Zhang \cite{XZ16} studied the discrete-time leader-follower game with time delay and the new co-states which capture the future information of the control and the new state which contains the past effects are introduced to overcome the noncausality of strategy design caused by the delay, then the same technique is used to deal with the continuous-time system. Then Xu et al. \cite{XSZ18} studied the leader-follower differential game with time delay appearing in the leader's control, the open-loop solution is given in the form of the conditional expectation with respect to several symmetric Riccati equations by mainly establishing the nonhomogeneous relationship between the forward and the backward variables. Moon and Ba\c{s}ar \cite{MB18} investigated an LQ mean field Stackelberg differential game with the adapted open-loop information structure of the leader where there are only one leader but arbitrarily large number of followers. Lin et al. \cite{LJZ19} studied the open-loop LQ Stackelberg game of the mean-field stochastic systems in finite horizon, and a sufficient condition for the existence and uniqueness of the stackelberg strategy was given in terms of the solvability of some Riccati equations and a convexity condition by introducing new state and costate variables. Shi et al. \cite{SWX16} introduced a new explanation for the asymmetric information feature that the information available to the follower is based on the some sub-$\sigma$-algebra of that available to the leader for the Stackelberg differential game. Then an LQ stochastic Stackelberg differential game with noisy observation was solved via some measure transformation, filtering technique, linear FBSDE and mean-field FBSDE decoupling technique, where not all the diffusion coefficients contain the control variables. Shi et al. \cite{SWX17} studied an LQ stochastic Stackelberg differential game with asymmetric information, where the control variables enter both diffusion coefficients of the state equation. Shi et al. \cite{SWX17o} investigated a kind of stochastic LQ Stackelberg differential game with overlapping information which means that the follower's and the leader's information have some joint part, while they have no inclusion relation. Li and Yu \cite{LY18} proved the uniqueness and solvability of a kind of coupled forward-backward stochastic differential equations (FBSDEs) with a multilevel self-similar domination-monotonicity structure, then this kind of FBSDEs is used to characterize the unique equilibrium of an LQ generalized Stackelberg game with hierarchy in a closed form.

Different from forward stochastic differential equations (SDEs) where a prescribed initial condition $x(0)=x_0$ is given, the BSDEs is short for a kind of backward SDEs with a given terminal condition $y(T)=\xi$. And BSDE admits a pair of adapted solution $(y(\cdot),z(\cdot))$ under some conditions, where the additional term $z(\cdot)$ may be interpreted as a risk-adjustment factor and is required for the equation to have adapted solution. The linear version of this type of equation was first introduced by Bismut \cite{Bis78} as the adjoint equation in the stochastic maximum principle. General nonlinear BSDEs, introduced independently by Pardoux and Peng \cite{PP90} and Duffie and Epstein \cite{DE92}, have received considerable research attention in recent years due to their nice structure and wide applicability in a number of different areas, especially in mathematical finance, optimal control and differential games. El Karoui et al. \cite{EPQ97} discussed different properties of BSDEs and their application to finance. Two recent monographs about BSDEs can be seen in Pardoux and R\u{a}\c{s}canu \cite{PR14} and Zhang \cite{Zhang17}.

The optimal control problem of BSDEs was first studied by Peng \cite{P92,P93} and El Karoui et al. \cite{EPQ97}, when solving the recursive utility maximization problems. Dokuchaev and Zhou \cite{DZ99} studied a stochastic control problem where the system dynamics is a controlled nonlinear BSDE. Kohlmann and Zhou \cite{KZ00} explored the relationship between BSDEs and stochastic controls by interpreting BSDEs as some stochastic optimal control problems. Chen and Zhou \cite{CZ00} investigated an optimization model of stochastic LQ regulators with indefinite control cost weighting matrices, involving a backward LQ problem. Lim and Zhou \cite{LZ01} studied an optimal control of linear BSDEs with a quadratic cost criteria, and the solution is obtained by using the completion-of-squares technique. Huang et al. \cite{HWX09} studied a partial information control problem of backward stochastic systems, and obtained a new stochastic maximum principle. Shi \cite{Shi11} investigated an optimal control problem for systems described by BSDEs with time delayed generators, and proved a sufficient maximum principle. The mean-field BSDE was firstly introduced by Buckdahn et al. \cite{BDLP09}. Ma and Liu \cite{ML17} investigated an optimal control of an infinite horizon system governed by mean-field BSDE with delay and partial information, and establish the existence and uniqueness results for a mean-field BSDE with average delay. Li et al. \cite{LSX19} studied the LQ optimal control problem for mean-filed BSDEs.

When it comes to the differential game problem of BSDEs, Hamadene and Lepeltier \cite{HL95} discussed a stochastic zero-sum differential games of the results on BSDEs, and obtained the existence of a saddle point in the bounded case under the Isaacs' condition. Yu and Ji \cite{YJ08} studied an existence and uniqueness result for an initial coupled FBSDE under some monotone conditions, which was applied to backward LQ non-zero sum stochastic differential game problem. Wang and Yu \cite{WY10} established a necessary condition and a sufficient condition in the form of maximum principle for open-loop equilibrium point of the game systems described by the BSDEs. Wang and Yu \cite{WY12} continued to establish a necessary condition in the form of maximum principle for open-loop Nash equilibrium point of this type of partial information game, and then gave a verification theorem which is a sufficient condition for Nash equilibrium point. Shi and Wang \cite{SW16} investigated a non-zero sum differential game, where the state dynamics follows a BSDE with time-delayed generator, and an Arrow's sufficient condition for open-loop equilibrium point is proved. Huang et al. \cite{HWW16} studied a backward mean-field linear-quadratic-Gaussian games of weakly coupled stochastic large-population system, and two classes of foregoing games are discussed and their decentralized strategies are derived through the consistency condition. Huang and Wang \cite{HW17} discussed a kind of non-zero sum differential game of mean-field BSDE. Wang et al. \cite{WXX18} studied a kind of LQ non-zero sum differential game driven by BSDE with asymmetric information. Aurell \cite{A18} studied a mean-field type games between two players with backward stochastic dynamics, and made up a class of non-zero sum, non-cooperating, differential games where the players' state dynamics solve a BSDE that depends on the marginal distributions of player states. Du et al. \cite{DHW19} studied the mean-field game of $N$ weakly-coupled linear BSDE system. Du and Wu \cite{DW19} investigated a new kind of Stackelberg differential game of mean-field BSDEs. Huang et al. \cite{HWZ19} focused on a kind of non-zero sum differential game driven by mean-field BSDE with asymmetric information.

Inspired by the above literatures, in this paper we study the Stackelberg game of BSDEs with partial information, where the coefficients of the backward game system and cost functionals are deterministic, and the control domain is convex. In our framework, we set that the information filtration available to the leader is the complete information filtration naturally generated by the random noise source, and the information filtration available to the follower is based on the sub-$\sigma$-algebra of that available to the leader. The novelty of the formulation and the contribution in this paper is the following. (1) A new kind of general Stackelberg game of BSDEs with partial information is introduced and studied by the maximum principle approach, where a terminal condition $\xi$ is given in advance. For the follower's problem, the partial information maximum principle and verification theorem are given, which are direct from Theorem 2.1 and Theorem 2.3 of Wang and Yu \cite{WY12}. For the leader's problem, the partial information maximum principle could be derived via the similar technique in Zuo and Min \cite{ZM13} which, however, did not give the corresponding sufficient condition. Therefore, in our paper, the partial information verification theorem is derived, by the Clarke generalized gradient. (2) For the LQ case, it consists of a stochastic optimal control problem of BSDE with partial information for the follower, and followed by a stochastic optimal control problem of coupled conditional mean-field forward-backward stochastic differential equations (FBSDEs) with complete information for the leader, which is different from that in the (forward) Stackelberg differential game studied in Shi et al. \cite{SWX17}. (3) For giving the state estimate feedback representations for the optimal control of the follower, two Riccati equations, a linear backward stochastic differential filtering equation (BSDFE), and a linear stochastic differential filtering equation (SDFE) are introduced. See Theorem 4.1. Then, four high-dimensional Riccati equations, a linear BSDFE, and a linear SDFE are introduced to represent the optimal control of the leader as the state estimate feedback form. See Theorem 4.2. (4) A pension fund problem of two players with asymmetric information in the financial market is studied, the Stackelberg equilibrium point is represented and the optimal initial wealth reserve is obtained explicitly.

The rest of this paper is organized as follows. In Section 2, the general Stackelberg game of BSDEs with partial information is formulated. Then this general problem is studied in Section 3. The follower's problem of the BSDE with partial information is considered first in Subsection 3.1, while the leader's problem of the conditional mean-field FBSDE is studied in Subsection 3.2. By the partial information maximum principle approach, necessary and sufficient conditions for the optimal controls of the follower and the leader's are given, respectively. Then the LQ Stackelberg game problem with partial information is investigated in Section 4. Specially, Subsection 4.1 is devoted to the solution of an LQ stochastic optimal control problem of BSDE with partial information of the follower, via two Riccati equations, a BSDFE and a SDFE, the optimal control of the follower is given in the state feedback form. Subsection 4.2 is devoted to the solution of an LQ stochastic optimal control problem of coupled conditional mean-field FBSDE with complete information of the leader, the optimal control of the leader is represented as the state feedback form by the solutions to four new high-dimensional Riccati equations, a BSDFE and a SDFE. In Section 5, the theoretic results in the previous sections are applied to a pension fund management problem of two players with asymmetric information in the financial market. Finally, Section 6 gives some concluding remarks.

\section{Problem Formulation}

In this paper, we use $\mathbb{R}^n$ to denote the Euclidean space of $n$-dimensional vectors, $\mathbb{R}^{n\times d}$ to denote the space of $n\times d$ matrices, and $\mathcal{S}^n$ to denote the space of $n\times n$ symmetric matrices. $\langle\cdot,\cdot\rangle$ and $|\cdot|$ are used to denote the scalar product and norm in the Euclidean space, respectively. A $\top$ appearing in the superscript of a matrix, denotes its transpose. $f_x,f_{xx}$ denote the first- and second-order partial derivatives with respect to $x$ for a differentiable function $f$, respectively.

Let $T>0$ be fixed. Consider a complete probability space $(\Omega,\mathcal{F},\mathbb{P})$ and two standard $m(\widetilde{m})$-dimensional Brownian motions $W(t)$ and $\widetilde{W}(t)$ with $W(0)=\widetilde{W}(0)=0$, which generates the filtration $\mathcal{F}_{t}=\sigma\{W(r),\widetilde{W}(r): 0\leq r\leq t\}$ augmented by all the $\mathbb{P}$-null sets in $\mathcal{F}$. $L_{\mathcal{F}_T}^2(\Omega,\mathbb{R}^n)$ denotes the set of $\mathbb{R}^n$-valued, $\mathcal{F}_T$-measurable, square-integrable random vectors, $L^2_\mathcal{F}(0,T;\mathbb{R}^n)$ denotes the set of $\mathbb{R}^n$-valued, $\mathcal{F}_t$-adapted, square integrable processes on $[0,T]$, $L^2_\mathcal{F}(0,T;\mathbb{R}^{n\times d})$ denotes the set of $n\times d$-matrix-valued, $\mathcal{F}_t$-adapted, square integrable processes on $[0,T]$, and $L^\infty(0,T;\mathbb{R}^{n\times d})$ denotes the set of $n\times d$-matrix-valued, bounded functions on $[0,T]$.

Let us consider the following controlled BSDE:
\begin{equation}\label{bsde1}
\left\{
\begin{aligned}
-dy^{v_1,v_2}(t)&=f(t,y^{v_1,v_2}(t),z^{v_1,v_2}(t),\tilde{z}^{v_1,v_2}(t),v_1(t),v_2(t))dt\\
                &\quad-z(t)^{v_1,v_2}dW(t)-\tilde{z}^{v_1,v_2}(t)d\widetilde{W}(t),\ t\in[0,T],\\
  y^{v_1,v_2}(T)&=\xi,
\end{aligned}
\right.
\end{equation}
where $f:[0,T]\times\mathbb{R}^n\times\mathbb{R}^{n\times m}\times\mathbb{R}^{n\times\widetilde{m}}\times\mathbb{R}^{k_1}\times\mathbb{R}^{k_2}\rightarrow\mathbb{R}^n$ is a given continuous function in $(t,y,z,\tilde{z},v_1,v_2)$ and $\xi\in L_{\mathcal{F}_T}^2(\Omega,\mathbb{R}^n)$ is given. $v_1(\cdot)\in U_1$ is the control process of the follower, and $v_2(\cdot)\in U_2$ is the control process of the leader, where $U_i$ is a nonempty convex subset of $\mathbb{R}^{k_i},\ i=1,2$. In the backward game system \eqref{bsde1}, the two players work together to achieve a common goal $\xi$ at the terminal time $T$.

Let $\mathcal{G}_t^i\subseteq\mathcal{F}_t$ be a given sub-filtration which represents the information available to the follower and the leader at time $t\in[0,T],\ i=1,2$, respectively, and $\mathcal{G}_t^1\subseteq\mathcal{G}_t^2\subseteq\mathcal{F}_t$. We define the admissible control sets by
\begin{equation}
\mathcal{U}_i[0,T]=\big\{v_i(\cdot)\in L^2_{\mathcal{G}^i}(0,T;\mathbb{R}^{k_i})\big|v_i(t)\in U_i,t\in[0,T],\ a.e,\ a.s.\big\},\ i=1,2,
\end{equation}
respectively.

We define the cost functionals of the follower and the leader as
\begin{equation}\label{cost functional}
J_i(v_1(\cdot),v_2(\cdot);\xi)=\mathbb{E}\bigg[\int_0^TL_i(t,y^{v_1,v_2}(t),z^{v_1,v_2}(t),\tilde{z}^{v_1,v_2}(t),v_1(t),v_2(t))dt+h_i(y^{v_1,v_2}(0))\bigg],
\end{equation}
for $i=1,2$. Here $L_i:[0,T]\times\mathbb{R}^n\times\mathbb{R}^{n\times m}\times\mathbb{R}^{n\times\widetilde{m}}\times\mathbb{R}^{k_1}\times\mathbb{R}^{k_2}\rightarrow\mathbb{R}$ are given continuous functions in $(t,y,z,\tilde{z},v_1,v_2)$ and $h_i:\mathbb{R}^n\rightarrow\mathbb{R}$ are given continuous functions, for $i=1,2$. We remark that the cost functional \eqref{cost functional} describe that the players have their own benefits except for the terminal common goal $\xi$.

Now, we give some assumptions that will be in force through this paper.

{\bf(A1)} {\it The function $f$ is continuously differentiable in $(y,z,\tilde{z},v_1,v_2)$. Moreover, the partial derivatives $f_y,f_z,f_{\tilde{z}},f_{v_1}$ and $f_{v_2}$ with respect to $y,z,\tilde{z},v_1$ and $v_2$ are uniformly bounded.}

From Pardoux and Peng \cite{PP90}, it is easy to see that if both $v_1(\cdot)\in\mathcal{U}_1[0,T]$, $v_2(\cdot)\in\mathcal{U}_2[0,T]$, and {\bf(A1)} holds, then BSDE \eqref{bsde1} admits a unique solution triple $(y^{v_1,v_2}(\cdot),z^{v_1,v_2}(\cdot),\tilde{z}^{v_1,v_2}(\cdot))\in L_\mathcal{F}^2(0,T;\mathbb{R}^n)\times L_\mathcal{F}^2(0,T;\mathbb{R}^{n\times m})\times L_\mathcal{F}^2(0,T;\mathbb{R}^{n\times\widetilde{m}})$, which we called the state trajectory.

{\bf(A2)} {\it $L_i$ is continuously differentiable with respect to $(y,z,\tilde{z},v_1,v_2)$ and $h_i$ is continuously differentiable in $y,\ i=1,2$. Moreover, there exists a constant $C$ such that $L_{iy},L_{iz},L_{i\tilde{z}},L_{iv_1},L_{iv_2}$ are bounded by $C(1+|y|+|z|+|\tilde{z}|+|v_1|+|v_2|)$, and $h_{iy}$ is bounded by $C(1+|y|),\ i=1,2$.}

\vspace{1mm}

The problem studied in this paper is proposed in the following definition.
\begin{definition}
The pair $(\bar{v}_1(\cdot),\bar{v}_2(\cdot))\in\mathcal{U}_1[0,T]\times\mathcal{U}_2[0,T]$ is called an optimal solution to the Stackelberg game of BSDEs with partial information, if it satisfies the following condition:\\
(i)\ For given $\xi\in L_{\mathcal{F}_T}^2(\Omega,\mathbb{R}^n)$ and any $v_2(\cdot)\in\mathcal{U}_2[0,T]$, there exists a map $\Gamma:\mathcal{U}_2[0,T]\times L_{\mathcal{F}_T}^2(\Omega,\mathbb{R}^n)\rightarrow\mathcal{U}_1[0,T]$ such that
\begin{equation}
J_1(\Gamma(v_2(\cdot),\xi),v_2(\cdot);\xi)=\min_{v_1(\cdot)\in\mathcal{U}_1[0,T]}J_1(v_1(\cdot),v_2(\cdot);\xi).
\end{equation}
(ii)\ There exists a unique $\bar{v}_2(\cdot)\in\mathcal{U}_2[0,T]$ such that
\begin{equation}
J_2(\Gamma(\bar{v}_2(\cdot),\xi),\bar{v}_2(\cdot);\xi)=\min_{v_2(\cdot)\in\mathcal{U}_2[0,T]}J_2(\Gamma(\bar{v}_2(\cdot),\xi),v_2(\cdot);\xi).
\end{equation}
(iii)\ The optimal strategy of the follower is $\bar{v}_1(\cdot)=\Gamma(\bar{v}_2(\cdot),\xi)$.
\end{definition}

We call the above problem a {\it Stackelberg game of BSDE with partial information}.

\section{The General Problem}

\subsection{Optimization for The Follower}

In this subsection, we seek the necessary and sufficient conditions of the partial information optimal control for the follower.

Let $\xi\in L_{\mathcal{F}_T}^2(\Omega,\mathbb{R}^n)$ be given, giving the leader's strategy $v_2(\cdot)\in\mathcal{U}_2[0,T]$. Let $\bar{v}_1(\cdot)$ be an optimal control of the follower, and $(y^{\bar{v}_1,v_2}(\cdot),z^{\bar{v}_1,v_2}(\cdot),\tilde{z}^{\bar{v}_1,v_2}(\cdot))$ be the corresponding state trajectory. Let the process $x(\cdot)\in L_\mathcal{F}^2(0,T;\mathbb{R}^n)$ satisfy the following adjoint equation:
\begin{equation}\label{adjoint eq}
\left\{
\begin{aligned}
dx(t)=&-H_{1y}(t,y^{\bar{v}_1,v_2}(t),z^{\bar{v}_1,v_2}(t),\tilde{z}^{\bar{v}_1,v_2}(t),\bar{v}_1(t),v_2(t),x(t))dt\\
      &-H_{1z}(t,y^{\bar{v}_1,v_2}(t),z^{\bar{v}_1,v_2}(t),\tilde{z}^{\bar{v}_1,v_2}(t),\bar{v}_1(t),v_2(t),x(t))dW(t)\\
      &-H_{1\tilde{z}}(t,y^{\bar{v}_1,v_2}(t),z^{\bar{v}_1,v_2}(t),\tilde{z}^{\bar{v}_1,v_2}(t),\bar{v}_1(t),v_2(t),x(t))d\widetilde{W}(t),\ t\in[0,T],\\
 x(0)=&\ h_{1y}(y^{\bar{v}_1,v_2}(0)),
\end{aligned}
\right.
\end{equation}
where $H_{1y},H_{1z}$ and $H_{1\tilde{z}}$ denote the partial derivatives of $H_1$ with respect to $y,z$ and $\tilde{z}$, respectively, and the Hamiltonian function $H_1:[0,T]\times\mathbb{R}^n\times\mathbb{R}^{n\times m}\times\mathbb{R}^{n\times\widetilde{m}}\times\mathbb{R}^{k_1}\times\mathbb{R}^{k_2}\times\mathbb{R}^n\rightarrow\mathbb{R}$ is defined as
\begin{equation}\label{H1}
H_1(t,y,z,\tilde{z},v_1,v_2,x)=-L_1(t,y,z,\tilde{z},v_1,v_2)-\langle f(t,y,z,\tilde{z},v_1,v_2),x\rangle.
\end{equation}

The following two results are direct from Theorems 2.1 and 2.3 of Wang and Yu \cite{WY12}.

\begin{theorem}(Partial information maximum principle for the follower)
Let {\bf(A1)}, {\bf(A2)} hold and $\xi\in L_{\mathcal{F}_T}^2(\Omega,\mathbb{R}^n)$. Giving the leader's strategy $v_2(\cdot)\in\mathcal{U}_2[0,T]$. Let $\bar{v}_1(\cdot)\in\mathcal{U}_1[0,T]$ be an optimal control of the follower and $(y^{\bar{v}_1,v_2}(\cdot),z^{\bar{v}_1,v_2}(\cdot),\tilde{z}^{\bar{v}_1,v_2}(\cdot))$ be the corresponding state trajectory. Then we have
\begin{equation}\label{maximum condition1}
\mathbb{E}\big[\big\langle H_{1v_1}(t,y^{\bar{v}_1,v_2}(t),z^{\bar{v}_1,v_2}(t),\tilde{z}^{\bar{v}_1,v_2}(t),\bar{v}_1(t),v_2(t),x(t)),v_1-\bar{v}_1(t)\big\rangle\big|\mathcal{G}_{t}^1\big]\leq0,
\end{equation}
for $a.e.\ t\in[0,T]$,\ a.s., for any $v_1\in U_1$, where $x(\cdot)$ is the solution to the adjoint equation \eqref{adjoint eq}.
\end{theorem}

\begin{theorem}(Partial information verification theorem for the follower)
Let {\bf(A1)}, {\bf(A2)} hold and $\xi\in L_{\mathcal{F}_T}^2(\Omega,\mathbb{R}^n)$. Giving the leader's strategy $v_2(\cdot)\in\mathcal{U}_2[0,T]$. Assume that $L_1$ is continuously differentiable in $v_1$, let $\bar{v}_1\in\mathcal{U}_1[0,T]$ be given such that
\begin{equation*}
\begin{aligned}
&L_{1y}(t,y^{\bar{v}_1,v_2}(t),z^{\bar{v}_1,v_2}(t),\tilde{z}^{\bar{v}_1,v_2}(t),\bar{v}_1(t),v_2(t)),\\
&L_{1z}(t,y^{\bar{v}_1,v_2}(t),z^{\bar{v}_1,v_2}(t),\tilde{z}^{\bar{v}_1,v_2}(t),\bar{v}_1(t),v_2(t)),\\
&L_{1\tilde{z}}(t,y^{\bar{v}_1,v_2}(t),z^{\bar{v}_1,v_2}(t),\tilde{z}^{\bar{v}_1,v_2}(t),\bar{v}_1(t),v_2(t)),\\
&L_{1v_1}(t,y^{\bar{v}_1,v_2}(t),z^{\bar{v}_1,v_2}(t),\tilde{z}^{\bar{v}_1,v_2}(t),\bar{v}_1(t),v_2(t))\in L_{\mathcal{F}}^2(0,T).
\end{aligned}
\end{equation*}
For any $(t,v_1)\in[0,T]\times U_1$, $L_1$ satisfies $\mathbb{E}\int_0^T|L_{1v_1}(t,y^{\bar{v}_1,v_2}(t),z^{\bar{v}_1,v_2}(t),\tilde{z}^{\bar{v}_1,v_2}(t),v_1,v_2(t))|dt<+\infty$. Suppose that the adjoint equation \eqref{adjoint eq} admits a solution $x(\cdot)\in L_{\mathcal{F}}^2(0,T;\mathbb{R}^n)$, and
\begin{equation}\label{maximum condition2}
\begin{aligned}
\mathbb{E}&\big[H_1(t,y^{\bar{v}_1,v_2}(t),z^{\bar{v}_1,v_2}(t),\tilde{z}^{\bar{v}_1,v_2}(t),\bar{v}_1(t),v_2(t),x(t))\big|\mathcal{G}_t^1\big]\\
&=\max_{v_1\in U_1}\mathbb{E}\big[H_1(t,y^{v_1,v_2}(t),z^{v_1,v_2}(t),\tilde{z}^{v_1,v_2}(t),v_1,v_2(t),x(t))\big|\mathcal{G}_t^1\big],\ \mbox{for all }t\in[0,T].
\end{aligned}
\end{equation}
Moreover, suppose $\mathbb{E}[H_{1v_1}(t,y^{v_1,v_2}(t),z^{v_1,v_2}(t),\tilde{z}^{v_1,v_2}(t),v_1,v_2(t),x(t))|\mathcal{G}_t^1]$ is continuous at $v_1=\bar{v}_1(t)$ for any $t\in[0,T]$. Suppose for all $t\in[0,T]$, $H_1(t,\cdot,\cdot,\cdot,\cdot,v_2(t),x(t))$ is concave, and $h_1(\cdot)$ is convex. Then $\bar{v}_1(\cdot)$ is an optimal control of the follower.
\end{theorem}

\subsection{Optimization for the leader}

In this subsection, we firstly restate the partial information stochastic optimal control problem of the leader in detail.

For any $v_2(\cdot)\in\mathcal{U}_2[0,T]$, by the maximum condition \eqref{maximum condition1}, we assume that a functional $\bar{v}_1(t)=\bar{v}_1(t,\hat{y}^{\bar{v}_1,\hat{v}_2}(t),\hat{z}^{\bar{v}_1,\hat{v}_2}(t),\hat{\tilde{z}}^{\bar{v}_1,\hat{v}_2}(t),\hat{v}_2(t),\hat{x}(t))$ is uniquely defined, where
\begin{equation}\label{conditional expectation}
\left\{
\begin{aligned}
&\hat{y}^{\bar{v}_1,\hat{v}_2}(t)=\mathbb{E}[y^{\bar{v}_1,v_2}(t)|\mathcal{G}_t^1],\ \hat{z}^{\bar{v}_1,\hat{v}_2}(t)=\mathbb{E}[z^{\bar{v}_1,v_2}(t)|\mathcal{G}_t^1],\ \hat{\tilde{z}}^{\bar{v}_1,\hat{v}_2}(t)=\mathbb{E}[\tilde{z}^{\bar{v}_1,v_2}(t)|\mathcal{G}_t^1],\\
&\hat{v}_2(t)=\mathbb{E}[v_2(t)|\mathcal{G}_t^1],\ \hat{x}(t)=\mathbb{E}[x(t)|\mathcal{G}_t^1].
\end{aligned}
\right.
\end{equation}
Then, substituting $\bar{v}_1(t)$ into the game system \eqref{bsde1}, and combining it with the corresponding adjoint equation \eqref{adjoint eq}, we derive the following FBSDE
\begin{equation}\label{fbsde2}
\left\{
\begin{aligned}
-dy^{v_2}(t)&=f^L(t,y^{v_2}(t),z^{v_2}(t),\tilde{z}^{v_2}(t),v_2(t))dt-z^{v_2}(t)dW(t)-\tilde{z}^{v_2}(t)d\widetilde{W}(t),\\
       dx(t)&=\big[f_y^L(t,y^{v_2}(t),z^{v_2}(t),\tilde{z}^{v_2}(t),v_2(t))^\top x(t)\\
            &\qquad+L_{1y}^L(t,y^{v_2}(t),z^{v_2}(t),\tilde{z}^{v_2}(t),v_2(t))\big]dt\\
            &\quad+\big[f_z^L(t,y^{v_2}(t),z^{v_2}(t),\tilde{z}^{v_2}(t),v_2(t))^\top x(t)\\
            &\qquad+L_{1z}^L(t,y^{v_2}(t),z^{v_2}(t),\tilde{z}^{v_2}(t),v_2(t))\big]dW(t)\\
            &\quad+\big[f_{\tilde{z}}^L(t,y^{v_2}(t),z^{v_2}(t),\tilde{z}^{v_2}(t),v_2(t))^\top x(t)\\
            &\qquad+L_{1\tilde{z}}^L(t,y^{v_2}(t),z^{v_2}(t),\tilde{z}^{v_2}(t),v_2(t))\big]d\widetilde{W}(t),\ t\in[0,T],\\
        x(0)&=h_{1y}(y^{v_2}(0)),\ y^{v_2}(T)=\xi,
\end{aligned}
\right.
\end{equation}
where for the simplicity of notations, we have denoted $y^{v_2}(\cdot)=y^{\bar{v}_1,v_2}(\cdot),z^{v_2}(\cdot)=z^{\bar{v}_1,v_2}(\cdot),\tilde{z}^{v_2}(\cdot)=\tilde{z}^{\bar{v}_1,v_2}(\cdot)$, and have defined $\Phi^L$ on $[0,T]\times\mathbb{R}^n\times\mathbb{R}^{n\times m}\times\mathbb{R}^{n\times\widetilde{m}}\times U_2$ as
\begin{equation*}
\begin{aligned}
&\Phi^{L}(t,y^{v_2}(\cdot),z^{v_2}(\cdot),\tilde{z}^{v_2}(\cdot),v_2(t))\\
&:=\Phi(t,y^{\bar{v}_1,v_2}(\cdot),z^{\bar{v}_1,v_2}(\cdot),\tilde{z}^{\bar{v}_1,v_2}(\cdot),\bar{v}_1(t,\hat{y}^{\bar{v}_1,\hat{v}_2}(t),\hat{z}^{\bar{v}_1,\hat{v}_2}(t),
 \hat{\tilde{z}}^{\bar{v}_1,\hat{v}_2}(t),
\hat{v}_2(t),\hat{x}(t)),v_2(t)),
\end{aligned}
\end{equation*}
for $\Phi=f,L_1$, respectively. Then we redefine
\begin{equation}\label{cost func2}
\begin{aligned}
&\widetilde{J}_2(v_2(\cdot);\xi):=J_2(\bar{v}_1(\cdot),v_2(\cdot);\xi)\\
&=\mathbb{E}\bigg[\int_0^TL_2(t,y^{\bar{v}_1,v_2}(t),z^{\bar{v}_1,v_2}(t),\tilde{z}^{\bar{v}_1,v_2}(t),\bar{v}_1(t,\hat{y}^{\bar{v}_1,\hat{v}_2}(t),\hat{z}^{\bar{v}_1,\hat{v}_2}(t),\hat{\tilde{z}}^{\bar{v}_1,
\hat{v}_2}(t),\hat{v}_2(t),\hat{x}(t)),v_2(t))dt\\
&\qquad+h_2(y^{\bar{v}_1,v_2}(0))\bigg]:=\mathbb{E}\bigg[\int_0^TL_2^L(t,y^{v_2}(t),z^{v_2}(t),\tilde{z}^{v_2}(t),v_2(t))dt+h_2(y^{v_2}(0))\bigg],
\end{aligned}
\end{equation}
where $L_2^L:[0,T]\times \mathbb{R}^n\times\mathbb{R}^{n\times m}\times\mathbb{R}^{n\times\widetilde{m}}\times U_2\rightarrow\mathbb{R}$ is similarly defined as above. The target of the leader is to find an optimal control $\bar{v}_2(\cdot)\in\mathcal{U}_2[0,T]$.

Suppose that there exists an optimal control $\bar{v}_2(\cdot)\in\mathcal{U}_2[0,T]$ for the leader, and the corresponding ``state trajectory" $(y^{\bar{v}_1,\bar{v}_2}(\cdot),z^{\bar{v}_1,\bar{v}_2}(\cdot),\tilde{z}^{\bar{v}_1,\bar{v}_2}(\cdot),\bar{x}(\cdot))$ is the solution to \eqref{fbsde2}. Define the Hamiltonian function of the leader $H_2:[0,T]\times\mathbb{R}^n\times\mathbb{R}^{n\times m}\times\mathbb{R}^{n\times\widetilde{m}}\times\mathbb{R}^{k_2}\times\mathbb{R}^n\times\mathbb{R}^n\times\mathbb{R}^{n\times m}\times\mathbb{R}^{n\times\widetilde{m}}\times\mathbb{R}^n\rightarrow\mathbb{R}$ as
\begin{equation}\label{H2}
\begin{aligned}
&H_2(t,y^{v_2},z^{v_2},\tilde{z}^{v_2},v_2,x,p,q_1,q_2,Q)=\big\langle b^L(t,y^{v_2},z^{v_2},\tilde{z}^{v_2},v_2),p\big\rangle+\big\langle\sigma_1^L(t,y^{v_2},z^{v_2},\tilde{z}^{v_2},v_2),q_1\big\rangle\\
&+\big\langle\sigma_2^L(t,y^{v_2},z^{v_2},\tilde{z}^{v_2},v_2),q_2\big\rangle+\big\langle f^L(t,y^{v_2},z^{v_2},\tilde{z}^{v_2},v_2),Q\big\rangle+L_2^L(t,y^{v_2},z^{v_2},\tilde{z}^{v_2},v_2),
\end{aligned}
\end{equation}
where
\begin{equation*}
\begin{aligned}
&b^L(t,y^{v_2},z^{v_2},\tilde{z}^{v_2},v_2):=f_y^L(t,y^{v_2},z^{v_2},\tilde{z}^{v_2},v_2)^\top x+L_{1y}^L(t,y^{v_2},z^{v_2},\tilde{z}^{v_2},v_2),\\
&\sigma_1^L(t,y^{v_2},z^{v_2},\tilde{z}^{v_2},v_2):=f_z^L(t,y^{v_2},z^{v_2},\tilde{z}^{v_2},v_2)^\top x+L_{1z}^L(t,y^{v_2},z^{v_2},\tilde{z}^{v_2},v_2),\\
&\sigma_2^L(t,y^{v_2},z^{v_2},\tilde{z}^{v_2},v_2):=f_{\tilde{z}}^L(t,y^{v_2},z^{v_2},\tilde{z}^{v_2},v_2)^\top x+L_{1\tilde{z}}^L(t,y^{v_2},z^{v_2},\tilde{z}^{v_2},v_2).
\end{aligned}
\end{equation*}

Let $(p(\cdot),q_1(\cdot),q_2(\cdot),Q(\cdot))\in\mathbb{R}^n\times\mathbb{R}^{n\times m}\times\mathbb{R}^{n\times\widetilde{m}}\times\mathbb{R}^n$ be the unique $\mathcal{F}_t$-adapted solution to the adjoint FBSDE of the leader:
\begin{equation}\label{fbsde33}
\left\{
\begin{aligned}
-dp(t)=&\bigg\{\bar{b}_x^L(t)^\top p(t)+\mathbb{E}\big[\bar{b}_{\hat{x}}^{L}(t)^\top p(t)\big|\mathcal{G}_t^1\big]
        +\sum_{j=1}^m\big[\bar{\sigma}_{1x}^{Lj}(t)^\top q_1^j(t)+\mathbb{E}\big[\bar{\sigma}_{1\hat{x}}^{Lj}(t)^\top q_1^j(t)\big|\mathcal{G}_t^1\big]\big]\\
       &+\sum_{j=1}^{\widetilde{m}}\big[\bar{\sigma}_{2x}^{Lj}(t)^\top q_2^j(t)+\mathbb{E}\big[\bar{\sigma}_{2\hat{x}}^{Lj}(t)^{\top}q_2^j(t)\big|\mathcal{G}_t^1\big]\big]
        +\bar{f}_x^L(t)^{\top}Q(t)+\mathbb{E}\big[\bar{f}_{\hat{x}}^L(t)^{\top}Q(t)\big|\mathcal{G}_t^1\big]\\
       &+\bar{L}_{2x}^L(t)^{\top}+\mathbb{E}\big[\bar{L}_{2\hat{x}}^L(t)^{\top}\big|\mathcal{G}_t^1\big]\bigg\}dt-q_1(t)dW(t)-q_2(t)d\widetilde{W}(t),\\
 dQ(t)=&\bigg\{\bar{b}_y^L(t)^\top p(t)+\mathbb{E}\big[\bar{b}_{\hat{y}}^{L}(t)^\top p(t)\big|\mathcal{G}_t^1\big]
        +\sum_{j=1}^m\big[\bar{\sigma}_{1y}^{Lj}(t)^\top q_1^j(t)+\mathbb{E}\big[\bar{\sigma}_{1\hat{y}}^{Lj}(t)^\top q_1^j(t)\big|\mathcal{G}_t^1\big]\big]\\
       &+\sum_{j=1}^{\widetilde{m}}\big[\bar{\sigma}_{2y}^{Lj}(t)^\top q_2^j(t)+\mathbb{E}\big[\bar{\sigma}_{2\hat{y}}^{Lj}(t)^\top q_2^j(t)\big|\mathcal{G}_t^1\big]\big]
        +\bar{f}_y^L(t)^{\top}Q(t)+\mathbb{E}\big[\bar{f}_{\hat{y}}^L(t)^{\top}Q(t)\big|\mathcal{G}_t^1\big]\\
       &+\bar{L}_{2y}^L(t)^{\top}+\mathbb{E}\big[\bar{L}_{2\hat{y}}^L(t)^{\top}\big|\mathcal{G}_t^1\big]\bigg\}dt\\
       &+\bigg\{\bar{b}_z^{L}(t)^\top p(t)+\mathbb{E}\big[\bar{b}_{\hat{z}}^{L}(t)^\top p(t)\big|\mathcal{G}_t^1\big]
        +\sum_{j=1}^m\big[\bar{\sigma}_{1z}^{Lj}(t)^\top q_1^j(t)+\mathbb{E}\big[\bar{\sigma}_{1\hat{z}}^{Lj}(t)^\top q_1^j(t)\big|\mathcal{G}_t^1\big]\big]\\
       &+\sum_{j=1}^{\widetilde{m}}\big[\bar{\sigma}_{2z}^{Lj}(t)^\top q_2^j(t)+\mathbb{E}\big[\bar{\sigma}_{2\hat{z}}^{Lj}(t)^\top q_2^j(t)\big|\mathcal{G}_t^1\big]\big]
        +\bar{f}_z^L(t)^{\top}Q(t)+\mathbb{E}\big[\bar{f}_{\hat{z}}^L(t)^{\top}Q(t)\big|\mathcal{G}_t^1\big]\\
       &+\bar{L}_{2z}^L(t)^{\top}+\mathbb{E}\big[\bar{L}_{2\hat{z}}^L(t)^{\top}\big|\mathcal{G}_t^1\big]\bigg\}dW(t)\\
       &+\bigg\{\bar{b}_{\tilde{z}}^L(t)^\top p(t)+\mathbb{E}\big[\bar{b}_{\hat{\tilde{z}}}^L(t)^{\top}p(t)\big|\mathcal{G}_t^1\big]
        +\sum_{j=1}^m\big[\bar{\sigma}_{1\tilde{z}}^{Lj}(t)^\top q_1^j(t)+\mathbb{E}\big[\bar{\sigma}_{1\hat{\tilde{z}}}^{Lj}(t)^\top q_1^j(t)\big|\mathcal{G}_t^1\big]\big]\\
       &+\sum_{j=1}^{\widetilde{m}}\big[\bar{\sigma}_{2\tilde{z}}^{Lj}(t)^\top q_2^j(t)+\mathbb{E}\big[\bar{\sigma}_{2\hat{\tilde{z}}}^{Lj}(t)^\top q_2^j(t)\big|\mathcal{G}_t^1\big]\big]
        +\bar{f}_{\tilde{z}}^L(t)^{\top}Q(t)+\mathbb{E}\big[\bar{f}_{\hat{\tilde{z}}}^L(t)^{\top}Q(t)\big|\mathcal{G}_t^1\big]\\
       &+\bar{L}_{2\tilde{z}}^L(t)^{\top}+\mathbb{E}\big[\bar{L}_{2\hat{\tilde{z}}}^L(t)^{\top}\big|\mathcal{G}_t^1\big]\bigg\}d\widetilde{W}(t),\ t\in[0,T],\\
  p(T)=&\ 0,\ Q(0)=h_{2y}(y^{\bar{v}_2}(0))+p(0)h_{1yy}(y^{\bar{v}_2}(0)),
\end{aligned}
\right.
\end{equation}
where we have used $\bar{\Phi}^L(t)=\Phi^L(t,y^{\bar{v}_2}(t),z^{\bar{v}_2}(t),\tilde{z}^{\bar{v}_2}(t),\bar{v}_2(t))$ for $\Phi=b,\sigma_1,\sigma_2,L_1,L_2$ and all their derivatives.

Now, we obtain the following theorems for the leader.

\begin{theorem}(Partial information maximum principle for the leader)
Suppose {\bf(A1)}, {\bf(A2)} holds, let $\bar{v}_2(\cdot)\in\mathcal{U}_2[0,T]$ be an optimal control of the leader and $(y^{\bar{v}_2}(\cdot),z^{\bar{v}_2}(\cdot),\tilde{z}^{\bar{v}_2}(\cdot),\bar{x}(\cdot))$ be the corresponding optimal state trajectory. Let $(p(\cdot),q_1(\cdot),q_2(\cdot),Q(\cdot))$ be the solution to \eqref{fbsde33}, then
\begin{equation}\label{maximum condition for leader1}
\begin{aligned}
&\mathbb{E}\big[\big\langle H_{2v_2}(t,y^{\bar{v}_2}(t),z^{\bar{v}_2}(t),\tilde{z}^{\bar{v}_2}(t),\bar{v}_2(t),\bar{x}(t),p(t),q_1(t),q_2(t),Q(t)),v_2-\bar{v}_2(t)\big\rangle\\
&+\big\langle \mathbb{E}\big[H_{2\hat{v}_2}(t,y^{\bar{v}_2}(t),z^{\bar{v}_2}(t),\tilde{z}^{\bar{v}_2}(t),\bar{v}_2(t),\bar{x}(t),p(t),q_1(t),q_2(t),Q(t))\big|\mathcal{G}_t^1\big],
 \hat{v}_2-\hat{\bar{v}}_2(t)\big\rangle\big|\mathcal{G}_t^2\big]\\
&\geq0,\ \mbox{for any }v_2\in U_2,\ a.e.\ t\in[0,T],\ a.s.
\end{aligned}
\end{equation}
\begin{proof}
The maximum condition \eqref{maximum condition for leader1} is similar to that in Theorem 2.3 of Shi et al. \cite{SWX16}, which can be obtained applying the convex variation and adjoint technique. We omit the detail for the space limit. See also Zuo and Min \cite{ZM13}, Wang et al. \cite{WXX17} for optimal control problems of mean-field FBSDEs with partial information.
\end{proof}
\end{theorem}

\begin{theorem}(Partial information verification theorem for the leader)
Let {\bf(A1)}, {\bf(A2)} hold, $\bar{v}_2(\cdot)\in\mathcal{U}_2$ and $(y^{\bar{v}_2}(\cdot),z^{\bar{v}_2}(\cdot),\tilde{z}^{\bar{v}_2}(\cdot),\bar{x}(\cdot))$ be the corresponding state trajectory with $h_{1yy}(y)\equiv\bar{h}_1\in\mathcal{S}^n$. Let $(p(\cdot),q_1(\cdot),q_2(\cdot),Q(\cdot))$ be the solution to \eqref{fbsde33}. For each $t\in[0,T]$, suppose that $H_2$ is convex in $(y,z,\tilde{z},v_2,x)$ and $h_2$ is convex in y, and
\begin{equation}\label{maximum condition for leader2}
\begin{aligned}
&\mathbb{E}\Big[H_2(t,y^{\bar{v}_2}(t),z^{\bar{v}_2}(t),\tilde{z}^{\bar{v}_2}(t),\bar{v}_2(t),\bar{x}(t),p(t),q_1(t),q_2(t),Q(t))\\
&\quad+\mathbb{E}\big[H_2(t,y^{\bar{v}_2}(t),z^{\bar{v}_2}(t),\tilde{z}^{\bar{v}_2}(t),\bar{v}_2(t),\bar{x}(t),p(t),q_1(t),q_2(t),Q(t))\big|\mathcal{G}_t^1\big]\Big|\mathcal{G}_t^2\Big]\\
&=\min_{v_2\in U_2}\mathbb{E}\Big[H_2(t,y^{v_2}(t),z^{v_2}(t),\tilde{z}^{v_2}(t),v_2,\bar{x}(t),p(t),q_1(t),q_2(t),Q(t))\\
&\qquad+\mathbb{E}\big[H_2(t,y^{v_2}(t),z^{v_2}(t),\tilde{z}^{v_2}(t),v_2,\bar{x}(t),p(t),q_1(t),q_2(t),Q(t))\big|\mathcal{G}_t^1\big]\Big|\mathcal{G}_t^2\Big],
\end{aligned}
\end{equation}
holds for $a.e.\ t\in[0,T],\ a.s.$, then $\bar{v}_2(\cdot)$ is an optimal control of the leader.
\end{theorem}

Before we prove this theorem, let us review some preliminaries of the Clarke generalized gradient, which was used to derive sufficient conditions in Yong and Zhou \cite{YZ99}.

Let $\mathcal{M}:\mathcal{X}\rightarrow\mathbb{R}$ be a locally Lipschitz continuous function, where $\mathcal{X}$ is a convex set in $\mathbb{R}^n$.

\begin{definition}
The Clarke generalized gradient of $\mathcal{M}$ at $\hat{x}\in\mathcal{X}$, denoted by $\partial\mathcal{M}(\hat{x})$, is a set defined by
\begin{equation*}
\partial\mathcal{M}(\hat{x}):=\bigg\{\zeta\in\mathbb{R}^n; \langle\zeta,\xi\rangle\leq\limsup_{x\in\mathcal{X},x+h\xi\in\mathcal{X},x\rightarrow\hat{x},h\rightarrow0+}\frac{\mathcal{M}(x+h\xi)-\mathcal{M}(x)}{h},\ \forall\ \xi\in\mathbb{R}^n\bigg\}.
\end{equation*}
\end{definition}

\begin{lemma}
The following properties hold:

(1) $\partial\mathcal{M}(\hat{x})$ is a nonempty convex set and satisfying $\partial(-\mathcal{M})(\hat{x})=-\partial\mathcal{M}(\hat{x})$.

(2) For any set $N\subset\mathcal{X}$ of measure zero,
\begin{equation*}
\partial\mathcal{M}(\hat{x})=co\Big\{\lim_{i\rightarrow\infty}\mathcal{M}_x(x_i): \mathcal{M}\mbox{\ is differentiable at\ }x_i,\ x_i\not\in N\ and\ x_i\rightarrow\hat{x}\Big\},
\end{equation*}
where ``co" denotes the convex hull of a set.

(3) If $\hat{x}$ attains the maximum or minimum of $\mathcal{M}$ over $\mathcal{X}$, then $0\in\partial\mathcal{M}(\hat{x})$.

(4) If $\mathcal{M}$ is a convex (respectively, concave) function, then $p\in\partial\mathcal{M}(\hat{x})$ if and only if
\begin{equation*}
\mathcal{M}(x)-\mathcal{M}(\hat{x})\geq\big(respectively,\ \leq\big)\langle p,\ x-\hat{x}\rangle,\ \forall x\in\mathcal{X}.
\end{equation*}
\end{lemma}

\begin{lemma}
Let $\rho(\cdot,\cdot)$ be a convex or concave function on $\mathbb{R}^d\times U$ with $U\subseteq\mathbb{R}^k$ being a convex body. Assume that $\rho(x,u)$ is Lipschitz continuous in $u$, differentiable in $x$, and $\rho_x(x,u)$ is continuous in $(x,u)$. For a given $(\bar{x},\bar{u})\in\mathbb{R}^d\times U$, if $r\in\partial_u\rho(\bar{x},\bar{u})$, then $(\rho_x(\bar{x},\bar{u}),r)\in\partial_{x,u}\rho(\bar{x},\bar{u})$.
\end{lemma}
\noindent{\it Proof of Theorem 3.4.}
By the condition \eqref{maximum condition for leader2} and Lemma 3.1-(3), we have
\begin{equation}
\begin{aligned}
0\in&\ \mathbb{E}\big[\partial_{v_2}H_2\big(t,y^{\bar{v}_2}(t),z^{\bar{v}_2}(t),\tilde{z}^{\bar{v}_2}(t),\bar{v}_2(t),\bar{x}(t),p(t),q_1(t),q_2(t),Q(t)\big)\\
    &+\mathbb{E}\big[\partial_{\hat{v}_2}H_2\big(t,y^{\bar{v}_2}(t),z^{\bar{v}_2}(t),\tilde{z}^{\bar{v}_2}(t),\bar{v}_2(t),\bar{x}(t),p(t),q_1(t),q_2(t),Q(t)\big)\big|\mathcal{G}_t^1\big]\big|\mathcal{G}_t^2\big].
\end{aligned}
\end{equation}
By Lemma 3.2, we further conclude that
\begin{equation}
\begin{aligned}
&\big(H_{2x,\hat{x}},H_{2y,\hat{y}},H_{2z,\hat{z}},H_{2\tilde{z},\hat{\tilde{z}}},0\big)\\
&\in\partial_{x,\hat{x},y,\hat{y},z,\hat{z},\tilde{z},\hat{\tilde{z}},v_2,\hat{v}_2}H_2\big(t,y^{\bar{v}_2}(t),z^{\bar{v}_2}(t),\tilde{z}^{\bar{v}_2}(t),\bar{v}_2(t),\bar{x}(t),p(t),q_1(t),q_2(t),Q(t)\big),
\end{aligned}
\end{equation}
where
\begin{equation*}
\begin{aligned}
H_{2\Phi,\hat{\Phi}}:&=H_{2\Phi,\hat{\Phi}}\big(t,y^{\bar{v}_2}(t),z^{\bar{v}_2}(t),\tilde{z}^{\bar{v}_2}(t),\bar{v}_2(t),\bar{x}(t),p(t),q_1(t),q_2(t),Q(t)\big)\\
                     &=\bar{b}_\Phi^L(t)^\top p(t)+\mathbb{E}\big[\bar{b}_{\hat{\Phi}}^L(t)^\top p(t)\big|\mathcal{G}_t^1\big]+\sum_{j=1}^m\Big[\bar{\sigma}_{1\Phi}^{Lj}(t)^\top q_1^j(t)
                      +\mathbb{E}\big[\bar{\sigma}_{1\hat{\Phi}}^{Lj}(t)^\top q_1^j(t)\big|\mathcal{G}_t^1\big]\Big]\\
                     &\quad+\sum_{j=1}^{\widetilde{m}}\Big[\bar{\sigma}_{2\Phi}^{Lj}(t)^\top q_2^j(t)+\mathbb{E}\big[\bar{\sigma}_{2\hat{\Phi}}^{Lj}(t)^\top q_2^j(t)\big|\mathcal{G}_t^1\big]\Big]
                      +\bar{f}_\Phi^L(t)^\top Q(t)+\mathbb{E}\big[\bar{f}_{\hat{{\Phi}}}^L(t)^\top Q(t)\big|\mathcal{G}_t^1\big]\\
                     &\quad+\bar{L}_{2\Phi}^L(t)^\top+\mathbb{E}\big[\bar{L}_{2\hat{\Phi}}^L(t)^\top\big|\mathcal{G}_t^1\big],\quad\mbox{for\ }\Phi=x,y,z,\tilde{z}.
\end{aligned}
\end{equation*}
Since $H_2$ is convex in $(x,\hat{x},y,\hat{y},z,\hat{z},\tilde{z},\hat{\tilde{z}},v_2,\bar{v}_2)$, one obtains
\begin{equation}\label{H2convex}
\hspace{-4mm}\begin{aligned}
    &\int_0^T\Big\{H_2\big(t,y^{v_2}(t),z^{v_2}(t),\tilde{z}^{v_2}(t),v_2(t),x(t),p(t),q_1(t),q_2(t),Q(t)\big)\\
    &\qquad-H_2\big(t,y^{\bar{v}_2}(t),z^{\bar{v}_2}(t),\tilde{z}^{\bar{v}_2}(t),\bar{v}_2(t),\bar{x}(t),p(t),q_1(t),q_2(t),Q(t)\big)\Big\}dt\\
\geq&\int_0^T\Big\{\big\langle H_{2x,\hat{x}}\big(t,y^{\bar{v}_2}(t),z^{\bar{v}_2}(t),\tilde{z}^{\bar{v}_2}(t),\bar{v}_2(t),\bar{x}(t),p(t),q_1(t),q_2(t),Q(t)\big),x^{v_2}(t)-x^{\bar{v}_2}(t)\big\rangle\\
    &\quad+\big\langle H_{2y,\hat{y}}\big(t,y^{\bar{v}_2}(t),z^{\bar{v}_2}(t),\tilde{z}^{\bar{v}_2}(t),\bar{v}_2(t),\bar{x}(t),p(t),q_1(t),q_2(t),Q(t)\big),y^{v_2}(t)-y^{\bar{v}_2}(t)\big\rangle\\
    &\quad+\big\langle H_{2z,\hat{z}}\big(t,y^{\bar{v}_2}(t),z^{\bar{v}_2}(t),\tilde{z}^{\bar{v}_2}(t),\bar{v}_2(t),\bar{x}(t),p(t),q_1(t),q_2(t),Q(t)\big),z^{v_2}(t)-z^{\bar{v}_2}(t)\big\rangle\\
    &\quad+\big\langle H_{2\tilde{z},\hat{\tilde{z}}}\big(t,y^{\bar{v}_2}(t),z^{\bar{v}_2}(t),\tilde{z}^{\bar{v}_2}(t),\bar{v}_2(t),\bar{x}(t),p(t),q_1(t),q_2(t),Q(t)\big),\tilde{z}^{v_2}(t)-\tilde{z}^{\bar{v}_2}(t)\big\rangle\Big\}dt.
\end{aligned}
\end{equation}
For any $v_2(\cdot)\in\mathcal{U}_2$, then we consider
\begin{equation*}
\widetilde{J}_2(v_2(\cdot);\xi)-\widetilde{J}_2(\bar{v}_2(\cdot);\xi)=\textrm{I}+\textrm{II},
\end{equation*}
with
\begin{equation*}
\begin{aligned}
\textrm{I}&=\mathbb{E}\int_0^T\Big[L_2^L(t,y^{v_2}(t),z^{v_2}(t),\tilde{z}^{v_2}(t),v_2(t))-L_2^L(t,y^{\bar{v}_2}(t),z^{\bar{v}_2}(t),\tilde{z}^{\bar{v}_2}(t),\bar{v}_2(t))\Big]dt,\\
\textrm{II}&=\mathbb{E}\big[h_2(y^{v_2}(0))-h_2(y^{\bar{v}_2}(0))\big].
\end{aligned}
\end{equation*}
Since $h_2$ is convex in $y$, we get
\begin{equation}\label{h2convex}
\textrm{II}\geq\mathbb{E}\big[\big\langle h_{2y}(y^{\bar{v}_2}(0)),y^{v_2}(0)-y^{\bar{v}_2}(0)\big\rangle\big].
\end{equation}
Applying It\^{o}'s formula to $-\langle Q(\cdot),y^{v_2}(\cdot)-y^{\bar{v}_2}(\cdot)\rangle$, then taking expectation on both sides, we have
\begin{equation*}
\begin{aligned}
&\mathbb{E}\big[\langle Q(0), y^{v_2}(0)-y^{\bar{v}_2}(0)\rangle-\langle Q(T), y^{v_2}(T)-y^{\bar{v}_2}(T)\rangle\big]\\
&=\mathbb{E}\int_0^T\Big\{-\big\langle H_{2y,\hat{y}}\big(t,y^{\bar{v}_2}(t),z^{\bar{v}_2}(t),\tilde{z}^{\bar{v}_2}(t),\bar{v}_2(t),\bar{x}(t),p(t),q_1(t),q_2(t),Q(t)\big),y^{v_2}(t)-y^{\bar{v}_2}(t)\big\rangle\\
&\qquad-\big\langle H_{2z,\hat{z}}\big(t,y^{\bar{v}_2}(t),z^{\bar{v}_2}(t),\tilde{z}^{\bar{v}_2}(t),\bar{v}_2(t),\bar{x}(t),p(t),q_1(t),q_2(t),Q(t)\big),z^{v_2}(t)-z^{\bar{v}_2}(t)\big\rangle\\
&\qquad-\big\langle H_{2\tilde{z},\hat{\tilde{z}}}\big(t,y^{\bar{v}_2}(t),z^{\bar{v}_2}(t),\tilde{z}^{\bar{v}_2}(t),\bar{v}_2(t),\bar{x}(t),p(t),q_1(t),q_2(t),Q(t)\big),\tilde{z}^{v_2}(t)-\tilde{z}^{\bar{v}_2}(t)\big\rangle\\
&\qquad+\big\langle Q(t), f^L(t,y^{v_2}(t),z^{v_2}(t),\tilde{z}^{v_2}(t),v_2(t))-f^L(t,y^{\bar{v}_2}(t),z^{\bar{v}_2}(t),\tilde{z}^{\bar{v}_2}(t),\bar{v}_2(t))\big\rangle\Big\}dt.
\end{aligned}
\end{equation*}
Similarly, we get
\begin{equation*}
\begin{aligned}
&\mathbb{E}\big[\langle p(0), x^{v_2}(0)-x^{\bar{v}_2}(0)\rangle-\langle p(T), x^{v_2}(T)-x^{\bar{v}_2}(T)\rangle\big]\\
&=\mathbb{E}\int_0^T\Big\{\langle H_{2x,\hat{x}}\big(t,y^{\bar{v}_2}(t),z^{\bar{v}_2}(t),\tilde{z}^{\bar{v}_2}(t),\bar{v}_2(t),\bar{x}(t),p(t),q_1(t),q_2(t),Q(t)\big),x^{v_2}(t)-x^{\bar{v}_2}(t)\rangle\\
&\qquad-\langle p(t),b^L(t,y^{v_2}(t),z^{v_2}(t),\tilde{z}^{v_2}(t),v_2(t))-b^L(t,y^{\bar{v}_2}(t),z^{\bar{v}_2}(t),\tilde{z}^{\bar{v}_2}(t),\bar{v}_2(t))\rangle\\
&\qquad-\langle q_1(t),\sigma_1^L(t,y^{v_2}(t),z^{v_2}(t),\tilde{z}^{v_2}(t),v_2(t))-\sigma_1^L(t,y^{\bar{v}_2}(t),z^{\bar{v}_2}(t),\tilde{z}^{\bar{v}_2}(t),\bar{v}_2(t))\rangle\\
&\qquad-\langle q_2(t),\sigma_2^L(t,y^{v_2}(t),z^{v_2}(t),\tilde{z}^{v_2}(t),v_2(t))-\sigma_2^L(t,y^{\bar{v}_2}(t),z^{\bar{v}_2}(t),\tilde{z}^{\bar{v}_2}(t),\bar{v}_2(t))\rangle\Big\}dt.
\end{aligned}
\end{equation*}
With the help of the initial value and terminal value of $Q(\cdot)$ and $p(\cdot)$, respectively, in the equation \eqref{fbsde33}, and the condition $h_{1yy}(y)\equiv\bar{h}_1$, we have
\begin{equation}\label{c1}
\begin{aligned}
&\mathbb{E}\big[\big\langle h_{2y}(y^{\bar{v}_2}(0)),y^{v_2}(0)-y^{\bar{v}_2}(0)\big\rangle+\big\langle p(0)\bar{h}_1, y^{v_2}(0)-y^{\bar{v}_2}(0)\big\rangle\big]\\
&=\mathbb{E}\int_0^T\Big\{-\big\langle H_{2y,\hat{y}}\big(t,y^{\bar{v}_2}(t),z^{\bar{v}_2}(t),\tilde{z}^{\bar{v}_2}(t),\bar{v}_2(t),\bar{x}(t),p(t),q_1(t),q_2(t),Q(t)\big),y^{v_2}(t)-y^{\bar{v}_2}(t)\big\rangle\\
&\qquad-\big\langle H_{2z,\hat{z}}\big(t,y^{\bar{v}_2}(t),z^{\bar{v}_2}(t),\tilde{z}^{\bar{v}_2}(t),\bar{v}_2(t),\bar{x}(t),p(t),q_1(t),q_2(t),Q(t)\big),z^{v_2}(t)-z^{\bar{v}_2}(t)\big\rangle\\
&\qquad-\big\langle H_{2\tilde{z},\hat{\tilde{z}}}\big(t,y^{\bar{v}_2}(t),z^{\bar{v}_2}(t),\tilde{z}^{\bar{v}_2}(t),\bar{v}_2(t),\bar{x}(t),p(t),q_1(t),q_2(t),Q(t)\big),\tilde{z}^{v_2}(t)-\tilde{z}^{\bar{v}_2}(t)\big\rangle\\
&\qquad+\big\langle Q(t), f^L(t,y^{v_2}(t),z^{v_2}(t),\tilde{z}^{v_2}(t),v_2(t))-f^L(t,y^{\bar{v}_2}(t),z^{\bar{v}_2}(t),\tilde{z}^{\bar{v}_2}(t),\bar{v}_2(t))\big\rangle\Big\}dt,
\end{aligned}
\end{equation}
and
\begin{equation}\label{c2}
\begin{aligned}
&\mathbb{E}\big[\langle p(0)\bar{h}_1,y^{v_2}(0)-y^{\bar{v}_2}(0)\rangle\big]
 \leq\mathbb{E}\big[\langle p(0),h_{1y}(y^{v_2}(0))-h_{1y}(y^{\bar{v}_2}(0))\rangle\big]=\mathbb{E}\big[\langle p(0),x^{v_2}(0)-x^{\bar{v}_2}(0)\rangle\big]\\
&=\mathbb{E}\int_0^T\Big\{\big\langle H_{2x,\hat{x}}\big(t,y^{\bar{v}_2}(t),z^{\bar{v}_2}(t),\tilde{z}^{\bar{v}_2}(t),\bar{v}_2(t),\bar{x}(t),p(t),q_1(t),q_2(t),Q(t)\big),x^{v_2}(t)-x^{\bar{v}_2}(t)\big\rangle\\
&\qquad-\big\langle p(t),b^L(t,y^{v_2}(t),z^{v_2}(t),\tilde{z}^{v_2}(t),v_2(t))-b^L(t,y^{\bar{v}_2}(t),z^{\bar{v}_2}(t),\tilde{z}^{\bar{v}_2}(t),\bar{v}_2(t))\big\rangle\\
&\qquad-\big\langle q_1(t),\sigma_1^L(t,y^{v_2}(t),z^{v_2}(t),\tilde{z}^{v_2}(t),v_2(t))-\sigma_1^L(t,y^{\bar{v}_2}(t),z^{\bar{v}_2}(t),\tilde{z}^{\bar{v}_2}(t),\bar{v}_2(t))\big\rangle\\
&\qquad-\big\langle q_2(t),\sigma_2^L(t,y^{v_2}(t),z^{v_2}(t),\tilde{z}^{v_2}(t),v_2(t))-\sigma_2^L(t,y^{\bar{v}_2}(t),z^{\bar{v}_2}(t),\tilde{z}^{\bar{v}_2}(t),\bar{v}_2(t))\big\rangle\Big\}dt.
\end{aligned}
\end{equation}
Putting \eqref{c2} into \eqref{c1}, we derive the following inequality
\begin{equation}
\begin{aligned}
&\mathbb{E}\big\langle h_{2y}(y^{\bar{v}_2}(0)),y^{v_2}(0)-y^{\bar{v}_2}(0)\big\rangle\\
&\geq\mathbb{E}\int_0^T\Big\{-\big\langle H_{2x,\hat{x}}\big(t,y^{\bar{v}_2}(t),z^{\bar{v}_2}(t),\tilde{z}^{\bar{v}_2}(t),\bar{v}_2(t),\bar{x}(t),p(t),q_1(t),q_2(t),Q(t)\big),x^{v_2}(t)-x^{\bar{v}_2}(t)\big\rangle\\
&\qquad-\big\langle H_{2y,\hat{y}}\big(t,y^{\bar{v}_2}(t),z^{\bar{v}_2}(t),\tilde{z}^{\bar{v}_2}(t),\bar{v}_2(t),\bar{x}(t),p(t),q_1(t),q_2(t),Q(t)\big),x^{v_2}(t)-x^{\bar{v}_2}(t)\big\rangle\\
&\qquad-\big\langle H_{2z,\hat{z}}\big(t,y^{\bar{v}_2}(t),z^{\bar{v}_2}(t),\tilde{z}^{\bar{v}_2}(t),\bar{v}_2(t),\bar{x}(t),p(t),q_1(t),q_2(t),Q(t)\big),z^{v_2}(t)-z^{\bar{v}_2}(t)\big\rangle\\
&\qquad-\big\langle H_{2\tilde{z},\hat{\tilde{z}}}\big(t,y^{\bar{v}_2}(t),z^{\bar{v}_2}(t),\tilde{z}^{\bar{v}_2}(t),\bar{v}_2(t),\bar{x}(t),p(t),q_1(t),q_2(t),Q(t)\big),\tilde{z}^{v_2}(t)-\tilde{z}^{\bar{v}_2}(t)\big\rangle\\
&\qquad+\big\langle Q(t), f^L(t,y^{v_2}(t),z^{v_2}(t),\tilde{z}^{v_2}(t),v_2(t))-f^L(t,y^{\bar{v}_2}(t),z^{\bar{v}_2}(t),\tilde{z}^{\bar{v}_2}(t),\bar{v}_2(t))\big\rangle\\
&\qquad+\big\langle p(t), b^L(t,y^{v_2}(t),z^{v_2}(t),\tilde{z}^{v_2}(t),v_2(t))-b^L(t,y^{\bar{v}_2}(t),z^{\bar{v}_2}(t),\tilde{z}^{\bar{v}_2}(t),\bar{v}_2(t))\big\rangle\\
&\qquad+\big\langle q_1(t),\sigma_1^L(t,y^{v_2}(t),z^{v_2}(t),\tilde{z}^{v_2}(t),v_2(t))-\sigma_1^L(t,y^{\bar{v}_2}(t),z^{\bar{v}_2}(t),\tilde{z}^{\bar{v}_2}(t),\bar{v}_2(t))\big\rangle\\
&\qquad+\big\langle q_2(t),\sigma_2^L(t,y^{v_2}(t),z^{v_2}(t),\tilde{z}^{v_2}(t),v_2(t))-\sigma_2^L(t,y^{\bar{v}_2}(t),z^{\bar{v}_2}(t),\tilde{z}^{\bar{v}_2}(t),\bar{v}_2(t))\big\rangle\Big\}dt\\
&=\mathbb{E}\int_0^T\Big\{-\big\langle H_{2x,\hat{x}}\big(t,y^{\bar{v}_2}(t),z^{\bar{v}_2}(t),\tilde{z}^{\bar{v}_2}(t),\bar{v}_2(t),\bar{x}(t),p(t),q_1(t),q_2(t),Q(t)\big),x^{v_2}(t)-x^{\bar{v}_2}(t)\big\rangle\\
&\qquad-\big\langle H_{2y,\hat{y}}\big(t,y^{\bar{v}_2}(t),z^{\bar{v}_2}(t),\tilde{z}^{\bar{v}_2}(t),\bar{v}_2(t),\bar{x}(t),p(t),q_1(t),q_2(t),Q(t)\big),x^{v_2}(t)-x^{\bar{v}_2}(t)\rangle\\
&\qquad-\big\langle H_{2z,\hat{z}}\big(t,y^{\bar{v}_2}(t),z^{\bar{v}_2}(t),\tilde{z}^{\bar{v}_2}(t),\bar{v}_2(t),\bar{x}(t),p(t),q_1(t),q_2(t),Q(t)\big),z^{v_2}(t)-z^{\bar{v}_2}(t)\rangle\\
&\qquad-\big\langle H_{2\tilde{z},\hat{\tilde{z}}}\big(t,y^{\bar{v}_2}(t),z^{\bar{v}_2}(t),\tilde{z}^{\bar{v}_2}(t),\bar{v}_2(t),\bar{x}(t),p(t),q_1(t),q_2(t),Q(t)\big),\tilde{z}^{v_2}(t)-\tilde{z}^{\bar{v}_2}(t)\rangle\\
&\qquad+H_2\big(t,y^{v_2}(t),z^{v_2}(t),\tilde{z}^{v_2}(t),v_2(t),x(t),p(t),q_1(t),q_2(t),Q(t)\big)\\
&\qquad-H_2\big(t,y^{\bar{v}_2}(t),z^{\bar{v}_2}(t),\tilde{z}^{\bar{v}_2}(t),\bar{v}_2(t),\bar{x}(t),p(t),q_1(t),q_2(t),Q(t)\big)\\
&\qquad-L_2^L(t,y^{v_2}(t),z^{v_2}(t),\tilde{z}^{v_2}(t),v_2(t))+L_2^L(t,y^{\bar{v}_2}(t),z^{\bar{v}_2}(t),\tilde{z}^{\bar{v}_2}(t),\bar{v}_2(t))\Big\}dt.
\end{aligned}
\end{equation}
Because of \eqref{H2convex}, we get
\begin{equation}
\begin{aligned}
\textrm{II}&\geq\mathbb{E}\big\langle h_{2y}(y^{\bar{v}_2}(0)),y^{v_2}(0)-y^{\bar{v}_2}(0)\big\rangle\\
&\geq\mathbb{E}\int_0^T\Big\{-L_2^L(t,y^{v_2}(t),z^{v_2}(t),\tilde{z}^{v_2}(t),v_2(t))+L_2^L(t,y^{\bar{v}_2}(t),z^{\bar{v}_2}(t),\tilde{z}^{\bar{v}_2}(t),\bar{v}_2(t))\Big\}dt.
\end{aligned}
\end{equation}
Therefore, we have
\begin{equation*}
\begin{aligned}
&\widetilde{J}_2(v_2(\cdot);\xi)-\widetilde{J}_2(\bar{v}_2(\cdot);\xi)\\
&\geq\mathbb{E}\int_0^{T}\Big\{L_2^L(t,y^{v_2}(t),z^{v_2}(t),\tilde{z}^{v_2}(t),v_2(t))-L_2^L(t,y^{\bar{v}_2}(t),z^{\bar{v}_2}(t),\tilde{z}^{\bar{v}_2}(t),\bar{v}_2(t))\Big\}dt\\
&\quad+\mathbb{E}\int_0^T\Big\{-L_2^L(t,y^{v_2}(t),z^{v_2}(t),\tilde{z}^{v_2}(t),v_2(t))+L_2^L(t,y^{\bar{v}_2}(t),z^{\bar{v}_2}(t),\tilde{z}^{\bar{v}_2}(t),\bar{v}_2(t)\Big\}dt=0.
\end{aligned}
\end{equation*}
Since $v_2(\cdot)\in\mathcal{U}_2$ is arbitrary, the desired result follows. $\Box$

\section{The Linear Quadratic Problem}

In this section, we aim to study an LQ case with $m=\widetilde{m}=1$ and to give some explicit forms of the previous results. Moreover, we only consider the special case when the follower's information filtration is $\mathcal{G}_t^1=\sigma\{W(r):0\leq r\leq t\}$, and the leader's information filtration is $\mathcal{G}_t^2=\mathcal{F}_t=\sigma\{W(r),\widetilde{W}(r):0\leq r\leq t\}$.

\subsection{Optimization for The Follower}

We consider the following controlled linear BSDE
\begin{equation}\label{LQbsde}
\left\{
\begin{aligned}
-dy^{v_1,v_2}(t)&=\big[A(t)y^{v_1,v_2}(t)+B_1(t)v_1(t)+B_2(t)v_2(t)+C_1(t)z^{v_1,v_2}(t)\\
                &\quad\ +C_2(t)\tilde{z}^{v_1,v_2}(t)\big]dt-z^{v_1,v_2}(t)dW(t)-\tilde{z}^{v_1,v_2}(t)d\widetilde{W}(t),\ t\in[0,T],\\
  y^{v_1,v_2}(T)&=\xi,
\end{aligned}
\right.
\end{equation}
and the cost functional
\begin{equation}\label{LQ cost functional}
\begin{aligned}
&J_1(v_1(\cdot),v_2(\cdot);\xi)=\frac{1}{2}\mathbb{E}\bigg\{\int_0^T\Big[\big\langle Q_1(t)y^{v_1,v_2}(t),y^{v_1,v_2}(t)\big\rangle+\big\langle R_1(t)v_1(t),v_1(t)\big\rangle\\
&\ +\big\langle S_1(t)z^{v_1,v_2}(t),z^{v_1,v_2}(t)\big\rangle+\big\langle S_2(t)\tilde{z}^{v_1,v_2}(t),\tilde{z}^{v_1,v_2}(t)\big\rangle\Big]dt+\big\langle G_1y^{v_1,v_2}(0),y^{v_1,v_2}(0)\big\rangle\bigg\},
\end{aligned}
\end{equation}
where $A(\cdot),B_1(\cdot),B_2(\cdot),C_1(\cdot),C_2(\cdot),Q_1(\cdot),R_1(\cdot),S_1(\cdot),S_2(\cdot)$ are deterministic matrix-valued functions, and $G_1$ is an $\mathbb{R}^n$-valued vector. We give the following assumptions.

{\bf(L1)}
\begin{equation*}
\left\{
\begin{aligned}
&A(\cdot),C_1(\cdot),C_2(\cdot)\in L^\infty(0,T;\mathbb{R}^{n\times n}),\\
&B_1(\cdot)\in L^\infty(0,T;\mathbb{R}^{n\times k_1}),\ B_2(\cdot)\in L^\infty(0,T;\mathbb{R}^{n\times k_2}).
\end{aligned}
\right.
\end{equation*}

{\bf(L2)}
\begin{equation*}
\left\{
\begin{aligned}
&Q_1(\cdot),S_1(\cdot),S_2(\cdot)\in L^\infty(0,T;\mathcal{S}^n),\ Q_1(\cdot),S_1(\cdot),S_2(\cdot)\geq0,\\
&R_1(\cdot)\in L^\infty(0,T;\mathcal{S}^{k_1}),\ R_1(\cdot)>0,\ G_1\in\mathcal{S}^n,\ G_1\geq0.
\end{aligned}
\right.
\end{equation*}

For the given control $v_2(\cdot)$, suppose that there exists a $\mathcal{G}_t^1$-adapted optimal control $\bar{v}_1(\cdot)$ of the follower, and the corresponding optimal state is $(y^{\bar{v}_1,v_2}(\cdot),z^{\bar{v}_1,v_2}(\cdot))$. According to Theorem 3.1, it is necessary to satisfy
\begin{equation}
\mathbb{E}\big[H_{1v_1}(t,y^{\bar{v}_1,v_2}(t),z^{\bar{v}_1,v_2}(t),\tilde{z}^{\bar{v}_1,v_2}(t),\bar{v}_1(t),v_2(t),x(t))\big|\mathcal{G}_t^1\big]=0,\ a.e.\ t\in[0,T],\ a.s.,
\end{equation}
where the Hamiltonian function of the follower is
\begin{equation}\label{Hamiltonian1}
\begin{aligned}
H_1(t,y,z,\tilde{z},v_1,v_2,x)=&-\frac{1}{2}\langle Q_1(t)y,y\rangle-\frac{1}{2}\langle R_1(t)v_1,v_1\rangle-\frac{1}{2}\langle S_1(t)z,z\rangle-\frac{1}{2}\langle S_2(t)\tilde{z},\tilde{z}\rangle\\
                             &-\langle A(t)y+B_1(t)v_1+B_2(t)v_2+C_1(t)z+C_2(t)\tilde{z},x\rangle.
\end{aligned}
\end{equation}
Then we have
\begin{equation}\label{barv1}
\bar{v}_1(t)=-R_1^{-1}(t)B_1(t)^\top\hat{x}(t),\ a.e.\ t\in[0,T],\ a.s.,
\end{equation}
where $(y^{\bar{v}_1,v_2}(\cdot),z^{\bar{v}_1,v_2}(\cdot),\tilde{z}^{\bar{v}_1,v_2}(\cdot),x(\cdot))$ satisfies
\begin{equation}\label{LQbsde2}
\left\{
\begin{aligned}
-dy^{\bar{v}_1,v_2}(t)&=\big[A(t)y^{\bar{v}_1,v_2}(t)-B_1(t)R_1^{-1}(t)B_1(t)^\top\hat{x}(t)+B_2(t)v_2(t)+C_1(t)z^{\bar{v}_1,v_2}(t)\\
                      &\quad\ +C_2(t)\tilde{z}^{\bar{v}_1,v_2}(t)\big]dt-z^{\bar{v}_1,v_2}(t)dW(t)-\tilde{z}^{\bar{v}_1,v_2}(t)d\widetilde{W}(t),\\
                 dx(t)&=\big[A(t)^\top x(t)+Q(t)y^{\bar{v}_1,v_2}(t)\big]dt+\big[C_1(t)^\top x(t)+S_1(t)z^{\bar{v}_1,v_2}(t)\big]dW(t)\\
                      &\quad\ +\big[C_2(t)^\top x(t)+S_2(t)\tilde{z}^{\bar{v}_1,v_2}(t)\big]d\widetilde{W}(t),\ t\in[0,T],\\
                  x(0)&=G_1y^{\bar{v}_1,v_2}(0),\ y^{\bar{v}_1,v_2}(T)=\xi.
\end{aligned}
\right.
\end{equation}

We then wish to obtain the state estimate feedback form of $\bar{v}_1(\cdot)$, from \eqref{barv1}. Therefore we set
\begin{equation}\label{relation1}
y^{\bar{v}_1,v_2}(t)=-P_1(t)x(t)-\phi(t),
\end{equation}
for some deterministic and differentiable $\mathbb{R}^{n\times n}$-matrix-valued function $P_1(\cdot)$ with $P_1(T)=0$, and $\mathbb{R}^n$-valued, $\mathcal{F}_t$-adapted process $\phi(\cdot)$ which admits the following BSDE:
\begin{equation}\label{phi}
\left\{
\begin{aligned}
d\phi(t)&=\Gamma(t)dt+\eta(t)dW(t),\ t\in[0,T],\\
\phi(T)&=-\xi.
\end{aligned}
\right.
\end{equation}
In the above equation, $\Gamma(\cdot),\eta(\cdot)$ are both $\mathbb{R}^n$-valued, $\mathcal{F}_t$-adapted processes, which are to be determined later. Applying It\^{o}'s formula to \eqref{relation1}, we get
\begin{equation}\label{after ito formula}
\begin{aligned}
dy^{\bar{v}_1,v_2}(t)=&\big[-\dot{P_1}(t)x(t)-P_1(t)A(t)^\top x(t)+P_1(t)Q_1(t)P_1(t)x(t)+P_1(t)Q_1(t)\phi(t)\\
                      &\ -\Gamma(t)\big]dt+\big[-P_1(t)(C_1(t)^\top x(t)+S_1(t)z^{\bar{v}_1,v_2}(t))-\eta(t)\big]dW(t)\\
                      &+\big[-P_1(t)(C_2(t)^\top x(t)+S_2(t)\tilde{z}^{\bar{v}_1,v_2}(t))\big]d\widetilde{W}(t).
\end{aligned}
\end{equation}
Comparing \eqref{after ito formula} with the first equation of \eqref{LQbsde2}, we derive
\begin{equation}\label{dt term}
\begin{aligned}
&\big[-A(t)P_1(t)-\dot{P}_1(t)-P_1(t)A(t)^\top+P_1(t)Q_1(t)P_1(t)\big]x(t)-B_1(t)R_1^{-1}(t)B_1(t)^\top\hat{x}(t)\\
&+B_2(t)v_2(t)+C_1(t)z^{\bar{v}_1,v_2}(t)+C_2(t)\tilde{z}^{\bar{v}_1,v_2}(t)+\big[P_1(t)Q_1(t)-A(t)\big]\phi(t)=\Gamma(t),
\end{aligned}
\end{equation}
and
\begin{equation}\label{z term}
\left\{
\begin{aligned}
&z^{\bar{v}_1,v_2}(t)=-(P_1(t)S_1(t)+I)^{-1}P_1(t)C_1(t)^\top x(t)-(P_1(t)S_1(t)+I)^{-1}\eta(t),\\
&\tilde{z}^{\bar{v}_1,v_2}(t)=-(P_1(t)S_2(t)+I)^{-1}P_1(t)C_2(t)^{\top}x(t).
\end{aligned}
\right.
\end{equation}
Taking $\mathbb{E}[\cdot|\mathcal{G}_t^1]$ on both sides of \eqref{relation1}, \eqref{phi}, \eqref{dt term} and \eqref{z term}, we get
\begin{equation}\label{filter y x}
\left\{
\begin{aligned}
&\hat{y}^{\bar{v}_1,\hat{v}_2}(t)=-P_1(t)\hat{x}(t)-\hat{\phi}(t),\\
&d\hat{\phi}(t)=\hat{\Gamma}(t)dt+\hat{\eta}(t)dW(t),\ \hat{\phi}(T)=-\xi,
\end{aligned}
\right.
\end{equation}
\begin{equation}\label{hatGamma}
\begin{aligned}
\hat{\Gamma}(t)&=\big[-\dot{P}_1(t)-A(t)P_1(t)-P_1(t)A(t)^{\top}+P_1(t)Q_1(t)P_1(t)-B_1(t)R_1^{-1}(t)B_1(t)^\top\big]\hat{x}(t)\\
               &\quad+B_2(t)\hat{v}_2(t)+C_1(t)\hat{z}^{\bar{v}_1,\hat{v}_2}(t)+C_2(t)\hat{\tilde{z}}^{\bar{v}_1,\hat{v}_2}(t)+\big[P_1(t)Q_1(t)-A(t)\big]\hat{\phi}(t),
\end{aligned}
\end{equation}
and
\begin{equation}\label{hatz and hatbarz}
\left\{
\begin{aligned}
      \hat{z}^{\bar{v}_1,\hat{v}_2}(t)&=-(P_1(t)S_1(t)+I)^{-1}P_1(t)C_1(t)^\top\hat{x}(t)-(P_1(t)S_1(t)+I)^{-1}\hat{\eta}(t),\\
\hat{\tilde{z}}^{\bar{v}_1,\hat{v}_2}(t)&=-(P_1(t)S_2(t)+I)^{-1}P_1(t)C_2(t)^\top\hat{x}(t).
\end{aligned}
\right.
\end{equation}
Inserting \eqref{hatz and hatbarz} into \eqref{hatGamma}, we get
\begin{equation}\label{hatGamma2}
\begin{aligned}
\hat{\Gamma}(t)&=\big[-\dot{P}_1(t)-A(t)P_1(t)-P_1(t)A(t)^\top+P_1(t)Q_1(t)P_1(t)-B_1(t)R_1^{-1}(t)B_1(t)^\top\\
               &\qquad-C_1(t)(P_1(t)S_1(t)+I)^{-1}P_1(t)C_1(t)^\top-C_2(t)(P_1(t)S_2(t)+I)^{-1}P_1(t)C_2(t)^\top\big]\hat{x}(t)\\
               &\quad+B_2(t)\hat{v}_2(t)-C_1(t)(P_1(t)S_1(t)+I)^{-1}\hat{\eta}(t)+\big[P_1(t)Q_1(t)-A(t)\big]\hat{\phi}(t).
\end{aligned}
\end{equation}
If the following Riccati equation
\begin{equation}\label{P_1}
\left\{
\begin{aligned}
&\dot{P}_1(t)+A(t)P_1(t)+P_1(t)A(t)^\top-P_1(t)Q_1(t)P_1(t)+B_1(t)R_1^{-1}(t)B_1(t)^\top\\
&+C_1(t)(P_1(t)S_1(t)+I)^{-1}P_1(t)C_1(t)^\top+C_2(t)(P_1(t)S_2(t)+I)^{-1}P_1(t)C_2(t)^\top=0,\ t\in[0,T],\\
&P_1(T)=0,
\end{aligned}
\right.
\end{equation}
admits a unique differentiable solution $P_1(\cdot)$, then
\begin{equation}\label{hatphi}
\left\{
\begin{aligned}
-d\hat{\phi}(t)&=\big\{-[P_1(t)Q_1(t)-A(t)]\hat{\phi}(t)+C_1(t)(P_1(t)S_1(t)+I)^{-1}\hat{\eta}(t)\\
               &\qquad-B_2(t)\hat{v}_2(t)\big\}dt-\hat{\eta}(t)dW(t),\ t\in[0,T],\\
  \hat{\phi}(T)&=-\hat{\xi}.
\end{aligned}
\right.
\end{equation}

Next, we set
\begin{equation}\label{relation2}
x(t)=P_2(t)y^{\bar{v}_1,v_2}(t)+\varphi(t),
\end{equation}
where $P_2(\cdot)$ is a deterministic and differentiable $\mathbb{R}^{n\times n}$-matrix-valued function with $P_2(0)=G_1$, and $\varphi(\cdot)$ is an $\mathbb{R}^n$-valued, $\mathcal{F}_t$-adapted process which satisfies the following SDE:
\begin{equation}\label{varphi}
\left\{
\begin{aligned}
d\varphi(t)&=\alpha(t)dt+\beta(t)dW(t),\ t\in[0,T],\\
 \varphi(0)&=0,
\end{aligned}
\right.
\end{equation}
where $\alpha(\cdot),\beta(\cdot)$ are both $\mathbb{R}^n$-valued, $\mathcal{F}_t$-adapted processes which are to be determined later.

Applying It\^{o}'s formula to \eqref{relation2}, we get
\begin{equation}\label{after ito formula2}
\begin{aligned}
dx(t)&=\big[\dot{P}_2(t)y^{\bar{v}_1,v_2}(t)-P_2(t)A(t)y^{\bar{v}_1,v_2}(t)+P_2(t)B_1(t)R_1^{-1}(t)B_1(t)^{\top}\hat{x}(t)\\
     &\qquad-P_2(t)B_2(t)v_2(t)-P_2(t)C_1(t)z^{\bar{v}_1,v_2}(t)-P_2(t)C_2(t)\tilde{z}^{\bar{v}_1,v_2}(t)+\alpha(t\big]dt\\
     &\quad+\big[P_2(t)z^{\bar{v}_1,v_2}(t)+\beta(t)\big]dW(t)+P_2(t)\tilde{z}^{\bar{v}_1,v_2}(t)d\widetilde{W}(t).
\end{aligned}
\end{equation}
Comparing \eqref{after ito formula2} with the second equation of \eqref{LQbsde2}, we have
\begin{equation}\label{after comparision}
\left\{
\begin{aligned}
&\alpha(t)=\big[-\dot{P}_2(t)+P_2(t)A(t)+Q_1(t)\big]y^{\bar{v}_1,v_2}(t)+A(t)^\top x(t)+P_2(t)C_1(t)z^{\bar{v}_1,v_2}(t)\\
&\qquad\quad+P_2(t)C_2(t)\tilde{z}^{\bar{v}_1,v_2}(t)-P_2(t)B_1(t)R_1^{-1}(t)B_1(t)^\top\hat{x}(t)+P_2(t)B_2(t)v_2(t),\\
&C_1(t)^\top x(t)+S_1(t)z^{\bar{v}_1,v_2}(t)=P_2(t)z^{\bar{v}_1,v_2}(t)+\beta(t),\\
&C_2(t)^\top x(t)+S_2(t)\tilde{z}^{\bar{v}_1,v_2}(t)=P_2(t)\tilde{z}^{\bar{v}_1,v_2}(t).
\end{aligned}
\right.
\end{equation}
Taking $\mathbb{E}[\cdot|\mathcal{G}_t^1]$ on both sides of \eqref{relation2}, \eqref{varphi} and \eqref{after comparision}, we have
\begin{equation}
\left\{
\begin{aligned}
&\hat{x}(t)=P_2(t)\hat{y}^{\bar{v}_1,\hat{v}_2}(t)+\hat{\varphi}(t),\\
&d\hat{\varphi}(t)=\hat{\alpha}(t)dt+\hat{\beta}(t)dW(t),\ \hat{\varphi}(0)=0,
\end{aligned}
\right.
\end{equation}
\begin{equation}
\begin{aligned}
\hat{\alpha}(t)&=\big[-\dot{P}_2(t)+P_2(t)A(t)+A(t)^\top P_2(t)+Q_1(t)-P_2(t)B_1(t)R_1^{-1}(t)B_1(t)^\top P_2(t)\\
               &\qquad-P_2(t)C_1(t)(P_1(t)S_1(t)+I)^{-1}P_1(t)C_1(t)^\top P_2(t)\\
               &\qquad-P_2(t)C_2(t)(P_1(t)S_2(t)+I)^{-1}P_1(t)C_2(t)^\top P_2(t)\big]\hat{y}^{\bar{v}_1,\hat{v}_2}(t)\\
               &\quad+\big[A(t)^\top-P_2(t)B_1(t)R_1^{-1}(t)B_1(t)^\top-P_2(t)C_1(t)(P_1(t)S_1(t)+I)^{-1}P_1(t)C_1(t)^\top\\
               &\qquad-P_2(t)C_2(t)(P_1(t)S_2(t)+I)^{-1}P_1(t)C_2(t)^\top\big]\hat{\varphi}(t)\\
               &\qquad-P_2(t)C_1(t)(P_1(t)S_1(t)+I)^{-1}\hat{\eta}(t)+P_2(t)B_2(t)\hat{v}_2(t),\\
 \hat{\beta}(t)&=\big[C_1(t)^\top-(S_1(t)-P_2(t))(P_1(t)S_1(t)+I)^{-1}P_1(t)C_1(t)^\top\big]\hat{x}(t)\\
               &\quad-(S_1(t)-P_2(t))(P_1(t)S_1(t)+I)^{-1}\hat{\eta}(t).
\end{aligned}
\end{equation}
If the following Riccati equation
\begin{equation}\label{P2}
\left\{
\begin{aligned}
&\dot{P}_2(t)-A(t)^{\top}P_2(t)-P_2(t)A(t)+P_2(t)B_1(t)R_1^{-1}(t)B_1(t)^\top P_2(t)\\
&+P_2(t)C_1(t)(P_1(t)S_1(t)+I)^{-1}P_1(t)C_1(t)^\top P_2(t)\\
&+P_2(t)C_2(t)(P_1(t)S_2(t)+I)^{-1}P_1(t)C_2(t)^\top P_2(t)-Q_1(t)=0,\ t\in[0,T],\\
&P_2(0)=G_1,
\end{aligned}
\right.
\end{equation}
admits a unique differentiable solution $P_2(\cdot)$, then
\begin{equation}\label{hatvarphi}
\begin{aligned}
d\hat{\varphi}(t)&=\big\{\big[A(t)^\top-P_2(t)B_1(t)R_1^{-1}(t)B_1(t)^\top-P_2(t)C_1(t)(P_1(t)S_1(t)+I)^{-1}P_1(t)C_1(t)^\top\\
                 &\qquad-P_2(t)C_2(t)(P_1(t)S_2(t)+I)^{-1}P_1(t)C_2(t)^\top]\hat{\varphi}(t)\\
                 &\qquad-P_2(t)C_1(t)(P_1(t)S_1(t)+I)^{-1}\hat{\eta}(t)+P_2(t)B_2(t)\hat{v}_2(t)\big\}dt\\
                 &\quad+\big\{\big[C_1(t)^\top-(S_1(t)-P_2(t))(P_1(t)S_1(t)+I)^{-1}P_1(t)C_1(t)^\top\big]\hat{x}(t)\\
                 &\qquad-(S_1(t)-P_2(t))(P_1(t)S_1(t)+I)^{-1}\hat{\eta}(t)\big\}dW(t).
\end{aligned}
\end{equation}

From the above relationship between $\hat{x}(t)$ and $\hat{y}^{\bar{v}_1,\hat{v}_2}$:
\begin{equation}\label{tworelation}
\left\{
\begin{aligned}
&\hat{y}^{\bar{v}_1,\hat{v}_2}(t)=-P_1(t)\hat{x}(t)-\hat{\phi}(t),\\
&\hat{x}(t)=P_2(t)\hat{y}^{\bar{v}_1,\hat{v}_2}(t)+\hat{\varphi}(t),
\end{aligned}
\right.
\end{equation}
we get
\begin{equation}\label{hatx}
\hat{x}(t)=-(I+P_2(t)P_1(t))^{-1}P_2(t)\hat{\phi}(t)+(I+P_2(t)P_1(t))^{-1}\hat{\varphi}(t).
\end{equation}
Putting \eqref{hatx} into \eqref{hatvarphi}, after dealing with the $dW(t)$ term and using the fact that
\begin{equation*}
(I+AB)^{-1}A=A(I+BA)^{-1},\mbox{ for any }A, B\in\mathcal{S}^n,
\end{equation*}
we derive
\begin{equation}\label{hatvarphi2}
\left\{
\begin{aligned}
d\hat{\varphi}(t)&=\big\{\big[A(t)^\top-P_2(t)B_1(t)R_1^{-1}(t)B_1(t)^\top-P_2(t)C_1(t)(P_1(t)S_1(t)+I)^{-1}\\
                 &\qquad\times P_1(t)C_1(t)^\top-P_2(t)C_2(t)(P_1(t)S_2(t)+I)^{-1}P_1(t)C_2(t)^\top\big]\hat{\varphi}(t)\\
                 &\qquad-P_2(t)C_1(t)(P_1(t)S_1(t)+I)^{-1}\hat{\eta}(t)+P_2(t)B_2(t)\hat{v}_2(t)\big\}dt\\
                 &\quad+\big\{-(I+P_2(t)P_1(t))(S_1(t)P_1(t)+I)^{-1}C_1(t)^{\top}(I+P_2(t)P_1(t))^{-1}P_2(t)\hat{\phi}(t)\\
                 &\qquad+(I+P_2(t)P_1(t))(S_1(t)P_1(t)+I)^{-1}C_1(t)^{\top}(I+P_2(t)P_1(t))^{-1}\hat{\varphi}(t)\\
                 &\qquad-(S_1(t)-P_2(t))(P_1(t)S_1(t)+I)^{-1}\hat{\eta}(t)\big\}dW(t),\ t\in[0,T],\\
\hat{\varphi}(0)&=0.
\end{aligned}
\right.
\end{equation}

\begin{Remark}
{\rm We introduce two Riccati equations for $P_1(\cdot)$ and $P_2(\cdot)$ to build the relation between $y^{\bar{v}_1,v_2}(\cdot)$ and $x(\cdot)$. Similarly to Lim and Zhou \cite{LZ01}, we can obtain the unique solvability of these two Riccati equations. And the unique solvability of \eqref{hatphi} and \eqref{hatvarphi2}, with the solutions $(\hat{\phi}(\cdot),\hat{\eta}(\cdot))$ and $\hat{\varphi}(\cdot)$ respectively, is evident as they can be regarded as two kinds of linear BSDE and SDE with bounded deterministic coefficients and square integrable nonhomogeneous terms.}\
\end{Remark}

We have the following result.
\begin{theorem}
Under the assumptions {\bf(L1)} and {\bf(L2)}, for any given $\xi\in L_{\mathcal{F}_T}^2(\Omega;R^n)$ and $v_2(\cdot)\in\mathcal{U}_2[0,T]$, the follower's problem is solvable with the optimal strategy $\bar{v}_1(\cdot)$ being of a state estimate feedback representation
\begin{equation}\label{barv_1}
\bar{v}_1(t)=-R_1^{-1}(t)B_1(t)^\top\big[P_2(t)\hat{y}^{\bar{v}_1,\hat{v}_2}(t)+\hat{\varphi}(t)\big],\ a.e.\ t\in[0,T],\ a.s.,
\end{equation}
where $P_2(\cdot)$ and $\hat{\varphi}(\cdot)$ are the solutions to \eqref{P2} and \eqref{hatvarphi2}, respectively. The optimal state trajectory $(\hat{y}^{\bar{v}_1,\hat{v}_2}(\cdot),\hat{z}^{\bar{v}_1,\hat{v}_2}(\cdot),\hat{\tilde{z}}^{\bar{v}_1,\hat{v}_2}(\cdot))$ is the unique solution to the following forward-backward stochastic differential filtering equation (FBSDFE):
\begin{equation}\label{fbsde3}
\left\{
\begin{aligned}
-d\hat{y}^{\bar{v}_1,\hat{v}_2}(t)&=\big[A(t)\hat{y}^{\bar{v}_1,\hat{v}_2}(t)-B_1(t)R_1^{-1}(t)B_1(t)^{\top}P_2(t)\hat{y}^{\bar{v}_1,\hat{v}_2}(t)\\
                                  &\qquad-B_1(t)R_1^{-1}(t)B_1(t)^{\top}\hat{\varphi}(t)+B_2(t)\hat{v}_2(t)\\
                                  &\qquad+C_1(t)\hat{z}^{\bar{v}_1,\hat{v}_2}(t)+C_2(t)\hat{\tilde{z}}^{\bar{v}_1,\hat{v}_2}(t)\big]dt-\hat{z}^{\bar{v}_1,\hat{v}_2}(t)dW(t),\\
                       d\hat{x}(t)&=\big[A(t)^{\top}\hat{x}(t)+Q_1(t)\hat{y}^{\bar{v}_1,\hat{v}_2}(t)\big]dt\\
                                  &\quad+\big[C_1(t)^{\top}\hat{x}(t)+S_1(t)\hat{z}^{\bar{v}_1,\hat{v}_2}(t)\big]dW(t),\ t\in[0,T],\\
                        \hat{x}(0)&=G_1\hat{y}^{\bar{v}_1,\hat{v}_2}(0),\ \hat{y}^{\bar{v}_1,\hat{v}_2}(T)=\hat{\xi}.
\end{aligned}
\right.
\end{equation}
\end{theorem}
\begin{proof}
For given $\xi$ and $v_2(\cdot)$, let $P_1(\cdot)$ satisfy \eqref{P_1}, by the standard BSDE theory we can solve \eqref{hatphi} to obtain $(\hat{\phi}(\cdot),\hat{\eta}(\cdot))$. Let $P_2(\cdot)$ satisfy \eqref{P2}, and by the standard SDE theory we can solve \eqref{hatvarphi2} to obatin $\hat{\varphi}(\cdot)$. Then according to the \eqref{hatz and hatbarz}, \eqref{hatx} and \eqref{filter y x}, we can obtain $(\hat{y}^{\bar{v}_1,\hat{v}_2}(\cdot),\hat{z}^{\bar{v}_1,\hat{v}_2}(\cdot),\hat{\tilde{z}}^{\bar{v}_1,\hat{v}_2}(\cdot))$, which are the $\mathcal{G}^1_t$-adapted solution to \eqref{fbsde3}. Therefore, the state estimate feedback representation \eqref{barv_1} can be obtained. The proof is complete.
\end{proof}

\subsection{Optimization for The Leader}

By \eqref{LQbsde} and \eqref{barv_1}, we have
\begin{equation}\label{ybarv1}
\left\{
\begin{aligned}
-dy^{\bar{v}_1,v_2}(t)&=\big[A(t)y^{\bar{v}_1,v_2}(t)-B_1(t)R_1^{-1}B_1(t)^\top P_2(t)\hat{y}^{\bar{v}_1,\hat{v}_2}(t)-B_1(t)R_1^{-1}B_1(t)^\top\hat{\varphi}(t)\\
                      &\qquad+B_2(t)v_2(t)+C_1(t)z^{\bar{v}_1,v_2}(t)+C_2(t)\tilde{z}^{\bar{v}_1,v_2}(t)\big]dt\\
                      &\quad-z^{\bar{v}_1,v_2}(t)dW(t)-\tilde{z}^{\bar{v}_1,v_2}(t)d\widetilde{W}(t),\ t\in[0,T],\\
  y^{\bar{v}_1,v_2}(T)&=\xi.
\end{aligned}
\right.
\end{equation}
From the relationship \eqref{tworelation} and \eqref{hatz and hatbarz}, we get
\begin{equation}\label{hatphi1 and hateta1}
\left\{
\begin{aligned}
\hat{\phi}(t)&=-(P_1(t)P_2(t)+I)\hat{y}^{\bar{v}_1,\hat{v}_2}(t)-P_1(t)\hat{\varphi}(t),\\
\hat{\eta}(t)&=-P_1(t)C_1(t)^\top P_2(t)\hat{y}^{\bar{v}_1,\hat{v}_2}(t)-P_1(t)C_1(t)^\top\hat{\varphi}(t)-(P_1(t)S_1(t)+I)\hat{z}^{\bar{v}_1,\hat{v}_2}(t).
\end{aligned}
\right.
\end{equation}
Substituting \eqref{hatphi1 and hateta1} into \eqref{hatvarphi2}, we get
\begin{equation}\label{hatvarphi3}
\left\{
\begin{aligned}
d\hat{\varphi}(t)&=\big\{\big[A(t)^\top-P_2(t)B_1(t)R_1^{-1}B_1(t)^\top-P_2(t)C_2(t)(P_1(t)S_2(t)+I)^{-1}P_1(t)\\
                 &\qquad\times C_2(t)^\top\big]\hat{\varphi}(t)+P_2(t)C_1(t)(P_1(t)S_1(t)+I)^{-1}P_1(t)C_1(t)^\top P_2(t)\hat{y}^{\bar{v}_1,\hat{v}_2}(t)\\
                 &\qquad+P_2(t)C_1(t)\hat{z}^{\bar{v}_1,\hat{v}_2}(t)+P_2(t)B_2(t)\hat{v}_2(t)\big\}dt+\big[C_1(t)^{\top}P_2(t)\hat{y}^{\bar{v}_1,\hat{v}_2}(t)\\
                 &\qquad+C_1(t)^{\top}\hat{\varphi}(t)+(S_1(t)-P_2(t))\hat{z}^{\bar{v}_1,\hat{v}_2}(t)\big]dW(t),\ t\in[0,T],\\
 \hat{\varphi}(0)&=0.
\end{aligned}
\right.
\end{equation}
Combining \eqref{ybarv1} and \eqref{hatvarphi3}, we can get the state equation of the leader:
\begin{equation}\label{leader state eq}
\left\{
\begin{aligned}
     d\hat{\varphi}(t)&=\big\{\big[A(t)^\top-P_2(t)B_1(t)R_1^{-1}B_1(t)^\top-P_2(t)C_2(t)(P_1(t)S_2(t)+I)^{-1}P_1(t)\\
                      &\qquad\times C_2(t)^\top\big]\hat{\varphi}(t)+P_2(t)C_1(t)(P_1(t)S_1(t)+I)^{-1}P_1(t)C_1(t)^\top P_2(t)\hat{y}^{\bar{v}_1,\hat{v}_2}(t)\\
                      &\qquad+P_2(t)C_1(t)\hat{z}^{\bar{v}_1,\hat{v}_2}(t)+P_2(t)B_2(t)\hat{v}_2(t)\big\}dt+\big[C_1(t)^{\top}P_2(t)\hat{y}^{\bar{v}_1,\hat{v}_2}(t)\\
                      &\qquad+C_1(t)^\top\hat{\varphi}(t)+(S_1(t)-P_2(t))\hat{z}^{\bar{v}_1,\hat{v}_2}(t)\big]dW(t),\\
-dy^{\bar{v}_1,v_2}(t)&=\big[A(t)y^{\bar{v}_1,v_2}(t)-B_1(t)R_1^{-1}B_1(t)^\top P_2(t)\hat{y}^{\bar{v}_1,\hat{v}_2}(t)-B_1(t)R_1^{-1}B_1(t)^\top\hat{\varphi}(t)\\
                      &\qquad+B_2(t)v_2(t)+C_1(t)z^{\bar{v}_1,v_2}(t)+C_2(t)\tilde{z}^{\bar{v}_1,v_2}(t)\big]dt\\
                      &\quad-z^{\bar{v}_1,v_2}(t)dW(t)-\tilde{z}^{\bar{v}_1,v_2}(t)d\widetilde{W}(t),\ t\in[0,T],\\
  y^{\bar{v}_1,v_2}(T)&=\xi,\ \hat{\varphi}(0)=0.
\end{aligned}
\right.
\end{equation}
For any given $\xi$ and $v_2(\cdot)$, from the proof above, the solvability for the solution $(y^{\bar{v}_1,v_2}(\cdot),z^{\bar{v}_1,v_2}(\cdot),\\\tilde{z}^{\bar{v}_1,v_2}(\cdot),\hat{\varphi}(\cdot))$ to \eqref{leader state eq} can be guaranteed though it is fully coupled.

The leader would like to choose an $\mathcal{F}_t$-adapted optimal control $\bar{v}_2(\cdot)$ such that the cost functional
\begin{equation}\label{LQ cost functional2}
\begin{aligned}
&J_2(\bar{v}_1(\cdot),v_2(\cdot);\xi)=\frac{1}{2}\mathbb{E}\bigg\{\int_0^T\Big[\langle Q_2(t)y^{\bar{v}_1,v_2}(t),y^{\bar{v}_1,v_2}(t)\rangle+\langle R_2(t)v_2(t),v_2(t)\rangle\\
&\ +\langle N_1(t)z^{\bar{v}_1,v_2}(t),z^{\bar{v}_1,v_2}(t)\rangle+\langle N_2(t)\tilde{z}^{\bar{v}_1,v_2}(t),\tilde{z}^{\bar{v}_1,v_2}(t)\rangle\Big]dt+\langle G_2(t)y^{\bar{v}_1,v_2}(0),y^{\bar{v}_1,v_2}(0)\rangle\bigg\}
\end{aligned}
\end{equation}
is minimized. We suppose the following holds.
\begin{equation}
{\bf(L3)}\left\{
\begin{aligned}
&Q_2(\cdot), N_1(\cdot),N_2(\cdot)\in L^{\infty}(0,T;\mathcal{S}^n),\ Q_2(\cdot), N_1(\cdot),N_2(\cdot)\geq0,\\
&R_2(\cdot)\in L^{\infty}(0,T;\mathcal{S}^{k_2}),\ R_2(\cdot)>0,\ G_2\in\mathcal{S}^n,G_2\geq0.
\end{aligned}
\right.
\end{equation}
Define the Hamiltonian function of the leader as
\begin{equation}\label{Hamiltonian2}
\begin{aligned}
&H_2(t,y^{\bar{v}_1,v_2},z^{\bar{v}_1,v_2},\tilde{z}^{\bar{v}_1,v_2},v_2,\hat{\varphi},p,q_1,q_2,Q)\\
&=\big\langle\big[A(t)^\top-P_2(t)B_1(t)R_1^{-1}B_1(t)^\top-P_2(t)C_2(t)(P_1(t)S_2(t)+I)^{-1}P_1(t)C_2(t)^\top\big]\hat{\varphi}\\
&\qquad+P_2(t)C_1(t)(P_1(t)S_1(t)+I)^{-1}P_1(t)C_1(t)^\top P_2(t)\hat{y}^{\bar{v}_1,\hat{v}_2}+P_2(t)C_1(t)\hat{z}^{\bar{v}_1,\hat{v}_2}\\
&\qquad+P_2(t)B_2(t)\hat{v}_2,p(t)\big\rangle+\big\langle C_1(t)^\top P_2(t)\hat{y}^{\bar{v}_1,\hat{v}_2}+C_1(t)^\top\hat{\varphi}+(S_1(t)-P_2(t))\hat{z}^{\bar{v}_1,\hat{v}_2},q_1\big\rangle\\
&\quad+\big\langle A(t)y^{\bar{v}_1,v_2}-B_1(t)R_1^{-1}B_1(t)^\top P_2(t)\hat{y}^{\bar{v}_1,\hat{v}_2}-B_1(t)R_1^{-1}B_1(t)^\top\hat{\varphi}+B_2(t)v_2\\
&\qquad+C_1(t)z^{\bar{v}_1,v_2}+C_2(t)\tilde{z}^{\bar{v}_1,v_2},Q\big\rangle+\frac{1}{2}\Big[\big\langle Q_2(t)y^{\bar{v}_1,v_2},y^{\bar{v}_1,v_2}\big\rangle+\big\langle R_2(t)v_2,v_2\big\rangle\\
&\qquad+\big\langle N_1(t)z^{\bar{v}_1,v_2},z^{\bar{v}_1,v_2}\big\rangle+\big\langle N_2(t)\tilde{z}^{\bar{v}_1,v_2},\tilde{z}^{\bar{v}_1,v_2}\big\rangle\Big].
\end{aligned}
\end{equation}
Noting that the variable $q_2$ does not appear explicitly. Suppose that there exists an $\mathcal{F}_t$-adapted optimal control $\bar{v}_2(\cdot)$ of the leader, and the corresponding optimal state trajectory is $(y^{\bar{v}_1,\bar{v}_2}(\cdot),z^{\bar{v}_1,\bar{v}_2}(\cdot),\tilde{z}^{\bar{v}_1,\bar{v}_2}(\cdot),\hat{\bar{\varphi}}(\cdot))\equiv(y^{\bar{v}_2}(\cdot),z^{\bar{v}_2}(\cdot),\tilde{z}^{\bar{v}_2}(\cdot),
\hat{\bar{\varphi}}(\cdot))$, then by Theorem 3.3, we obtain
\begin{equation}\label{barv_2-initial}
B_2(t)^{\top}Q(t)+R_2(t)\bar{v}_2(t)+B_2(t)^{\top}P_2(t)\hat{p}(t)=0,\ a.e.\ t\in[0,T],\ a.s.,
\end{equation}
where the $\mathcal{F}_t$-adapted process triple $(p(\cdot),q_1(\cdot),Q(\cdot))$ satisfies
\begin{equation}\label{leader adjoint eq}
\left\{
\begin{aligned}
 dQ(t)&=\big\{\big[P_2(t)C_1(t)(P_1(t)S_1(t)+I)^{-1}P_1(t)C_1(t)^{\top}P_2(t)\big]^\top\hat{p}(t)+P_2(t)C_1(t)\hat{q}_1(t)\\
      &\qquad+A(t)^\top Q(t)-P_2(t)B_1(t)R_1^{-1}(t)B_1(t)^\top\hat{Q}(t)+Q_2(t)^\top y^{\bar{v}_2}(t)\big\}dt\\
      &\quad+\big[C_1(t)^\top P_2(t)\hat{p}(t)+(S_1(t)-P_2(t))^\top\hat{q}_1(t)+C_1(t)^\top Q(t)\\
      &\qquad+N_1(t)z^{\bar{v}_1,\bar{v}_2}(t)\big]dW(t)+\big[C_2(t)^\top Q(t)+N_2(t)\tilde{z}^{\bar{v}_1,\bar{v}_2}(t)\big]d\widetilde{W}(t),\\
-dp(t)&=\big\{\big[A(t)^\top-P_2(t)B_1(t)R_1^{-1}(t)B_1(t)^\top-P_2(t)C_2(t)(P_1(t)S_2(t)+I)^{-1}\\
      &\qquad\times P_1(t)C_2(t)^\top\big]^\top p(t)+C_1(t)q_1(t)-B_1(t)R_1^{-1}(t)B_1(t)^\top Q(t)\big\}dt\\
      &\quad-q_1(t)dW(t),\ t\in[0,T],\\
  Q(0)&=G_2(t)y^{\bar{v}_2}(0)+G_1p(0),\ p(T)=0.
\end{aligned}
\right.
\end{equation}

Then, to make the problem clear, let us put \eqref{leader state eq}, \eqref{leader adjoint eq} and \eqref{barv_2-initial} together:
\begin{equation}\label{fbsdes}
\left\{
\begin{aligned}
d\hat{\bar{\varphi}}(t)&=\big\{\big[A(t)^\top-P_2(t)B_1(t)R_1^{-1}B_1(t)^\top-P_2(t)C_2(t)(P_1(t)S_2(t)+I)^{-1}P_1(t)\\
                       &\qquad\times C_2(t)^\top\big]\hat{\bar{\varphi}}(t)+P_2(t)C_1(t)(P_1(t)S_1(t)+I)^{-1}P_1(t)C_1(t)^\top P_2(t)\hat{y}^{\hat{\bar{v}}_2}(t)\\
                       &\qquad+P_2(t)C_1(t)\hat{z}^{\hat{\bar{v}}_2}(t)+P_2(t)B_2(t)\hat{\bar{v}}_2(t)\big\}dt\\
                       &\quad+\big[C_1(t)^{\top}P_2(t)\hat{y}^{\hat{\bar{v}}_2}(t)+C_1(t)^\top\hat{\bar{\varphi}}(t)+(S_1(t)-P_2(t))\hat{z}^{\hat{\bar{v}}_2}(t)\big]dW(t),\\
                  dQ(t)&=\big\{\big[P_2(t)C_1(t)(P_1(t)S_1(t)+I)^{-1}P_1(t)C_1(t)^{\top}P_2(t)\big]^\top\hat{p}(t)+P_2(t)C_1(t)\hat{q}_1(t)\\
                       &\qquad+A(t)^\top Q(t)-P_2(t)B_1(t)R_1^{-1}(t)B_1(t)^\top\hat{Q}(t)+Q_2(t)^\top y^{\bar{v}_2}(t)\big\}dt\\
                       &\quad+\big[C_1(t)^\top P_2(t)\hat{p}(t)+(S_1(t)-P_2(t))^\top\hat{q}_1(t)+C_1(t)^\top Q(t)\\
                       &\qquad+N_1(t)z^{\bar{v}_2}(t)\big]dW(t)+\big[C_2(t)^\top Q(t)+N_2(t)\tilde{z}^{\bar{v}_2}(t)\big]d\widetilde{W}(t),\\
                 -dp(t)&=\big\{\big[A(t)^\top-P_2(t)B_1(t)R_1^{-1}(t)B_1(t)^\top-P_2(t)C_2(t)(P_1(t)S_2(t)+I)^{-1}\\
                       &\qquad\times P_1(t)C_2(t)^\top]^\top p(t)+C_1(t)q_1(t)-B_1(t)R_1^{-1}(t)B_1(t)^\top Q(t)\big\}dt\\
                       &\quad-q_1(t)dW(t),\\
     -dy^{\bar{v}_2}(t)&=\big[A(t)y^{\bar{v}_2}(t)-B_1(t)R_1^{-1}B_1(t)^\top P_2(t)\hat{y}^{\hat{\bar{v}}_2}(t)-B_1(t)R_1^{-1}B_1(t)^\top\hat{\bar{\varphi}}(t)\\
                       &\qquad+B_2(t)\bar{v}_2(t)+C_1(t)z^{\bar{v}_2}(t)+C_2(t)\tilde{z}^{\bar{v}_2}(t)\big]dt\\
                       &\quad-z^{\bar{v}_2}(t)dW(t)-\tilde{z}^{\bar{v}_2}(t)d\widetilde{W}(t),\ t\in[0,T],\\
 \hat{\bar{\varphi}}(0)&=0,\ Q(0)=G_1p(0)+G_2y^{\bar{v}_2}(0),\ p(T)=0,\ y^{\bar{v}_2}(T)=\xi,\\
                       &\hspace{-1.3cm}B_2(t)^\top Q(t)+R_2(t)\bar{v}_2(t)+B_2(t)^\top P_2(t)\hat{p}(t)=0,\ a.e.\ t\in[0,T],\ a.s.
\end{aligned}
\right.
\end{equation}
We may look at the above equations in a different way. To this end, let us set (The time variable $t$ is omitted.)
\begin{equation}\label{notation}
X=
\begin{pmatrix}
\hat{\bar{\varphi}}\\
Q
\end{pmatrix},\ \
Y=
\begin{pmatrix}
p\\
y^{\bar{v}_2}
\end{pmatrix},\ \
Z=
\begin{pmatrix}
q_1\\
z^{\bar{v}_2}
\end{pmatrix},\ \
\widetilde{Z}=
\begin{pmatrix}
0\\
\tilde{z}^{\bar{v}_2}
\end{pmatrix},\ \
\end{equation}
and
\begin{equation*}
\left\{
\begin{aligned}
&\mathcal{A}_1=
\begin{pmatrix}
A^{\top}-P_2B_1R_1^{-1}B_1^\top-P_2C_2(P_1S_2+I)^{-1}P_1C_2^\top&0\\
0&A^\top
\end{pmatrix},\ \
\widetilde{\mathcal{A}}_1=
\begin{pmatrix}
C_1^\top&0\\
0&C_1^\top
\end{pmatrix},\\
&\mathcal{A}_2=
\begin{pmatrix}
0&0\\
0&-P_2B_1R_1^{-1}B_1^\top
\end{pmatrix},\ \
\widetilde{\mathcal{A}}_2=
\begin{pmatrix}
0&0\\
0&C_2^\top
\end{pmatrix},\ \
\mathcal{B}_1=
\begin{pmatrix}
0&0\\
0&Q_2
\end{pmatrix},\ \
\widetilde{\mathcal{B}}_2=
\begin{pmatrix}
0\\
B_2
\end{pmatrix},\\
&\mathcal{B}_2=
\begin{pmatrix}
0&P_2C_1(P_1S_1+I)^{-1}P_1C_1^\top P_2\\
\big[P_2C_1(P_1S_1+I)^{-1}P_1C_1^\top P_2\big]^\top&0
\end{pmatrix},\ \
\widetilde{\mathcal{C}}_1=
\begin{pmatrix}
0&0\\
0&N_1
\end{pmatrix},\\
&\widetilde{\mathcal{C}}_2=
\begin{pmatrix}
0&S_1-P_2\\
(S_1-P_2)^\top&0
\end{pmatrix},\ \
\widetilde{\mathcal{C}}_3=
\begin{pmatrix}
0&0\\
0&N_2
\end{pmatrix},\ \
\mathcal{F}_1=
\begin{pmatrix}
0&-B_1R_1^{-1}B_1^\top\\
-B_1R_1^{-1}B_1^\top&0
\end{pmatrix},\\
&\mathcal{C}=
\begin{pmatrix}
0&P_2C_1\\
P_2C_1&0
\end{pmatrix},\ \
\mathcal{D}=
\begin{pmatrix}
P_2B_2\\
0
\end{pmatrix},\ \
\bar{\xi}=
\begin{pmatrix}
0\\
\xi
\end{pmatrix},\ \
\bar{G}=
\begin{pmatrix}
0&0\\
G_1&G_2
\end{pmatrix}.
\end{aligned}
\right.
\end{equation*}
With this, \eqref{fbsdes} is equivalent to the FBSDE:
\begin{equation}\label{2fbsdes}
\left\{
\begin{aligned}
 dX(t)&=\big[\mathcal{A}_1X(t)+\mathcal{A}_2\hat{X}(t)+\mathcal{B}_1Y(t)+\mathcal{B}_2\hat{Y}(t)+\mathcal{C}\hat{Z}(t)+\mathcal{D}\hat{\bar{v}}_2(t)\big]dt\\
      &\quad+\big[\widetilde{\mathcal{A}}_1X(t)+\mathcal{C}^\top\hat{Y}(t)+\widetilde{\mathcal{C}}_1Z(t)+\widetilde{\mathcal{C}}_2\hat{Z}(t)\big]dW(t)
       +\big[\widetilde{\mathcal{A}}_2X(t)+\widetilde{\mathcal{C}}_3\widetilde{Z}(t)\big]d\widetilde{W}(t),\\
-dY(t)&=\big[\mathcal{F}_1X(t)+\mathcal{A}_1^\top Y(t)+\mathcal{A}_2^\top\hat{Y}(t)+\widetilde{\mathcal{A}}_1^\top Z(t)+\widetilde{\mathcal{A}}_2^\top\widetilde{Z}(t)+\widetilde{\mathcal{B}}_2\bar{v}_2(t)\big]dt\\
      &\quad-Z(t)dW(t)-\widetilde{Z}(t)d\widetilde{W}(t),\ t\in[0,T],\\
  X(0)&=\bar{G}Y(0),\ Y(T)=\bar{\xi},\\
\widetilde{\mathcal{B}}_2^\top X(t)&+R_2(t)\bar{v}_2(t)+\mathcal{D}^\top\hat{Y}(t)=0,\ a.e.\ t\in[0,T],\ a.s.
\end{aligned}
\right.
\end{equation}
Then we have
\begin{equation}\label{barv_2}
\bar{v}_2(t)=-R_2^{-1}(t)\big[\widetilde{\mathcal{B}}_2^\top X(t)+\mathcal{D}^\top\hat{Y}(t)\big],\ a.e.\ t\in[0,T],\ a.s.
\end{equation}

Taking $\mathbb{E}[\cdot|\mathcal{G}_t^1]$ on \eqref{barv_2}, we get
\begin{equation}\label{hatbarv_2}
\hat{\bar{v}}_2(t)=-R_2^{-1}(t)\big[\widetilde{\mathcal{B}}_2^\top \hat{X}(t)+\mathcal{D}^\top\hat{Y}(t)\big],\ a.e.\ t\in[0,T],\ a.s.
\end{equation}
Putting \eqref{barv_2} and \eqref{hatbarv_2} into \eqref{2fbsdes}, we get
\begin{equation}\label{2fbsdes2}
\left\{
\begin{aligned}
 dX(t)&=\big[\mathcal{A}_1X(t)+(\mathcal{A}_2-\mathcal{D}R_2^{-1}(t)\widetilde{\mathcal{B}}_2^\top)\hat{X}(t)+\mathcal{B}_1Y(t)+(\mathcal{B}_2-\mathcal{D}R_2^{-1}(t)\mathcal{D}^\top)\hat{Y}(t)\\
      &\qquad+\mathcal{C}\hat{Z}(t)\big]dt+\big[\widetilde{\mathcal{A}}_1X(t)+\mathcal{C}^{\top}\hat{Y}(t)+\widetilde{\mathcal{C}}_1Z(t)+\widetilde{\mathcal{C}}_2\hat{Z}(t)\big]dW(t)\\
      &\quad+\big[\widetilde{\mathcal{A}}_2X(t)+\widetilde{\mathcal{C}}_3\widetilde{Z}(t)\big]d\widetilde{W}(t),\\
-dY(t)&=\big[(\mathcal{F}_1-\widetilde{\mathcal{B}}_2R_2^{-1}(t)\widetilde{\mathcal{B}}_2^\top)X(t)+\mathcal{A}_1^\top Y(t)+(\mathcal{A}_2^\top-\widetilde{\mathcal{B}}_2R_2^{-1}(t)\mathcal{D}^\top)\hat{Y}(t)\\
      &\qquad+\widetilde{\mathcal{A}}_1^\top Z(t)+\widetilde{\mathcal{A}}_2^\top\widetilde{Z}(t)\big]dt-Z(t)dW(t)-\widetilde{Z}(t)d\widetilde{W}(t),\ t\in[0,T],\\
  X(0)&=\bar{G}Y(0),\ Y(T)=\bar{\xi}.
\end{aligned}
\right.
\end{equation}
Noting that \eqref{2fbsdes2} is a coupled FBSDE, we need to decouple it with the similar method before.

We set
\begin{equation}\label{Relation1}
Y(t)=\Pi_1(t)X(t)+\Pi_2(t)\hat{X}(t)+\widetilde{\phi}(t),
\end{equation}
where $\Pi_1(\cdot),\Pi_2(\cdot)$ are $\mathbb{R}^{2n\times 2n}$-matrix-valued deterministic, differentiable functions with $\Pi_1(T)=\Pi_2(T)=0$, and $(\widetilde{\phi}(\cdot),\widetilde{\eta}(\cdot))$ are both $\mathbb{R}^{2n}$-valued, $\mathcal{F}_t$-adapted processes which satisfy the following BSDE:
\begin{equation}\label{widephi}
\left\{
\begin{aligned}
-d\widetilde{\phi}(t)&=\widetilde{\alpha}(t)dt-\widetilde{\eta}(t)dW(t),\ t\in[0,T],\\
  \widetilde{\phi}(T)&=\bar{\xi},
\end{aligned}
\right.
\end{equation}
where $\widetilde{\alpha}(\cdot)$ is an $\mathbb{R}^{2n}$-valued, $\mathcal{F}_t$-adapted process which is to be determined later.

Applying Lemma 5.4 in Xiong \cite{Xiong08} to the forward equation in \eqref{2fbsdes2}, we get
\begin{equation}\label{hatX}
\begin{aligned}
d\hat{X}(t)&=\big[\big(\mathcal{A}_1+\mathcal{A}_2-\mathcal{D}R_2^{-1}(t)\widetilde{\mathcal{B}}_2^\top\big)\hat{X}(t)+\big(\mathcal{B}_1+\mathcal{B}_2-\mathcal{D}R_2^{-1}(t)\mathcal{D}^\top\big)\hat{Y}(t)+\mathcal{C}\hat{Z}(t)\big]dt\\
           &\quad+\big[\widetilde{\mathcal{A}}_1\hat{X}(t)+\mathcal{C}^\top\hat{Y}(t)+(\widetilde{\mathcal{C}}_1+\widetilde{\mathcal{C}}_2)\hat{Z}(t)\big]dW(t).
\end{aligned}
\end{equation}
Applying It\^{o}'s formula to \eqref{Relation1}, noting \eqref{hatX}, we have
\begin{equation}\label{afterito}
\begin{aligned}
dY(t)&=\big[(\dot{\Pi}_1(t)+\Pi_1(t)\mathcal{A}_1)X(t)+\Pi_1(t)(\mathcal{A}_2-\mathcal{D}R_2^{-1}(t)\widetilde{\mathcal{B}}_2^\top)\hat{X}(t)+\Pi_1(t)\mathcal{B}_1Y(t)\\
     &\qquad+\Pi_1(t)(\mathcal{B}_2-\mathcal{D}R_2^{-1}(t)\mathcal{D}^\top)\hat{Y}(t)+\Pi_1(t)\mathcal{C}\hat{Z}(t)+\dot{\Pi}_2(t)\hat{X}(t)\\
     &\qquad+\Pi_2(t)(\mathcal{A}_1+\mathcal{A}_2-\mathcal{D}R_2^{-1}(t)\widetilde{\mathcal{B}}_2^\top)\hat{X}(t)+\Pi_2(t)(\mathcal{B}_1+\mathcal{B}_2-\mathcal{D}R_2^{-1}(t)\mathcal{D}^\top)\hat{Y}(t)\\
     &\qquad+\Pi_2(t)\mathcal{C}_2\hat{Z}(t)-\widetilde{\alpha}(t)\big]dt+\big[\Pi_1(t)\widetilde{\mathcal{A}}_1X(t)+\Pi_1(t)\mathcal{C}^{\top}\hat{Y}(t)+\Pi_1(t)\widetilde{\mathcal{C}}_1Z(t)\\
     &\qquad+\Pi_1(t)\widetilde{\mathcal{C}}_2\hat{Z}(t)+\Pi_2(t)\widetilde{\mathcal{A}}_1\hat{X}(t)+\Pi_2(t)\mathcal{C}^\top\hat{Y}(t)+\Pi_2(t)(\widetilde{\mathcal{C}}_1+\widetilde{\mathcal{C}}_2)\hat{Z}(t)\\
     &\qquad+\widetilde{\eta}(t)\big]dW(t)+\big[\Pi_1(t)\widetilde{\mathcal{A}}_2X(t)+\Pi_1(t)\widetilde{\mathcal{C}}_3\widetilde{Z}(t)\big]d\widetilde{W}(t)\\
     &=-\big[(\mathcal{F}_1-\widetilde{\mathcal{B}}_2R_2^{-1}(t)\widetilde{\mathcal{B}}_2^\top)X(t)+\mathcal{A}_1^\top Y(t)+(\mathcal{A}_2^\top-\widetilde{\mathcal{B}}_2R_2^{-1}(t)\mathcal{D}^\top)\hat{Y}(t)+\widetilde{\mathcal{A}}_1^\top Z(t)\\
     &\qquad+\widetilde{\mathcal{A}}_2^\top\widetilde{Z}(t)\big]dt+Z(t)dW(t)+\widetilde{Z}(t)d\widetilde{W}(t).
\end{aligned}
\end{equation}
Comparing the diffusion coefficients of $dW(\cdot)$ term and $d\widetilde{W}(\cdot)$ term on both sides of \eqref{afterito}, we derive
\begin{equation}\label{Z}
\begin{aligned}
Z(t)&=\Pi_1(t)\widetilde{\mathcal{A}}_1X(t)+\Pi_1(t)\mathcal{C}^{\top}\hat{Y}(t)+\Pi_1(t)\widetilde{\mathcal{C}}_1Z(t)+\Pi_1(t)\widetilde{\mathcal{C}}_2\hat{Z}(t)\\
    &\quad+\Pi_2(t)\widetilde{\mathcal{A}}_1\hat{X}(t)+\Pi_2(t)\mathcal{C}^{\top}\hat{Y}(t)+\Pi_2(t)(\widetilde{\mathcal{C}}_1+\widetilde{\mathcal{C}}_2)\hat{Z}(t)+\widetilde{\eta}(t),
\end{aligned}
\end{equation}
and
\begin{equation}\label{widetildeZ}
\begin{aligned}
\widetilde{Z}(t)=\Pi_1(t)\widetilde{\mathcal{A}}_2X(t)+\Pi_1(t)\widetilde{\mathcal{C}}_3\widetilde{Z}(t).
\end{aligned}
\end{equation}
Then comparing the drift coefficients of the $dt$ term on both sides of \eqref{afterito}, we have
\begin{equation}\label{dtterm}
\begin{aligned}
&(\dot{\Pi}_1(t)+\Pi_1(t)\mathcal{A}_1)X(t)+\Pi_1(t)(\mathcal{A}_2-\mathcal{D}R_2^{-1}(t)\widetilde{\mathcal{B}}_2^\top)\hat{X}(t)+\Pi_1(t)\mathcal{B}_1Y(t)\\
&+\Pi_1(t)(\mathcal{B}_2-\mathcal{D}R_2^{-1}(t)\mathcal{D}^\top)\hat{Y}(t)+\Pi_1(t)\mathcal{C}\hat{Z}(t)+\dot{\Pi}_2(t)\hat{X}(t)\\
&+\Pi_2(t)(\mathcal{A}_1+\mathcal{A}_2-\mathcal{D}R_2^{-1}(t)\widetilde{\mathcal{B}}_2^\top)\hat{X}(t)+\Pi_2(t)(\mathcal{B}_1+\mathcal{B}_2-\mathcal{D}R_2^{-1}\mathcal{D}^\top)\hat{Y}(t)\\
&+\Pi_2(t)\mathcal{C}\hat{Z}(t)+(\mathcal{F}_1-\widetilde{\mathcal{B}}_2R_2^{-1}(t)\widetilde{\mathcal{B}}_2^\top)X(t)+\mathcal{A}_1^\top Y(t)\\
&+(\mathcal{A}_2^\top-\widetilde{\mathcal{B}}_2R_2^{-1}(t)\mathcal{D}^\top)\hat{Y}(t)+\widetilde{\mathcal{A}}_1^\top Z(t)+\widetilde{\mathcal{A}}_2^\top\widetilde{Z}(t)=\widetilde{\alpha}(t).
\end{aligned}
\end{equation}
Taking $\mathbb{E}[\cdot|\mathcal{G}_t^1]$ on \eqref{Z}, we get
\begin{equation}\label{hatZ}
\hat{Z}(t)=\Sigma_1(\Pi_1(t)+\Pi_2(t))\widetilde{\mathcal{A}}_1\hat{X}(t)+\Sigma_1(\Pi_1(t)+\Pi_2(t))\mathcal{C}^\top\hat{Y}(t)+\Sigma_1\hat{\widetilde{\eta}}(t),
\end{equation}
where (The time variable $t$ is omitted.)
\begin{equation*}
\Sigma_1=[I-(\Pi_1(t)+\Pi_2(t))(\widetilde{\mathcal{C}}_1+\widetilde{\mathcal{C}}_2)]^{-1}.
\end{equation*}
Putting \eqref{hatZ} into \eqref{Z}, we get
\begin{equation}\label{Z2}
\begin{aligned}
Z(t)&=\Sigma_4\Pi_1(t)\widetilde{\mathcal{A}}_1X(t)+\big[\Sigma_4\Sigma_2+\Sigma_4\Sigma_3\Sigma_1(\Pi_1(t)+\Pi_2(t))\widetilde{\mathcal{A}}_1+\Sigma_4\Pi_2(t)\widetilde{\mathcal{A}}_1\\
    &\quad+\Sigma_4\Sigma_3\Sigma_1\Sigma_2\big]\hat{X}(t)+\big[\Sigma_4\Sigma_3\Sigma_1(\Pi_1(t)+\Pi_2(t))\mathcal{C}^\top+\Sigma_4(\Pi_1(t)+\Pi_2(t))\mathcal{C}^\top\big]\hat{\widetilde{\phi}}(t)\\
    &\quad+\Sigma_4\Sigma_3\Sigma_1\hat{\widetilde{\eta}}(t)+\Sigma_4\widetilde{\eta}(t),
\end{aligned}
\end{equation}
where
\begin{equation*}
\left\{
\begin{aligned}
&\Sigma_2=(\Pi_1(t)+\Pi_2(t))\mathcal{C}^\top(\Pi_1(t)+\Pi_2(t)),\ \Sigma_3=(\Pi_1(t)+\Pi_2(t))\widetilde{\mathcal{C}}_2+\Pi_2(t)\widetilde{\mathcal{C}}_1,\\
&\Sigma_4=(I-\Pi_1(t)\widetilde{\mathcal{C}}_1)^{-1}.
\end{aligned}
\right.
\end{equation*}
From \eqref{widetildeZ}, we can get
\begin{equation}\label{widetildeZ2}
\widetilde{Z}(t)=(I-\Pi_1(t)\widetilde{\mathcal{C}}_3)^{-1}\Pi_1(t)\widetilde{\mathcal{A}}_2X(t).
\end{equation}
Putting \eqref{hatZ}, \eqref{Z2} and \eqref{widetildeZ2} into \eqref{dtterm}, we drive
\begin{equation}
\begin{aligned}
\widetilde{\alpha}(t)=&\big[\dot{\Pi}_1(t)+\Pi_1(t)\mathcal{A}_1+\Pi_1(t)\mathcal{B}_1\Pi_1(t)+\Sigma_9+\mathcal{A}_1^\top\Pi_1(t)+\widetilde{\mathcal{A}}_1^\top\Sigma_4\Pi_1(t)\widetilde{\mathcal{A}}_1\\
&\quad+\widetilde{\mathcal{A}}_2^\top\Sigma_{10}\Pi_1(t)\widetilde{\mathcal{A}}_2\big]X(t)+\Big\{\Pi_1(t)\Sigma_5+\Pi_1(t)\mathcal{B}_1\Pi_2(t)+\Pi_1(t)\Sigma_6(\Pi_1(t)+\Pi_2(t))\\
&\quad+(\Pi_1(t)+\Pi_2(t))\mathcal{C}\Sigma_1(\Pi_1(t)+\Pi_2(t))\widetilde{\mathcal{A}}_1+(\Pi_1(t)+\Pi_2(t))\mathcal{C}\Sigma_1\Sigma_2\\
&\quad+\dot{\Pi}_2(t)+\Pi_2(t)\Sigma_7+\Pi_2\Sigma_8(\Pi_1(t)+\Pi_2(t))+\mathcal{A}_1^\top\Pi_2(t)+\Sigma_5^{\top}(\Pi_1(t)+\Pi_2(t))\\
&\quad+\widetilde{\mathcal{A}}_1^\top\big[\Sigma_4\Sigma_2+\Sigma_4\Sigma_3\Sigma_1(\Pi_1(t)+\Pi_2(t))\widetilde{\mathcal{A}}_1+\Sigma_4\Pi_2(t)\widetilde{\mathcal{A}}_1+\Sigma_4\Sigma_3\Sigma_1\Sigma_2]\Big\}\hat{X}(t)\\
&+\Big\{\Pi_1(t)\Sigma_6+(\Pi_1(t)+\Pi_2(t))\mathcal{C}\Sigma_1(\Pi_1(t)+\Pi_2(t))\mathcal{C}^\top+\Pi_2(t)\Sigma_8+\Sigma_5^\top\\
&\quad+\widetilde{\mathcal{A}}_1^\top\big[\Sigma_4\Sigma_3\Sigma_1(\Pi_1(t)+\Pi_2(t))\mathcal{C}^\top+\Sigma_4(\Pi_1(t)+\Pi_2(t))\mathcal{C}^\top\big]\Big\}\hat{\widetilde{\phi}}(t)\\
&+\big[(\Pi_1(t)+\Pi_2(t))\mathcal{C}\Sigma_1+\widetilde{\mathcal{A}}_1^\top\Sigma_4\Sigma_3\Sigma_1\big]\hat{\widetilde{\eta}}(t)
 +\big[\Pi_1(t)\mathcal{B}_1+\mathcal{A}_1^\top\big]\widetilde{\phi}(t)+\widetilde{\mathcal{A}}_1^\top\Sigma_4\widetilde{\eta}(t),
\end{aligned}
\end{equation}
where
\begin{equation*}
\left\{
\begin{aligned}
&\Sigma_5=\mathcal{A}_2-\mathcal{D}R_2^{-1}(t)\widetilde{\mathcal{B}}_2^\top,\quad\Sigma_6=\mathcal{B}_2-\mathcal{D}R_2^{-1}(t)\mathcal{D}^\top,\quad
\Sigma_7=\mathcal{A}_1+\mathcal{A}_2-\mathcal{D}R_2^{-1}(t)\widetilde{\mathcal{B}}_2^\top,\\
&\Sigma_8=\mathcal{B}_1+\mathcal{B}_2-\mathcal{D}R_2^{-1}(t)\mathcal{D}^\top,\quad\Sigma_9=\mathcal{F}_1-\widetilde{\mathcal{B}}_2R_2^{-1}(t)\widetilde{\mathcal{B}}_2^\top,\quad
\Sigma_{10}=(I-\Pi_1(t)\widetilde{\mathcal{C}}_3)^{-1}.
\end{aligned}
\right.
\end{equation*}
Then, if $\Pi_1(\cdot)$ and $\Pi_2(\cdot)$ satisfy the following two Riccati equations, one by one:
\begin{equation}\label{Pi1}
\left\{
\begin{aligned}
&\dot{\Pi}_1(t)+\Pi_1(t)\mathcal{A}_1+\mathcal{A}_1^\top\Pi_1(t)+\Pi_1(t)\mathcal{B}_1\Pi_1(t)+\widetilde{\mathcal{A}}_1^\top\Sigma_4\Pi_1(t)\widetilde{\mathcal{A}}_1\\
&\ +\widetilde{\mathcal{A}}_2^\top\Sigma_{10}\Pi_1(t)\widetilde{\mathcal{A}}_2+\Sigma_9=0,\ t\in[0,T],\ \Pi_1(T)=0,
\end{aligned}
\right.
\end{equation}
and
\begin{equation}\label{Pi2}
\left\{
\begin{aligned}
&\dot{\Pi}_2(t)+(\Pi_1(t)+\Pi_2(t))\Sigma_5+\Sigma_5^\top(\Pi_1(t)+\Pi_2(t))+\Pi_2(t)\mathcal{A}_1+\mathcal{A}_1^\top\Pi_2(t)\\
&+\Pi_1(t)\mathcal{B}_1\Pi_2(t)+\Pi_2(t)\mathcal{B}_1\Pi_1(t)+(\Pi_1(t)+\Pi_2(t))\Sigma_6(\Pi_1(t)+\Pi_2(t))+\Pi_2(t)\mathcal{B}_1\Pi_2(t)\\
&+(\Pi_1(t)+\Pi_2(t))\mathcal{C}\Sigma_1(\Pi_1(t)+\Pi_2(t))\widetilde{\mathcal{A}}_1+(\Pi_1(t)+\Pi_2(t))\mathcal{C}\Sigma_1\Sigma_2+\widetilde{\mathcal{A}}_1^\top\Sigma_4\Sigma_2\\
&+\widetilde{\mathcal{A}}_1^\top\Sigma_4\Sigma_3\Sigma_1(\Pi_1(t)+\Pi_2(t))\widetilde{\mathcal{A}}_1
+\widetilde{\mathcal{A}}_1^\top\Sigma_4\Pi_2(t)\widetilde{\mathcal{A}}_1+\widetilde{\mathcal{A}}_1^\top\Sigma_4\Sigma_3\Sigma_1\Sigma_2=0,\ t\in[0,T],\\
&\Pi_2(T)=0,
\end{aligned}
\right.
\end{equation}
we have
\begin{equation}\label{widetildealpha}
\begin{aligned}
\widetilde{\alpha}(t)=&\big[\Pi_1(t)\Sigma_6+(\Pi_1(t)+\Pi_2(t))\mathcal{C}\Sigma_1(\Pi_1(t)+\Pi_2(t))\mathcal{C}^\top+\Pi_2(t)\Sigma_8+\Sigma_5^\top\\
&\quad+\widetilde{\mathcal{A}}_1^\top\Sigma_4\Sigma_3\Sigma_1(\Pi_1(t)+\Pi_2(t))\mathcal{C}^\top+\widetilde{\mathcal{A}}_1^\top\Sigma_4(\Pi_1(t)+\Pi_2(t))\mathcal{C}^\top\big]\hat{\widetilde{\phi}}(t)\\
&+\big[(\Pi_1(t)+\Pi_2(t))\mathcal{C}\Sigma_1+\widetilde{\mathcal{A}}_1^\top\Sigma_4\Sigma_3\Sigma_1\big]\hat{\widetilde{\eta}}(t)
 +\big[\Pi_1(t)\mathcal{B}_1+\mathcal{A}_1^\top\big]\widetilde{\phi}(t)+\widetilde{\mathcal{A}}_1^\top\Sigma_4\widetilde{\eta}(t).
\end{aligned}
\end{equation}
Taking $\mathbb{E}[\cdot|\mathcal{G}_t^1]$ on both sides of \eqref{widetildealpha}, then we get
\begin{equation}\label{hatwidetildealpha}
\begin{aligned}
\hat{\widetilde{\alpha}}(t)=&\big[\Pi_1(t)\Sigma_6+(\Pi_1(t)+\Pi_2(t))\mathcal{C}\Sigma_1(\Pi_1(t)+\Pi_2(t))\mathcal{C}^\top+\Pi_2(t)\Sigma_8+\Sigma_5^\top\\
                            &\ +\widetilde{\mathcal{A}}_1^\top\Sigma_4\Sigma_3\Sigma_1(\Pi_1(t)+\Pi_2(t))\mathcal{C}^\top+\widetilde{\mathcal{A}}_1^\top\Sigma_4(\Pi_1(t)+\Pi_2(t))\mathcal{C}^\top
                             +\Pi_1(t)\mathcal{B}_1\\
                            &\ +\mathcal{A}_1^\top\big]\hat{\widetilde{\phi}}(t)+\big[(\Pi_1(t)+\Pi_2(t))\mathcal{C}\Sigma_1+\widetilde{\mathcal{A}}_1^\top\Sigma_4\Sigma_3\Sigma_1
                             +\widetilde{\mathcal{A}}_1^\top\Sigma_4\big]\hat{\widetilde{\eta}}(t).
\end{aligned}
\end{equation}
After taking $\mathbb{E}[\cdot|\mathcal{G}_t^1]$ on \eqref{widephi}, noting \eqref{hatwidetildealpha}, we can derive the equation of $(\hat{\widetilde{\phi}}(\cdot),\hat{\widetilde{\eta}}(\cdot))$:
\begin{equation}\label{hatwidetildephi}
\left\{
\begin{aligned}
-d\hat{\widetilde{\phi}}(t)&=\Big\{\big[\Pi_1(t)\Sigma_6+(\Pi_1(t)+\Pi_2(t))\mathcal{C}\Sigma_1(\Pi_1(t)+\Pi_2(t))\mathcal{C}^\top+\Pi_2(t)\Sigma_8+\Sigma_5^\top\\
                           &\qquad+\widetilde{\mathcal{A}}_1^\top\Sigma_4\Sigma_3\Sigma_1(\Pi_1(t)+\Pi_2(t))\mathcal{C}^\top
                            +\widetilde{\mathcal{A}}_1^\top\Sigma_4(\Pi_1(t)+\Pi_2(t))\mathcal{C}^\top+\Pi_1(t)\mathcal{B}_1\\
                           &\qquad+\mathcal{A}_1^\top\big]\hat{\widetilde{\phi}}(t)+\big[(\Pi_1(t)+\Pi_2(t))\mathcal{C}\Sigma_1+\widetilde{\mathcal{A}}_1^\top\Sigma_4\Sigma_3\Sigma_1
                             +\widetilde{\mathcal{A}}_1^\top\Sigma_4\big]\hat{\widetilde{\eta}}(t)\Big\}dt\\
                           &\quad-\hat{\widetilde{\eta}}(t)dW(t),\ t\in[0,T],\\
  \hat{\widetilde{\phi}}(T)&=\hat{\bar{\xi}}.
\end{aligned}
\right.
\end{equation}

In the meanwhile, we set
\begin{equation}\label{Relation2}
X(t)=\Pi_3(t)Y(t)+\Pi_4(t)\hat{Y}(t)+\widetilde{\varphi}(t),
\end{equation}
where $\Pi_3(\cdot),\Pi_4(\cdot)$ are $\mathbb{R}^{2n\times 2n}$-matrix-valued deterministic, differentiable functions with $\Pi_3(0)=\bar{G}$, $\Pi_4(0)=0$, and $\widetilde{\varphi}(\cdot)$ is an $\mathbb{R}^{2n}$-valued, $\mathcal{F}_t$-adapted process which satisfies the following SDE:
\begin{equation}\label{widevarphi}
\left\{
\begin{aligned}
 d\widetilde{\varphi}(t)&=\widetilde{\beta}(t)dt+\widetilde{\gamma}(t)dW(t),\ t\in[0,T],\\
  \widetilde{\varphi}(0)&=0,
\end{aligned}
\right.
\end{equation}
where $\widetilde{\beta}(\cdot),\widetilde{\gamma}(\cdot)$ are both $\mathbb{R}^{2n}$-valued, $\mathcal{F}_t$-adapted processes which are to be determined later.

Using the above notations, we can reformulate the equation of $Y(\cdot)$ in \eqref{2fbsdes2} as
\begin{equation}
\left\{
\begin{aligned}
-dY(t)&=\big[\Sigma_9X(t)+\mathcal{A}_1^{\top}Y(t)+\Sigma_5^\top\hat{Y}(t)+\widetilde{\mathcal{A}}_1^\top Z(t)+\widetilde{\mathcal{A}}_2^\top\widetilde{Z}(t)\big]dt\\
      &\quad-Z(t)dW(t)-\widetilde{Z}(t)d\widetilde{W}(t),\ t\in[0,T],\\
  Y(T)&=\bar{\xi}.
\end{aligned}
\right.
\end{equation}
By applying Lemma 5.4 in \cite{Xiong08}, we get
\begin{equation}
\left\{
\begin{aligned}
-d\hat{Y}(t)&=\big[\Sigma_9\hat{X}(t)+\Sigma_7^\top\hat{Y}(t)+\widetilde{\mathcal{A}}_1^\top\hat{Z}(t)+\widetilde{\mathcal{A}}_2^\top\hat{\widetilde{Z}}(t)\big]dt-\hat{Z}(t)dW(t),\ t\in[0,T],\\
  \hat{Y}(T)&=\hat{\bar{\xi}}.
\end{aligned}
\right.
\end{equation}
Applying It\^{o}'s formula to \eqref{Relation2}, we have
\begin{equation}\label{afterito2}
\begin{aligned}
dX(t)=&\big[\dot{\Pi}_3(t)Y(t)-\Pi_3(t)\Sigma_9X(t)-\Pi_3(t)\mathcal{A}_1^\top Y(t)-\Pi_3(t)\Sigma_5^\top\hat{Y}(t)-\Pi_3(t)\widetilde{\mathcal{A}}_1^\top Z(t)\\
      &\ -\Pi_3(t)\widetilde{\mathcal{A}}_2^\top\widetilde{Z}(t)+\dot{\Pi}_4(t)\hat{Y}(t)-\Pi_4(t)\Sigma_9\hat{X}(t)-\Pi_4(t)\Sigma_7^\top\hat{Y}(t)-\Pi_4(t)\widetilde{\mathcal{A}}_1^\top\hat{Z}(t)\\
      &\ -\Pi_4(t)\widetilde{\mathcal{A}}_2^\top\hat{\widetilde{Z}}(t)+\widetilde{\beta}(t)\big]dt+\big[\Pi_3(t)Z(t)+\Pi_4(t)\hat{Z}(t)+\widetilde{\gamma}(t)\big]dW(t)\\
      &\ +\Pi_3(t)\widetilde{Z}(t)d\widetilde{W}(t)=\big[\mathcal{A}_1X(t)+\Sigma_5\hat{X}(t)+\mathcal{B}_1Y(t)+\Sigma_6\hat{Y}(t)+\mathcal{C}\hat{Z}(t)\big]dt\\
      &+\big[\widetilde{\mathcal{A}}_1X(t)+\mathcal{C}^\top\hat{Y}(t)+\widetilde{\mathcal{C}}_1Z(t)+\widetilde{\mathcal{C}}_2\hat{Z}(t)\big]dW(t)
      +\big[\widetilde{\mathcal{A}}_2X(t)+\widetilde{\mathcal{C}}_3\widetilde{Z}(t)\big]d\widetilde{W}(t).
\end{aligned}
\end{equation}
Then comparing the diffusion terms of $dW(\cdot)$ and $d\widetilde{W}(\cdot)$ on both sides of \eqref{afterito2}, we obtain
\begin{equation}\label{comparedW}
\left\{
\begin{aligned}
&\Pi_3(t)Z(t)+\Pi_4(t)\hat{Z}(t)+\widetilde{\gamma}(t)=\widetilde{\mathcal{A}}_1X(t)+\mathcal{C}^\top\hat{Y}(t)+\widetilde{\mathcal{C}}_1Z(t)+\widetilde{\mathcal{C}}_2\hat{Z}(t),\\
&\Pi_3(t)\widetilde{Z}(t)=\widetilde{\mathcal{A}}_2X(t)+\widetilde{\mathcal{C}}_3\widetilde{Z}(t).
\end{aligned}
\right.
\end{equation}
Comparing the drift term of $dt$ on both sides of \eqref{afterito2}, we derive
\begin{equation}\label{driftdt}
\begin{aligned}
&\dot{\Pi}_3(t)Y(t)-\Pi_3(t)\Sigma_9X(t)-\Pi_3(t)\mathcal{A}_1^\top Y(t)-\Pi_3(t)\Sigma_5^\top\hat{Y}(t)-\Pi_3(t)\widetilde{\mathcal{A}}_1^\top Z(t)\\
&-\Pi_3(t)\widetilde{\mathcal{A}}_2^\top\widetilde{Z}(t)+\dot{\Pi}_4(t)\hat{Y}(t)-\Pi_4(t)\Sigma_9\hat{X}(t)-\Pi_4(t)\Sigma_7^\top\hat{Y}(t)-\Pi_4(t)\widetilde{\mathcal{A}}_1^\top\hat{Z}(t)\\
&-\Pi_4(t)\widetilde{\mathcal{A}}_2^\top\hat{\widetilde{Z}}(t)+\widetilde{\beta}(t)-\mathcal{A}_1X(t)-\Sigma_5\hat{X}(t)-\mathcal{B}_1Y(t)-\Sigma_6\hat{Y}(t)-\mathcal{C}\hat{Z}(t)=0.
\end{aligned}
\end{equation}
Putting \eqref{hatZ} and \eqref{Z2} into \eqref{comparedW}, we get
\begin{equation}\label{widetildegamma}
\begin{aligned}
\widetilde{\gamma}(t)&=\big[\widetilde{\mathcal{A}}_1\Pi_3(t)+(\widetilde{\mathcal{C}}_1-\Pi_3(t))\Sigma_4\Pi_1(t)\widetilde{\mathcal{A}}_1\Pi_3(t)\big]Y(t)
                      +\Big\{\widetilde{\mathcal{A}}_1\Pi_4(t)+\mathcal{C}^\top+(\widetilde{\mathcal{C}}_1-\Pi_3(t))\\
                     &\qquad\times\Sigma_4\Pi_1(t)\widetilde{\mathcal{A}}_1\Pi_4(t)
                      +(\widetilde{\mathcal{C}}_1-\Pi_3(t))\big[\Sigma_4\Sigma_2+\Sigma_4\Sigma_3\Sigma_1(\Pi_1(t)+\Pi_2(t))\widetilde{\mathcal{A}}_1\\
                     &\qquad+\Sigma_4\Pi_2(t)\widetilde{\mathcal{A}}_1+\Sigma_4\Sigma_3\Sigma_1\Sigma_2\big](\Pi_3(t)+\Pi_4(t))+(\widetilde{\mathcal{C}}_2-\Pi_4(t))\Sigma_1(\Pi_1(t)\\
                     &\qquad+\Pi_2(t))\widetilde{\mathcal{A}}_1(\Pi_3(t)+\Pi_4(t))+(\widetilde{\mathcal{C}}_2-\Pi_4(t))\Sigma_1(\Pi_1(t)+\Pi_2(t))\mathcal{C}^\top\Big\}\hat{Y}(t)\\
                     &\quad+\Big\{(\widetilde{\mathcal{C}}_1-\Pi_3(t))\big[\Sigma_4\Sigma_2+\Sigma_4\Sigma_3\Sigma_1(\Pi_1(t)+\Pi_2(t))\widetilde{\mathcal{A}}_1+\Sigma_4\Pi_2(t)\widetilde{\mathcal{A}}_1
                      +\Sigma_4\Sigma_3\Sigma_1\Sigma_2\big]\\
                     &\qquad+(\widetilde{\mathcal{C}}_2-\Pi_4(t))\Sigma_1(\Pi_1(t)+\Pi_2(t))\widetilde{\mathcal{A}}_1\Big\}\hat{\widetilde{\varphi}}(t)
                      +\big[\widetilde{\mathcal{A}}_1+(\widetilde{\mathcal{C}}_1-\Pi_3(t))\Sigma_4\Pi_1(t)\widetilde{\mathcal{A}}_1\big]\widetilde{\varphi}(t)\\
                     &\quad+(\widetilde{\mathcal{C}}_1-\Pi_3(t))\big[\Sigma_4\Sigma_3\Sigma_1(\Pi_1(t)+\Pi_2(t))\mathcal{C}^\top+\Sigma_4(\Pi_1(t)+\Pi_2(t))\mathcal{C}^\top\big]\hat{\widetilde{\phi}}(t)\\
                     &\quad+\big[(\widetilde{\mathcal{C}}_1-\Pi_3(t))\Sigma_4\Sigma_3\Sigma_1
                      +(\widetilde{\mathcal{C}}_2-\Pi_4(t))\Sigma_1\big]\hat{\widetilde{\eta}}(t)+(\widetilde{\mathcal{C}}_1-\Pi_3(t))\Sigma_4\widetilde{\eta}(t).
\end{aligned}
\end{equation}
Then taking $\mathbb{E}[\cdot|\mathcal{G}_t^1]$ on \eqref{widetildeZ2}, noting \eqref{Relation2}, we derive
\begin{equation}\label{hatwidetildeZ}
\hat{\widetilde{Z}}(t)=\Sigma_{10}\Pi_1(t)\widetilde{\mathcal{A}}_2\big[(\Pi_3(t)+\Pi_4(t))\hat{Y}(t)+\hat{\widetilde{\varphi}}(t)\big].
\end{equation}
Substituting \eqref{hatZ}, \eqref{Z2}, \eqref{widetildeZ2} and \eqref{hatwidetildeZ} into \eqref{driftdt}, we get
\begin{equation*}
\begin{aligned}
&\big[\dot{\Pi}_3(t)-\mathcal{A}_1\Pi_3(t)-\Pi_3(t)\Sigma_9\Pi_3(t)-\Pi_3(t)\mathcal{A}_1^\top-\Pi_3(t)\widetilde{\mathcal{A}}_1^\top\Sigma_4\Pi_1(t)\widetilde{\mathcal{A}}_1\Pi_3(t)\\
&\quad-\Pi_3(t)\widetilde{\mathcal{A}}_2^\top\Sigma_{10}\Pi_1(t)\widetilde{\mathcal{A}}_2\Pi_3(t)-\mathcal{B}_1\big]Y(t)+\Big\{\dot{\Pi}_4(t)-\Pi_3(t)\Sigma_9\Pi_4(t)-\Pi_3(t)\Sigma_5^\top\\
&\quad-\Pi_3(t)\widetilde{\mathcal{A}}_1^\top\Sigma_4\Pi_1(t)\widetilde{\mathcal{A}}_1\Pi_4(t)-\Pi_3(t)\widetilde{\mathcal{A}}_1^\top\big[\Sigma_4\Sigma_2+\Sigma_4\Sigma_3\Sigma_1(\Pi_1(t)+\Pi_2(t))\widetilde{\mathcal{A}}_1\\
&\quad+\Sigma_4\Pi_2(t)\widetilde{\mathcal{A}}_1+\Sigma_4\Sigma_3\Sigma_1\Sigma_2\big](\Pi_3(t)+\Pi_4(t))-\Pi_3(t)\widetilde{\mathcal{A}}_2^\top\Sigma_{10}\Pi_1(t)\widetilde{\mathcal{A}}_2\Pi_4(t)\\
&\quad-\Pi_4(t)\Sigma_9(\Pi_3(t)+\Pi_4(t))-\Pi_4(t)\Sigma_7^\top-\Pi_4(t)\widetilde{\mathcal{A}}_1^{\top}\Sigma_1(\Pi_1(t)+\Pi_2(t))\widetilde{\mathcal{A}}_1(\Pi_3(t)\\
&\quad+\Pi_4(t))-\Pi_4(t)\widetilde{\mathcal{A}}_1^\top\Sigma_1(\Pi_1(t)+\Pi_2(t))\mathcal{C}^\top-\Pi_4(t)\widetilde{\mathcal{A}}_2^\top\Sigma_{10}\Pi_1(t)\widetilde{\mathcal{A}}_2(\Pi_3(t)+\Pi_4(t))\\
&\quad-\mathcal{A}_1\Pi_4(t)-\Sigma_5(\Pi_3(t)+\Pi_4(t))-\Sigma_6-\mathcal{C}\Sigma_1(\Pi_1(t)+\Pi_2(t))\widetilde{\mathcal{A}}_1(\Pi_3(t)+\Pi_4(t))\\
&\quad-\mathcal{C}\Sigma_1(\Pi_1(t)+\Pi_2(t))\mathcal{C}^\top\Big\}\hat{Y}(t)+\Big\{-\Pi_3(t)\widetilde{\mathcal{A}}_1^\top\big[\Sigma_4\Sigma_2
 +\Sigma_4\Sigma_3\Sigma_1(\Pi_1(t)+\Pi_2(t))\widetilde{\mathcal{A}}_1\\
&\quad+\Sigma_4\Pi_2(t)\widetilde{\mathcal{A}}_1+\Sigma_4\Sigma_3\Sigma_1\Sigma_2\big]-\Pi_4(t)\Sigma_9-\Pi_4(t)\widetilde{\mathcal{A}}_1^\top\Sigma_1(\Pi_1(t)+\Pi_2(t))\widetilde{\mathcal{A}}_1\\
&\quad-\Pi_4(t)\widetilde{\mathcal{A}}_2^\top\Sigma_{10}\Pi_1(t)\widetilde{\mathcal{A}}_2-\Sigma_5-\mathcal{C}\Sigma_1(\Pi_1(t)+\Pi_2(t))\widetilde{\mathcal{A}}_1\Big\}\hat{\widetilde{\varphi}}(t)\\
\end{aligned}
\end{equation*}
\begin{equation}
\begin{aligned}
&+\big[-\Pi_3(t)\Sigma_9-\Pi_3(t)\widetilde{\mathcal{A}}_1^\top\Sigma_4\Pi_1(t)\widetilde{\mathcal{A}}_1
 -\Pi_3(t)\widetilde{\mathcal{A}}_2^\top\Sigma_{10}\Pi_1(t)\widetilde{\mathcal{A}}_2-\mathcal{A}_1\big]\widetilde{\varphi}(t)\\
&\quad-\Pi_3(t)\widetilde{\mathcal{A}}_1^\top\big[\Sigma_4\Sigma_3\Sigma_1(\Pi_1(t)+\Pi_2(t))\mathcal{C}^\top+\Sigma_4(\Pi_1(t)+\Pi_2(t))\mathcal{C}^\top\big]\hat{\widetilde{\phi}}(t)\\
&+\big[-\Pi_3(t)\widetilde{\mathcal{A}}_1^\top\Sigma_4\Sigma_3\Sigma_1-\Pi_4(t)\widetilde{\mathcal{A}}_1^\top\Sigma_1-\mathcal{C}\Sigma_1\big]\hat{\widetilde{\eta}}(t)
 -\Pi_3(t)\widetilde{\mathcal{A}}_1^\top\Sigma_4\widetilde{\eta}(t)+\widetilde{\beta}(t)=0,
\end{aligned}
\end{equation}
If $\Pi_3(\cdot)$ and $\Pi_4(\cdot)$ satisfy the following two Riccati equations, one by one:
\begin{equation}\label{Pi3}
\left\{
\begin{aligned}
&\dot{\Pi}_3(t)-\mathcal{A}_1\Pi_3(t)-\Pi_3(t)\mathcal{A}_1^\top-\Pi_3(t)\Sigma_9\Pi_3(t)-\Pi_3(t)\widetilde{\mathcal{A}}_1^\top\Sigma_4\Pi_1(t)\widetilde{\mathcal{A}}_1\Pi_3(t)\\
&-\Pi_3(t)\widetilde{\mathcal{A}}_2^\top\Sigma_{10}\Pi_1(t)\widetilde{\mathcal{A}}_2\Pi_3(t)-\mathcal{B}_1=0,\ t\in[0,T],\ \Pi_3(0)=\bar{G},
\end{aligned}
\right.
\end{equation}
and
\begin{equation}\label{Pi4}
\left\{
\begin{aligned}
&\dot{\Pi}_4(t)-\Pi_3(t)\Sigma_5^\top-\Sigma_5\Pi_3(t)-\Pi_4(t)\mathcal{A}_1^\top-\mathcal{A}_1\Pi_4(t)-\Pi_4(t)\Sigma_5^\top-\Sigma_5\Pi_4(t)\\
&-\Pi_3(t)\Sigma_9\Pi_4(t)-\Pi_4(t)\Sigma_9\Pi_3(t)-\Pi_4(t)\Sigma_9\Pi_4(t)-\Pi_3(t)\widetilde{\mathcal{A}}_1^\top\Sigma_4\Pi_1(t)\widetilde{\mathcal{A}}_1\Pi_4(t)\\
&-\Pi_3(t)\widetilde{\mathcal{A}}_1^\top\Sigma_4\Sigma_2(\Pi_3(t)+\Pi_4(t))-\Pi_3(t)\widetilde{\mathcal{A}}_1^\top\Sigma_4\Sigma_3\Sigma_1(\Pi_1(t)+\Pi_2(t))\widetilde{\mathcal{A}}_1\\
&\times(\Pi_3(t)+\Pi_4(t))-\Pi_3(t)\widetilde{\mathcal{A}}_1^\top\Sigma_4\Sigma_2\widetilde{\mathcal{A}}_1(\Pi_3(t)+\Pi_4(t))\\
&-\Pi_3(t)\widetilde{\mathcal{A}}_1^\top\Sigma_4\Sigma_3\Sigma_1\Sigma_2(\Pi_3(t)+\Pi_4(t))
 -\Pi_3(t)\widetilde{\mathcal{A}}_2^\top\Sigma_{10}\Pi_1\widetilde{\mathcal{A}}_2\Pi_4(t)\\
&-\Pi_4(t)\widetilde{\mathcal{A}}_1^\top\Sigma_1(\Pi_1(t)+\Pi_2(t))\widetilde{\mathcal{A}}_1(\Pi_3(t)+\Pi_4(t))
 -\Pi_4(t)\widetilde{\mathcal{A}}_1^\top\Sigma_1(\Pi_1(t)+\Pi_2(t))\mathcal{C}^\top\\
&-\Pi_4(t)\widetilde{\mathcal{A}}_2^\top\Sigma_{10}\Pi_1(t)\widetilde{\mathcal{A}}_2(\Pi_3(t)+\Pi_4(t))-\mathcal{C}\Sigma_1(\Pi_1(t)+\Pi_2(t))\mathcal{C}^\top\\
&-\mathcal{C}\Sigma_1(\Pi_1(t)+\Pi_2(t))\widetilde{\mathcal{A}}_1(\Pi_3(t)+\Pi_4(t))-\Sigma_6=0,\ t\in[0,T],\ \Pi_4(0)=0,
\end{aligned}
\right.
\end{equation}
then we get
\begin{equation}\label{widetildebeta}
\begin{aligned}
\widetilde{\beta}(t)=&\Big\{\Pi_3(t)\widetilde{\mathcal{A}}_1^\top\big[\Sigma_4\Sigma_2
                      +\Sigma_4\Sigma_3\Sigma_1(\Pi_1(t)+\Pi_2(t))\widetilde{\mathcal{A}}_1+\Sigma_4\Pi_2(t)\widetilde{\mathcal{A}}_1+\Sigma_4\Sigma_3\Sigma_1\Sigma_2\big]+\Pi_4(t)\Sigma_9\\
                     &\quad-\Pi_4(t)\widetilde{\mathcal{A}}_1^\top\Sigma_1(\Pi_1(t)+\Pi_2(t))\widetilde{\mathcal{A}}_1
                      +\Pi_4(t)\widetilde{\mathcal{A}}_2^\top\Sigma_{10}\Pi_1(t)\widetilde{\mathcal{A}}_2+\Sigma_5+\mathcal{C}\Sigma_1(\Pi_1(t)\\
                     &\quad+\Pi_2(t))\widetilde{\mathcal{A}}_1\Big\}\hat{\widetilde{\varphi}}(t)
                      +\big[\Pi_3(t)\Sigma_9+\Pi_3(t)\widetilde{\mathcal{A}}_1^\top\Sigma_4\Pi_1(t)\widetilde{\mathcal{A}}_1+\Pi_3(t)\widetilde{\mathcal{A}}_2^\top\Sigma_{10}\Pi_1(t)\widetilde{\mathcal{A}}_2\\
                     &\quad+\mathcal{A}_1\big]\widetilde{\varphi}(t)+\Pi_3(t)\widetilde{\mathcal{A}}_1^\top\big[\Sigma_4\Sigma_3\Sigma_1(\Pi_1(t)+\Pi_2(t))\mathcal{C}^\top
                      +\Sigma_4(\Pi_1(t)+\Pi_2(t))\mathcal{C}^\top\big]\hat{\widetilde{\phi}}(t)\\
                     &+\big[\Pi_3(t)\widetilde{\mathcal{A}}_1^\top\Sigma_4\Sigma_3\Sigma_1+\Pi_4(t)\widetilde{\mathcal{A}}_1^\top\Sigma_1+\mathcal{C}\Sigma_1\big]\hat{\widetilde{\eta}}(t)
                      +\Pi_3(t)\widetilde{\mathcal{A}}_1^\top\Sigma_4\widetilde{\eta}(t).
\end{aligned}
\end{equation}
Taking $\mathbb{E}[\cdot|\mathcal{G}_t^1]$ on \eqref{widetildegamma}, we get
\begin{equation}\label{hatwidetildegamma}
\begin{aligned}
\hat{\widetilde{\gamma}}(t)&=\Big\{\widetilde{\mathcal{A}}_1\Pi_3(t)+(\widetilde{\mathcal{C}}_1-\Pi_3(t))\Sigma_4\Pi_1(t)\widetilde{\mathcal{A}}_1\Pi_3(t)+\widetilde{\mathcal{A}}_1\Pi_4(t)+\mathcal{C}^\top+(\widetilde{\mathcal{C}}_1-\Pi_3(t))\\
                           &\qquad\times\Sigma_4\Pi_1(t)\widetilde{\mathcal{A}}_1\Pi_4(t)
                            +(\widetilde{\mathcal{C}}_1-\Pi_3(t))\big[\Sigma_4\Sigma_2+\Sigma_4\Sigma_3\Sigma_1(\Pi_1(t)+\Pi_2(t))\widetilde{\mathcal{A}}_1\\
                           &\qquad+\Sigma_4\Pi_2(t)\widetilde{\mathcal{A}}_1+\Sigma_4\Sigma_3\Sigma_1\Sigma_2\big](\Pi_3(t)+\Pi_4(t))+(\widetilde{\mathcal{C}}_2-\Pi_4(t))\Sigma_1(\Pi_1(t)\\
                           &\qquad+\Pi_2(t))\widetilde{\mathcal{A}}_1(\Pi_3(t)+\Pi_4(t))+(\widetilde{\mathcal{C}}_2-\Pi_4(t))\Sigma_1(\Pi_1(t)+\Pi_2(t))\mathcal{C}^\top\Big\}\hat{Y}(t)\\
                           &\quad+\Big\{(\widetilde{\mathcal{C}}_1-\Pi_3(t))\big[\Sigma_4\Sigma_2+\Sigma_4\Sigma_3\Sigma_1(\Pi_1(t)+\Pi_2(t))\widetilde{\mathcal{A}}_1+\Sigma_4\Pi_2(t)\widetilde{\mathcal{A}}_1
                            +\Sigma_4\Sigma_3\Sigma_1\Sigma_2\big]\\
                           &\qquad+(\widetilde{\mathcal{C}}_2-\Pi_4(t))\Sigma_1(\Pi_1(t)+\Pi_2(t))\widetilde{\mathcal{A}}_1
                            +\widetilde{\mathcal{A}}_1+(\widetilde{\mathcal{C}}_1-\Pi_3(t))\Sigma_4\Pi_1(t)\widetilde{\mathcal{A}}_1\Big\}\hat{\widetilde{\varphi}}(t)\\
                           &\quad+(\widetilde{\mathcal{C}}_1-\Pi_3(t))\big[\Sigma_4\Sigma_3\Sigma_1(\Pi_1(t)+\Pi_2(t))\mathcal{C}^\top+\Sigma_4(\Pi_1(t)+\Pi_2(t))\mathcal{C}^\top\big]\hat{\widetilde{\phi}}(t)\\
                           &\quad+\big[(\widetilde{\mathcal{C}}_1-\Pi_3(t))\Sigma_4\Sigma_3\Sigma_1
                            +(\widetilde{\mathcal{C}}_2-\Pi_4(t))\Sigma_1+(\widetilde{\mathcal{C}}_1-\Pi_3(t))\Sigma_4\big]\hat{\widetilde{\eta}}(t).
\end{aligned}
\end{equation}
Then taking $\mathbb{E}[\cdot|\mathcal{G}_t^1]$ on \eqref{Relation1} and \eqref{Relation2}, we get
\begin{equation}\label{hatYhatX}
\left\{
\begin{aligned}
&\hat{X}(t)=(\Pi_3(t)+\Pi_4(t))\hat{Y}(t)+\hat{\widetilde{\varphi}}(t),\\
&\hat{Y}(t)=(\Pi_1(t)+\Pi_2(t))\hat{X}(t)+\hat{\widetilde{\phi}}(t).
\end{aligned}
\right.
\end{equation}
Then
\begin{equation}\label{hatYhatY}
\hat{Y}(t)=\Sigma_{11}(\Pi_1(t)+\Pi_2(t))\hat{\widetilde{\varphi}}(t)+\Sigma_{11}\hat{\widetilde{\phi}}(t),
\end{equation}
where, we set
\begin{equation*}
\Sigma_{11}=[I-(\Pi_1(t)+\Pi_2(t))(\Pi_3(t)+\Pi_4(t))]^{-1}.
\end{equation*}
Therefore, we can rewrite the equation \eqref{hatwidetildegamma} as
\begin{equation}\label{hatwidetildegamma2}
\begin{aligned}
\hat{\widetilde{\gamma}}(t)&=\Big\{\big\{\widetilde{\mathcal{A}}_1(\Pi_3(t)+\Pi_4(t))+(\widetilde{\mathcal{C}}_1-\Pi_3(t))\Sigma_4\Pi_1(t)\widetilde{\mathcal{A}}_1(\Pi_3(t)+\Pi_4(t))+\mathcal{C}^\top\\
&\qquad+(\widetilde{\mathcal{C}}_1-\Pi_3(t))\big[\Sigma_4\Sigma_2+\Sigma_4\Sigma_3\Sigma_1(\Pi_1(t)+\Pi_2(t))\widetilde{\mathcal{A}}_1+\Sigma_4\Pi_2(t)\widetilde{\mathcal{A}}_1\\
&\qquad+\Sigma_4\Sigma_3\Sigma_1\Sigma_2\big](\Pi_3(t)+\Pi_4(t))+(\widetilde{\mathcal{C}}_2-\Pi_4(t))\Sigma_1(\Pi_1(t)+\Pi_2(t))\widetilde{\mathcal{A}}_1(\Pi_3(t)+\Pi_4(t))\\
&\qquad+(\widetilde{\mathcal{C}}_2-\Pi_4(t))\Sigma_1(\Pi_1(t)+\Pi_2(t))\mathcal{C}^\top\big\}\Sigma_{11}(\Pi_1(t)+\Pi_2(t))\\
&\qquad+(\widetilde{\mathcal{C}}_1-\Pi_3(t))\big[\Sigma_4\Sigma_2+\Sigma_4\Sigma_3\Sigma_1(\Pi_1(t)+\Pi_2(t))\widetilde{\mathcal{A}}_1+\Sigma_4\Pi_2(t)\widetilde{\mathcal{A}}_1+\Sigma_4\Sigma_3\Sigma_1\Sigma_2\big]\\
&\qquad+(\widetilde{\mathcal{C}}_2-\Pi_4(t))\Sigma_1(\Pi_1(t)+\Pi_2(t))\widetilde{\mathcal{A}}_1+\widetilde{\mathcal{A}}_1
 +(\widetilde{\mathcal{C}}_1-\Pi_3(t))\Sigma_4\Pi_1(t)\widetilde{\mathcal{A}}_1\Big\}\hat{\widetilde{\varphi}}(t)\\
&\quad+\Big\{\big\{\widetilde{\mathcal{A}}_1(\Pi_3(t)+\Pi_4(t))+(\widetilde{\mathcal{C}}_1-\Pi_3(t))\Sigma_4\Pi_1(t)\widetilde{\mathcal{A}}_1(\Pi_3(t)+\Pi_4(t))+\mathcal{C}^\top\\
&\qquad+(\widetilde{\mathcal{C}}_1-\Pi_3(t))\big[\Sigma_4\Sigma_2+\Sigma_4\Sigma_3\Sigma_1(\Pi_1(t)+\Pi_2(t))\widetilde{\mathcal{A}}_1+\Sigma_4\Pi_2(t)\widetilde{\mathcal{A}}_1\\
&\qquad+\Sigma_4\Sigma_3\Sigma_1\Sigma_2\big](\Pi_3(t)+\Pi_4(t))+(\widetilde{\mathcal{C}}_2-\Pi_4(t))\Sigma_1(\Pi_1(t)+\Pi_2(t))\widetilde{\mathcal{A}}_1(\Pi_3(t)+\Pi_4(t))\\
&\qquad+(\widetilde{\mathcal{C}}_2-\Pi_4(t))\Sigma_1(\Pi_1+\Pi_2)\mathcal{C}^\top\big\}\Sigma_{11}\\
&\qquad+(\widetilde{\mathcal{C}}_1-\Pi_3(t))\big[\Sigma_4\Sigma_3\Sigma_1(\Pi_1(t)+\Pi_2(t))\mathcal{C}^\top+\Sigma_4(\Pi_1(t)+\Pi_2(t))\mathcal{C}^\top\big]\Big\}\hat{\widetilde{\phi}}(t)\\
&\quad+\Big\{(\widetilde{\mathcal{C}}_1-\Pi_3(t))\Sigma_4\Sigma_3\Sigma_1+(\widetilde{\mathcal{C}}_2-\Pi_4(t))\Sigma_1+(\widetilde{\mathcal{C}}_1-\Pi_3(t))\Sigma_4\Big\}\hat{\widetilde{\eta}}(t).
\end{aligned}
\end{equation}
By taking $\mathbb{E}[\cdot|\mathcal{G}_t^1]$ on \eqref{widetildebeta}, we can get
\begin{equation}\label{hatwidetildebeta}
\begin{aligned}
\hat{\widetilde{\beta}}(t)=&\Big\{\Pi_3(t)\widetilde{\mathcal{A}}_1^\top\big[\Sigma_4\Sigma_2+\Sigma_4\Sigma_3\Sigma_1(\Pi_1(t)+\Pi_2(t))\widetilde{\mathcal{A}}_1
 +\Sigma_4\Pi_2(t)\widetilde{\mathcal{A}}_1+\Sigma_4\Sigma_3\Sigma_1\Sigma_2\big]\\
&\quad+\Pi_4(t)\Sigma_9+\Pi_4(t)\widetilde{\mathcal{A}}_1^\top\Sigma_1(\Pi_1(t)+\Pi_2(t))\widetilde{\mathcal{A}}_1\\
&+\Pi_4(t)\widetilde{\mathcal{A}}_2^\top\Sigma_{10}\Pi_1(t)\widetilde{\mathcal{A}}_2+\Sigma_5+\mathcal{C}\Sigma_1(\Pi_1(t)+\Pi_2(t))\widetilde{\mathcal{A}}_1+\Pi_3(t)\Sigma_9\\
&+\Pi_3(t)\widetilde{\mathcal{A}}_1^\top\Sigma_4\Pi_1(t)\widetilde{\mathcal{A}}_1+\Pi_3(t)\widetilde{\mathcal{A}}_2^\top\Sigma_{10}\Pi_1(t)\widetilde{\mathcal{A}}_2+\mathcal{A}_1\Big\}\hat{\widetilde{\varphi}}(t)\\
&+\Pi_3(t)\widetilde{\mathcal{A}}_1^\top\big[\Sigma_4\Sigma_3\Sigma_1(\Pi_1(t)+\Pi_2(t))\mathcal{C}^\top+\Sigma_4(\Pi_1+\Pi_2)\mathcal{C}^\top\big]\hat{\widetilde{\phi}}(t)\\
&+\big[\Pi_3(t)\widetilde{\mathcal{A}}_1^\top\Sigma_4\Sigma_3\Sigma_1+\Pi_4(t)\widetilde{\mathcal{A}}_1^\top\Sigma_1+\mathcal{C}\Sigma_1+\Pi_3(t)\widetilde{\mathcal{A}}_1^\top\Sigma_4\big]\hat{\widetilde{\eta}}(t).
\end{aligned}
\end{equation}
Thus, we can derive the equation of $\hat{\widetilde{\varphi}}(\cdot)$:
\begin{equation}\label{hatwidetildevarphi}
\left\{
\begin{aligned}
d\hat{\widetilde{\varphi}}(t)&=\hat{\widetilde{\beta}}(t)dt+\hat{\widetilde{\gamma}}(t)dW(t),\ t\in[0,T],\\
 \hat{\widetilde{\varphi}}(0)&=0,
\end{aligned}
\right.
\end{equation}
where $\hat{\widetilde{\beta}}(t)$ and $\hat{\widetilde{\gamma}}(t)$ satisfy \eqref{hatwidetildebeta} and \eqref{hatwidetildegamma2}, respectively.

\begin{theorem}
Under assumption {\bf(L1)}, {\bf(L2)} and {\bf(L3)}, suppose the Riccati equations \eqref{Pi1}, \eqref{Pi2}, \eqref{Pi3} and \eqref{Pi4} admit differentiable solutions $\Pi_1(\cdot)$, $\Pi_2(\cdot)$, $\Pi_3(\cdot)$ and $\Pi_4(\cdot)$, respectively. Then the leader's problem is solvable with the optimal strategy $\bar{v}_2(\cdot)$ being of a state estimate feedback representation
\begin{equation}\label{barv2}
\begin{aligned}
\bar{v}_2(t)&=-R_2^{-1}(t)\widetilde{\mathcal{B}}_2^\top\Pi_3(t)Y(t)-R_2^{-1}(t)\big[\widetilde{\mathcal{B}}_2^\top\Pi_4(t)+\mathcal{D}^\top\big]\hat{Y}(t)\\
            &\quad-R_2^{-1}(t)\widetilde{\mathcal{B}}_2^\top\widetilde{\varphi}(t),\ a.e.\ t\in[0,T],\ a.s.,
\end{aligned}
\end{equation}
where $Y(\cdot)$, $\hat{Y}(\cdot)$ satisfy the following BSDEs, respectively:
\begin{equation}\label{Y}
\left\{
\begin{aligned}
-dY(t)&=\big[(\Sigma_9\Pi_3(t)+\mathcal{A}_1^\top)Y(t)+(\Sigma_9\Pi_4(t)+\Sigma_5^\top)\hat{Y}(t)+\Sigma_9\widetilde{\varphi}(t)+\widetilde{\mathcal{A}}_1^\top Z(t)+\widetilde{\mathcal{A}}_2^\top\widetilde{Z}(t)\big]dt\\
      &\quad-Z(t)dW(t)-\widetilde{Z}(t)d\widetilde{W}(t),\ t\in[0,T],\\
Y(T)&=\bar{\xi},
\end{aligned}
\right.
\end{equation}
\begin{equation}\label{hatY}
\left\{
\begin{aligned}
-d\hat{Y}(t)&=\big[(\Sigma_9(\Pi_3(t)+\Pi_4(t))+\mathcal{A}_1^\top+\Sigma_5^\top)\hat{Y}(t)+\Sigma_9\hat{\widetilde{\varphi}}(t)+\widetilde{\mathcal{A}}_1^\top \hat{Z}(t)+\widetilde{\mathcal{A}}_2^\top\hat{\widetilde{Z}}(t)\big]dt\\
            &\quad-\hat{Z}(t)dW(t),\ t\in[0,T],\\
\hat{Y}(T)&=\hat{\bar{\xi}},
\end{aligned}
\right.
\end{equation}
and $\widetilde{\varphi}(\cdot)$ satisfies the SDE \eqref{widevarphi}.
Meanwhile, by \eqref{Relation1} and \eqref{Relation2}, we can get
\begin{equation}\label{Yvarphi}
\begin{aligned}
&(I-\Pi_1(t)\Pi_3(t))Y(t)-\big[\Pi_1(t)\Pi_4(t)+\Pi_2(t)(\Pi_3(t)+\Pi_4(t))\big]\hat{Y}(t)\\
&=\Pi_2(t)\hat{\widetilde{\varphi}}(t)+\widetilde{\phi}(t)+\Pi_1(t)\widetilde{\varphi}(t).
\end{aligned}
\end{equation}
\end{theorem}

\begin{proof}
For given $\xi$, let $\Pi_1(\cdot)$ and $\Pi_2(\cdot)$ satisfy \eqref{Pi1} and \eqref{Pi2}, respectively. By the standard BSDE theory, we can solve \eqref{hatwidetildephi} to obtain $(\hat{\widetilde{\phi}}(\cdot),\hat{\widetilde{\eta}}(\cdot))$, and due to \eqref{widetildealpha}, we can solve \eqref{widephi} to obtain $(\widetilde{\phi}(\cdot),\widetilde{\eta}(\cdot))$. Let $\Pi_3(\cdot)$ and $\Pi_4(\cdot)$ satisfy \eqref{Pi3} and \eqref{Pi4}, respectively. By the standard SDE theory, we can solve \eqref{hatwidetildevarphi} to obtain $\hat{\widetilde{\varphi}}(\cdot)$, then $\hat{Y}(\cdot)$ can be solved by \eqref{hatYhatY}. Due to \eqref{Y} and \eqref{Yvarphi}, we can get $Y(\cdot)$ and $\widetilde{\varphi}(\cdot)$, thus the state estimate feedback representation \eqref{barv2} can be obtained. The proof is complete.
\end{proof}

Likewise, from \eqref{barv_1}, the optimal control $\bar{v}_1(\cdot)$ of the follower can also be represented in a nonanticipating way:
\begin{equation}\label{barv 1 }
\begin{aligned}
\bar{v}_1(t)=&-R_1^{-1}(t)B_1(t)^\top\Big\{\big[(0,P_2(t))+(1,0)\Pi_4(t)\big]\hat{Y}(t)+(1,0)\Pi_3(t)Y(t)+(1,0)\widetilde{\varphi}(t)\Big\},\\
&\ a.e.\ t\in[0,T],\ a.s.
\end{aligned}
\end{equation}

\section{Application to Pension Fund Management Problem}

In this section, we are denoted to study a defined benefit (DB) pension fund management problem arising from financial markets, which naturally motivate the above theoretical research of the LQ Stackelberg game for BSDE with partial information.

It is well known that a pension fund can be classified into two main categories: Defined benefit (DB) pension scheme and defined contribution (DC) pension scheme. In a DB scheme, the benefits are fixed in advance by the sponsor, and the contributions are designed to assure the future payments to claim holders in their retirement period. There are two corresponding representative members who makes contributions continuously over time to the pension fund in $[0,T]$. One of the members is the leader with the regular premium proportion $v_2$ as his contribution, who is usually regarded as the supervisory, government or company. And the other one is the follower with the regular premium proportion $v_1$ as his contribution, who is usually regarded as the individual producer or retail investor. Premiums are a proportion of salary or income which are continuously deposited into the pension fund plan member's account as the contributions.

We consider a continuous-time setup, and the dynamics of pension fund plan member's account is given by
\begin{equation}
dF(t)=F(t)d\Delta(t)+(v_1(t)+v_2(t)-DB)dt,
\end{equation}
where $F(t)$ is the value process of pension fund plan member's account at time $t$, $d\Delta(t)$ is the instantaneous return during the time interval $(t,t+dt)$, $v_1(\cdot)$ and $v_2(\cdot)$ are the premium proportions of follower and leader which acts as our control variables, respectively. $DB$ is the pension scheme benefit outgo which is assumed to be a constant for sake of simplicity.

Suppose that the pension fund is invested in a risk-free asset (bond) and two risky assets (stocks). The price $S_0(t)$ of the bond at time $t$ is given by
\begin{equation}\label{bond}
\left\{
\begin{aligned}
dS_0(t)&=r(t)S_0(t)dt,\ t\geq0,\\
 S_0(0)&=1,
\end{aligned}
\right.
\end{equation}
where $r(t)>0$ is the instantaneous rate of return at time $t$.

The prices $S_1(t)$ and $S_2(t)$ of the two stocks at time $t$ are given by
\begin{equation}\label{stock1}
\left\{
\begin{aligned}
dS_1(t)&=S_1(t)[\mu_1(t)dt+\sigma(t)dW(t)],\ t\geq0,\\
 S_1(0)&=S_0^1,
\end{aligned}
\right.
\end{equation}
\begin{equation}\label{stock2}
\left\{
\begin{aligned}
dS_2(t)&=S_2(t)[\mu_2(t)dt+\widetilde{\sigma}(t)d\widetilde{W}(t)],\ t\geq0,\\
 S_2(0)&=S_0^2,
\end{aligned}
\right.
\end{equation}
respectively, where $W(\cdot)$ and $\widetilde{W}(\cdot)$ are two independent one-dimension Brownian motion. Here $\mu_i(t)>r(t), i=1,2$ are the instantaneous rates of expected return and $\sigma(t),\widetilde{\sigma}(t)>0$ are the instantaneous rates of volatility, at time $t$. We assume that $\mu_1(\cdot),\mu_2(\cdot),r(\cdot),\sigma(\cdot)$ and $\widetilde{\sigma}(\cdot)$ are deterministic bounded functions, and $\sigma^{-1}(\cdot)$ and $\widetilde{\sigma}^{-1}(\cdot)$ are also bounded.

In the real financial market, it is reasonable for the investors to make decisions based on the historical price of the risky asset $S_1(\cdot)$ and $S_2(\cdot)$. Therefore, the observable filtration at time $t$ can be set as $\mathcal{F}_t=\sigma\{S_1(s),S_2(s)|0\leq s\leq t\}$ and it is clear that $\mathcal{F}_t=\sigma\{W(s),\widetilde{W}(s)|0\leq s\leq t\}$. However, in our Stackelberg game background, there exists two different asymmetric information  for two players, to some degree, because of some practical phenomenon such as insider trading or the information asymmetry. So we assume that the one who plays a leader's role knows the full information from the financial market including the price of the risky assets $S_1(\cdot)$ and $S_2(\cdot)$, which is called $\mathcal{F}_t=\sigma\{W(s),\widetilde{W}(s)|0\leq s\leq t\}$, but the other one who plays a follower's role only knows the partial information about the price $S_1(\cdot)$ coming from $\mathcal{G}_t^1=\sigma\{W(s)|0\leq s\leq t\}$. Obviously, $\mathcal{G}_t^1\subset\mathcal{F}_t$.

Suppose that the proportion $\pi_1(t)$ and $\pi_2(t)$ of the pension fund is to be allocated in the two stock, respectively, while $1-\pi_1(t)-\pi_2(t)$ is to be allocated in the bond, at time $t$. Thus the instantaneous return becomes
\begin{equation}\label{Delta}
\begin{aligned}
d\Delta(t)&=\big[r(t)+(\mu_1(t)-r(t))\pi_1(t)+(\mu_2(t)-r(t))\pi_2(t)\big]dt\\
          &\quad+\sigma(t)\pi_1(t)dW(t)+\widetilde{\sigma}(t)\pi_2(t)d\widetilde{W}(t).
\end{aligned}
\end{equation}
Therefore, the pension fund dynamics can be written as the following form:
\begin{equation}\label{F}
\begin{aligned}
dF(t)&=\big[r(t)F(t)+(\mu_1(t)-r(t))\pi_1(t)F(t)+(\mu_2(t)-r(t))\pi_2(t)F(t)\\
     &\quad+v_1(t)+v_2(t)-DB\big]dt+\sigma(t)\pi_1(t)F(t)dW(t)+\widetilde{\sigma}(t)\pi_2(t)F(t)d\widetilde{W}(t).
\end{aligned}
\end{equation}
On the one hand, if the pension fund manager wants to achieve the wealth level $\xi$ at the terminal time $T$ to fulfill his/her obligations, then the dynamics of pension fund plan member's account is
\begin{equation}\label{dynamic}
\left\{
\begin{aligned}
dF(t)&=\big[r(t)F(t)+(\mu_1(t)-r(t))\pi_1(t)F(t)+(\mu_2(t)-r(t))\pi_2(t)F(t)+v_1(t)\\
     &\quad+v_2(t)-DB\big]dt+\sigma(t)\pi_1(t)F(t)dW(t)+\widetilde{\sigma}(t)\pi_2(t)F(t)d\widetilde{W}(t),\ t\in[0,T],\\
 F(T)&=\xi.
\end{aligned}
\right.
\end{equation}
On the other hand, if we set $\sigma(\cdot)\pi_1(\cdot)F(\cdot)=Z(\cdot)$ and $\widetilde{\sigma}(\cdot)\pi_2(\cdot)F(\cdot)=\widetilde{Z}(\cdot)$, then the above equation is equivalent to the BSDE
\begin{equation}\label{BSDE}
\left\{
\begin{aligned}
-dF(t)&=-\bigg[r(t)F(t)+\frac{\mu_1(t)-r(t)}{\sigma(t)}Z(t)+\frac{\mu_2(t)-r(t)}{\widetilde{\sigma}(t)}\widetilde{Z}(t)+v_1(t)+v_2(t)-DB\bigg]dt\\
      &\quad-Z(t)dW(t)-\widetilde{Z}(t)d\widetilde{W}(t),\ t\in[0,T],\\
  F(T)&=\xi.
\end{aligned}
\right.
\end{equation}
where the control processes $v_1(\cdot)$ and $v_2(\cdot)$ are adapted to the information filtration $\mathcal{G}_t^1$ and $\mathcal{F}_t$, respectively.

Let $\mathcal{U}_1[0,T]=\big\{v_1(\cdot)\in L_{\mathcal{G}^1}^2(0,T;\mathbb{R})|v_1(t)\in\mathbb{R},t\in[0,T]\big\}$ and $\mathcal{U}_2[0,T]=\big\{v_2(\cdot)\in L_{\mathcal{F}}^2(0,T;\mathbb{R})|\\
v_2(t)\in\mathbb{R},t\in[0,T]\big\}$ denote the admissible control sets for the follower and leader, respectively. For any $(v_1(\cdot),v_2(\cdot))\in\mathcal{U}_1\times\mathcal{U}_2$, the BSDE \eqref{BSDE} admits a unique solution triple $(F(\cdot),Z(\cdot),\widetilde{Z}(\cdot))$ in $L_{\mathcal{F}}^2(0,T;\mathbb{R})\times L_{\mathcal{F}}^2(0,T;\mathbb{R})\times L_{\mathcal{F}}^2(0,T;\mathbb{R})$.

Let us introduce the cost functionals
\begin{equation}\label{cost func3}
J_i(v_1(\cdot),v_2(\cdot);\xi)=\mathbb{E}\bigg[\int_0^T\frac{1}{2}e^{-\beta t}(v_i(t)-NC)^2dt+F^2(0)\bigg],\ i=1,2,
\end{equation}
where $\beta$ is a discount factor and $NC$ is a preset target, say, the normal cost.
The aim of the members is to minimize the cost functional $J_i(v_1(\cdot),v_2(\cdot);\xi)$ over $\mathcal{U}_i$, $i=1,2$. Recall that the first term of $J_i(u_1(\cdot),u_2(\cdot);\xi)$  is the running cost due to the deviation of the contribution from the preset target level. This term is introduced here to measure the stability of the DB pension scheme. The second term $F(0)$ is just the initial reserve to operate the scheme.

Let us now explain the leader-follower feature of the game. At time $t$, first, the big company (leader) announces his/her contribution (premium proportion) $v_2(t)$. Second, with the help of the part of informations the retail investor (follower) knows, he/she would like to set his/her contribution (premium proportion) $\bar{v}_1(t)$ as his optimal response to the company's announced decisions so that $J_1(\bar{v}_1(\cdot),v_2(\cdot);\xi)$ is the minimum of $J_1(v_1(\cdot),v_2(\cdot);\xi)$ over $v_1(\cdot)\in\mathcal{U}_1$. Knowing the follower would take such an optimal control $\bar{v}_1(\cdot)$ (supposing it exists, which depends on the choice $v_2(\cdot)$ of the leader and the initial state $\xi$, in general), and having the advantages over the follower in case of possessing more information, the big company (leader) would like to choose some $\bar{v}_2(\cdot)\in\mathcal{U}_2$ to minimize $J_2(\bar{v}_1(\cdot),v_2(\cdot);\xi)$ over $v_2(\cdot)\in\mathcal{U}_2$.

We aim to find the Stackelberg equilibrium point $(\bar{v}_1(\cdot),\bar{v}_2(\cdot))\in\mathcal{U}_1\times\mathcal{U}_2$, which is the optimal control pairs of the Stackelberg game of BSDE with partial information.

There is much literature to study the pension fund management problem by stochastic control approach, such as Huang et al. \cite{HWX09}, Di Giacinto et al. \cite{DFG11}, etc. However, our problems are essentially different in that we study the pension fund problem in the framework of Stackelberg game of BSDE with partial information. For more details about financial applications for partial information differential games, please refer to Wang and Yu \cite{WY12}, Shi and Wang \cite{SW16}, Huang et al. \cite{HWZ19}, Xiong et al. \cite{XZZ19}, etc.

It is obvious that this problems can be regarded as a special case of that in Section 4. So we can use the results to solve it. For the simplicity of the calculations in this example, we set $DB=NC=0$. Comparing to \eqref{LQbsde}, \eqref{LQ cost functional} and \eqref{LQ cost functional2}, we know in this section $A(t)=-r(t)$, $B_1(t)=B_2(t)=-1$, $C_1(t)=-\frac{\mu_1(t)-r(t)}{\sigma(t)}$, $C_2(t)=-\frac{\mu_2(t)-r(t)}{\widetilde{\sigma}(t)}$, $Q_1(t)=Q_2(t)=0$, $R_1(t)=R_2(t)=e^{-\beta t}$, $S_1(t)=S_2(t)=0$, $N_1(t)=N_2(t)=0$ and $G_1(t)=G_2(t)=2$ for any $t\in[0,T]$.

Firstly, we solve the follower's problem. For any given $\xi\in L_{\mathcal{F}_T}^2(\Omega;\mathbb{R})$ and $v_2(\cdot)\in\mathcal{U}_2[0,T]$, using Theorem 4.1, we can get
\begin{equation}\label{5v1}
\bar{v}_1(t)=e^{\beta t}\big[P_2(t)\hat{y}^{\bar{v}_1,\hat{v}_2}+\hat{\varphi}(t)\big],
\end{equation}
where $(\hat{y}^{\bar{v}_1,\hat{v}_2}(\cdot),\hat{z}^{\bar{v}_1,\hat{v}_2}(\cdot),\hat{\widetilde{z}}^{\bar{v}_1,\hat{v}_2}(\cdot))$ satisfy the following FBSDE
\begin{equation}
\left\{
\begin{aligned}
          -d\hat{y}^{\bar{v}_1,\hat{v}_2}(t)&=\Big[-(e^{\beta t}P_2(t)+r(t))\hat{y}^{\bar{v}_1,\hat{v}_2}(t)-e^{\beta t}\hat{\varphi}(t)-\hat{v}_2(t)-\frac{\mu_1(t)-r(t)}{\sigma(t)}\hat{z}^{\bar{v}_1,\hat{v}_2}(t)\\
                                            &\quad-\frac{\mu_2(t)-r(t)}{\widetilde{\sigma}(t)}\hat{\widetilde{z}}^{\bar{v}_1,\hat{v}_2}(t)\Big]dt-\hat{z}^{\bar{v}_1,\hat{v}_2}(t)dW(t),\\
                                 d\hat{x}(t)&=-r(t)\hat{x}(t)dt-\frac{\mu_1(t)-r(t)}{\sigma(t)}\hat{x}(t)dW(t),\ t\in[0,T],\\
\hat{\widetilde{z}}^{\bar{v}_1,\hat{v}_2}(t)&=\frac{\mu_2(t)-r(t)}{\widetilde{\sigma}(t)}P_1(t)\hat{x}(t),\ t\in[0,T],\ \hat{y}^{\bar{v}_1,\hat{v}_2}(T)=\hat{\xi},\ \hat{x}(0)=2\hat{y}^{\bar{v}_1,\hat{v}_2}(0),
\end{aligned}
\right.
\end{equation}
$P_1(\cdot)$ and $P_2(\cdot)$ satisfy the following two Riccati equations
\begin{equation}
\left\{
\begin{aligned}
&\dot{P}_1(t)+\Big[\Big(\frac{\mu_1(t)-r(t)}{\sigma(t)}\Big)^2+\Big(\frac{\mu_2(t)-r(t)}{\widetilde{\sigma}(t)}\Big)^2-2r(t)\Big]P_1(t)+e^{\beta t}=0,\ t\in[0,T],\\
&P_1(T)=0,
\end{aligned}
\right.
\end{equation}
\begin{equation}
\left\{
\begin{aligned}
&\dot{P}_2(t)+\Big[\Big(\frac{\mu_1(t)-r(t)}{\sigma(t)}\Big)^2+\Big(\frac{\mu_2(t)-r(t)}{\widetilde{\sigma}(t)}\Big)^2\Big]P_2^2(t)P_1(t)+e^{\beta t}P_2^2(t)\\
&+2r(t)P_2(t)=0,\ t\in[0,T],\ P_2(0)=2,
\end{aligned}
\right.
\end{equation}
respectively, and $(\hat{\varphi}(\cdot),\hat{\phi}(\cdot),\hat{\eta}(\cdot))$ satisfy the following FBSDE
\begin{equation}
\left\{
\begin{aligned}
d\hat{\varphi}(t)&=\bigg\{\bigg[-\Big[\Big(\frac{\mu_1(t)-r(t)}{\sigma(t)}\Big)^2+\Big(\frac{\mu_2(t)-r(t)}{\widetilde{\sigma}(t)}\Big)^2\Big]P_1(t)P_2(t)-e^{\beta t}P_2(t)\\
                 &\qquad-r(t)\bigg]\hat{\varphi}(t)+\frac{\mu_1(t)-r(t)}{\sigma(t)}P_2(t)\hat{\eta}(t)-P_2(t)\hat{v}_2(t)\bigg\}dt\\
                 &\quad+\Big\{\frac{\mu_1(t)-r(t)}{\sigma(t)}P_2(t)\hat{\phi}(t)-\frac{\mu_1(t)-r(t)}{\sigma(t)}\hat{\varphi}(t)+P_2(t)\hat{\eta}(t)\Big\}dW(t),\\
  -d\hat{\phi}(t)&=\Big\{-r(t)\hat{\phi}(t)+\hat{v}_2(t)-\frac{\mu_1(t)-r(t)}{\sigma(t)}\hat{\eta}(t)\Big\}dt-\hat{\eta}(t)dW(t),\ t\in[0,T],\\
    \hat{\phi}(T)&=-\hat{\xi},\ \hat{\varphi}(0)=0.
\end{aligned}
\right.
\end{equation}

Next, we solve the leader's problem, noting \eqref{notation} and putting
\begin{equation*}
\left\{
\begin{aligned}
&\mathcal{A}_1=
\begin{pmatrix}
-r(t)-e^{\beta t}P_2(t)-\Big(\frac{\mu_2(t)-r(t)}{\widetilde{\sigma}(t)}\Big)^2P_1(t)P_2(t)&0\\
0&-r(t)
\end{pmatrix},\ \
\mathcal{A}_2=
\begin{pmatrix}
0&0\\
0&-e^{\beta t}P_2(t)
\end{pmatrix},\\
&\widetilde{\mathcal{A}}_1=
\begin{pmatrix}
-\frac{\mu_1(t)-r(t)}{\sigma(t)}&0\\
0&-\frac{\mu_1(t)-r(t)}{\sigma(t)}
\end{pmatrix},\ \
\widetilde{\mathcal{A}}_2=
\begin{pmatrix}
0&0\\
0&-\frac{\mu_2(t)-r(t)}{\widetilde{\sigma}(t)}
\end{pmatrix},\ \
\mathcal{B}_1=
\begin{pmatrix}
0&0\\
0&0
\end{pmatrix},\\
&\widetilde{\mathcal{B}}_2=
\begin{pmatrix}
0\\
-1
\end{pmatrix},\ \
\mathcal{B}_2=
\begin{pmatrix}
0&\Big(\frac{\mu_1(t)-r(t)}{\sigma(t)}\Big)^2P_2^2(t)P_1(t)\\
\Big(\frac{\mu_1(t)-r(t)}{\sigma(t)}\Big)^2P_2^2(t)P_1(t)&0
\end{pmatrix},\\
&\mathcal{C}=
\begin{pmatrix}
0&-\frac{\mu_1(t)-r(t)}{\sigma(t)}P_2(t)\\
-\frac{\mu_1(t)-r(t)}{\sigma(t)}P_2(t)&0
\end{pmatrix},\ \
\widetilde{\mathcal{C}}_1=
\begin{pmatrix}
0&0\\
0&0
\end{pmatrix},\ \
\widetilde{\mathcal{C}}_2=
\begin{pmatrix}
0&-P_2(t)\\
-P_2(t)&0
\end{pmatrix},
\end{aligned}
\right.
\end{equation*}
\begin{equation*}
\left\{
\begin{aligned}
&\widetilde{\mathcal{C}}_3=
\begin{pmatrix}
0&0\\
0&0
\end{pmatrix},\ \
\mathcal{D}=
\begin{pmatrix}
-P_2(t)\\
0
\end{pmatrix},\ \
\mathcal{F}_1=
\begin{pmatrix}
0&-e^{\beta t}\\
-e^{\beta t}&0
\end{pmatrix},\ \
\bar{\xi}=
\begin{pmatrix}
0\\
\xi
\end{pmatrix},\ \
\bar{G}=
\begin{pmatrix}
0&0\\
2&2
\end{pmatrix}.
\end{aligned}
\right.
\end{equation*}
By Theorem 4.2, we can get
\begin{equation}\label{5v2}
\bar{v}_2(t)=-e^{\beta t}\widetilde{\mathcal{B}}_2^\top\Pi_3(t)Y(t)-e^{\beta t}\big[\widetilde{\mathcal{B}}_2^\top\Pi_4(t)+\mathcal{D}^\top\big]\hat{Y}(t)-e^{\beta t}\widetilde{\mathcal{B}}_2^\top\widetilde{\varphi}(t),
\end{equation}
where $(Y(\cdot),Z(\cdot),\widetilde{Z}(\cdot))$ satisfy the following 2-dimensional BSDE
{\small\begin{equation}
\left\{
\begin{aligned}
-dY(t)&=\bigg\{\bigg[\begin{pmatrix}
                      0&-e^{\beta t}\\
                      -e^{\beta t}&-e^{\beta t}
                     \end{pmatrix}\Pi_3(t)+\begin{pmatrix}
                                         -r(t)-e^{\beta t}P_2(t)-\Big(\frac{\mu_2(t)-r(t)}{\widetilde{\sigma}(t)}\Big)^2P_1(t)P_2(t)&0\\
                                         0&-r(t)
                                        \end{pmatrix}\bigg]Y(t)\\
      &\qquad+\bigg[\begin{pmatrix}
                     0&-e^{\beta t}\\
                     -e^{\beta t}&-e^{\beta t}
                    \end{pmatrix}\Pi_4(t)+\begin{pmatrix}
                                           0&0\\
                                           -e^{\beta t}P_2(t)&-e^{\beta t}P_2(t)
                                          \end{pmatrix}\bigg]\hat{Y}(t)+\begin{pmatrix}
                                                                         0&-e^{\beta t}\\
                                                                         -e^{\beta t}&-e^{\beta t}
                                                                        \end{pmatrix}\widetilde{\varphi}(t)\\
      &\qquad+\begin{pmatrix}
               -\frac{\mu_1(t)-r(t)}{\sigma(t)}&0\\
               0&-\frac{\mu_1(t)-r(t)}{\sigma(t)}
              \end{pmatrix}Z(t)+\begin{pmatrix}
                                 0&0\\
                                 0&-\frac{\mu_2(t)-r(t)}{\widetilde{\sigma}(t)}
                                \end{pmatrix}\widetilde{Z}(t)\bigg\}dt\\
      &\quad-Z(t)dW(t)-\widetilde{Z}(t)d\widetilde{W}(t),\ t\in[0,T],\ Y(T)=\bar{\xi},
\end{aligned}
\right.
\end{equation}}
$\widetilde{\varphi}(t)$ satisfies the 2-dimensional SDE \eqref{widevarphi}, and $\Pi_3(\cdot)$ and $\Pi_4(\cdot)$ satisfy \eqref{Pi3} and \eqref{Pi4}, respectively.

By a dual technique, we have
\begin{equation}
\begin{aligned}
Y(t)&=\mathbb{E}\Bigg[\Gamma_t(T)\bar{\xi}+\int_t^T\Bigg\{\bigg[\begin{pmatrix}
                                                                 0&-e^{\beta s}\\
                                                                 -e^{\beta s}&-e^{\beta s}
                                                                \end{pmatrix}\Pi_4(t)+\begin{pmatrix}
                                                                                    0&0\\
                                                                                    -e^{\beta s}P_2(s)&-e^{\beta s}P_2(s)
                                                                                   \end{pmatrix}\bigg]\hat{Y}(s)\\
&\qquad+\begin{pmatrix}
         0&-e^{\beta s}\\
         -e^{\beta s}&-e^{\beta s}
        \end{pmatrix}\widetilde{\varphi}(s)\Bigg\}\Gamma_t(s)ds\Bigg|\mathcal{F}_t\Bigg],
\end{aligned}
\end{equation}
where for $t\in[0,T]$, $\Gamma_t(\cdot)$ is the unique solution to
{\small\begin{equation}
\left\{
\begin{aligned}
d\Gamma_t(s)&=\Bigg[\begin{pmatrix}
                     0&-e^{\beta s}\\
                     -e^{\beta s}&-e^{\beta s}
                    \end{pmatrix}\Pi_3+\begin{pmatrix}
                                        -r(s)-e^{\beta s}P_2(s)-\Big(\frac{\mu_2(s)-r(s)}{\widetilde{\sigma}(s)}\Big)^2P_1(s)P_2(s)&0\\
                                        0&-r(s)
                                       \end{pmatrix}\Bigg]\Gamma_t(s)ds\\
             &\quad+\begin{pmatrix}
                     -\frac{\mu_1(s)-r(s)}{\sigma(s)}&0\\
                     0&-\frac{\mu_1(s)-r(s)}{\sigma(s)}
                    \end{pmatrix}\Gamma_t(s)dW(s)+\begin{pmatrix}
                                                   0&0\\
                                                   0&-\frac{\mu_2(s)-r(s)}{\widetilde{\sigma}(s)}
                                                  \end{pmatrix}\Gamma_t(s)d\widetilde{W}(s),\ s\in[t,T],\\
\Gamma_t(t)&=1.
\end{aligned}
\right.
\end{equation}}
Thus, $(\bar{v}_1(\cdot),\bar{v}_2(\cdot))$ determined by \eqref{5v1} and \eqref{5v2} is a Stackelberg equilibrium point of our Stackelberg game of BSDEs with partial information.

Finally, the optimal initial wealth reserve $y^{\bar{v}_1,\bar{v}_2}(0)$ is the second component of the following 2-dimensional vector
\begin{equation}
\begin{aligned}
Y(0)&=\mathbb{E}\Bigg[\Gamma_0(T)\bar{\xi}+\int_0^T\Bigg\{\bigg[\begin{pmatrix}
                                                                 0&-e^{\beta t}\\
                                                                 -e^{\beta t}&-e^{\beta t}
                                                                \end{pmatrix}\Pi_4(t)+\begin{pmatrix}
                                                                                       0&0\\
                                                                                       -e^{\beta t}P_2(t)&-e^{\beta t}P_2(t)
                                                                                      \end{pmatrix}\bigg]\hat{Y}(t)\\
    &\qquad+\begin{pmatrix}
             0&-e^{\beta t}\\
             -e^{\beta t}&-e^{\beta t}
            \end{pmatrix}\widetilde{\varphi}(t)\Bigg\}\Gamma_0(t)dt\Bigg].
\end{aligned}
\end{equation}

\section{Concluding Remarks}

In this paper, we have discussed the Stackelberg game of BSDEs with partial information. The general problem is studied first and then the LQ special case is researched in some state estimate representations for the Stackelberg equilibrium point, for the follower and the leader, respectively. Theoretical results are applied to the pension fund management problem.

Possible extensions to the Stackelberg game with noisy observations are desired to be researched, and the solvability of the related Riccati equations are very challenging and difficult research topics. We will consider these problems in our future research.

\end{document}